\documentclass[12pt]{amsart}
\usepackage{ifthen}
\usepackage[utf8]{inputenc}
\usepackage{setspace}
\usepackage[english]{babel}
\usepackage{enumerate}
\usepackage{thmtools}
\usepackage[shortlabels]{enumitem}
\usepackage[T1]{fontenc}
\usepackage{tikz}
\usetikzlibrary{shapes.geometric,intersections,decorations.markings,snakes}
\usetikzlibrary{calc,intersections,through,backgrounds}
\usetikzlibrary{patterns}

\newlength\friezelen 
\usepackage{array} 
\newcolumntype{Q}{>{\centering}p{\friezelen}<{}}

\usepackage{mathtools}
\usepackage{verbatim}
\usepackage{comment}
\usetikzlibrary{arrows}
\tikzset {->-/.style={decoration={markings, mark=at position .5 with {\arrow{latex}}}, postaction={decorate}}}
\tikzset {-->-/.style={decoration={markings, mark=at position .5 with {\arrow[scale=2]{latex}}}, postaction={decorate}}}
\newcommand{\midarrow}{\tikz \draw[-triangle 90] (0,0) -- +(.1,0);}

\newcommand{\midrevarrow}{\tikz \draw[-triangle 90] (0,0) -- +(-.1,0);}
\usepackage{comment}
\usepackage{graphicx}
\usepackage{color}
\usepackage{etoolbox}
\usepackage[margin=1in]{geometry}
\usepackage{amsmath,amsthm,amssymb,graphicx,tikz,tikz-cd}
\usepackage{hyperref}
\hypersetup{
    colorlinks = true,
    citecolor = orange,%
}
\usepackage[noabbrev]{cleveref}
\usepackage{subcaption}

\makeatletter
\patchcmd{\@settitle}{\uppercasenonmath\@title}{}{}{}
\patchcmd{\@setauthors}{\MakeUppercase}{}{}{}
\patchcmd{\section}{\scshape}{}{}{}
\makeatother
\makeatletter
\patchcmd{\@maketitle}
  {\ifx\@empty\@dedicatory}
  {\ifx\@empty\@date \else {\vskip2ex 
  \centering\footnotesize\@date\par\vskip1ex}\fi
   \ifx\@empty\@dedicatory}
  {}{}
\patchcmd{\@adminfootnotes}
  {\ifx\@empty\@date\else \@footnotetext{\@setdate}\fi}
  {}{}{}
\makeatother

\theoremstyle{plain}
\newtheorem{theorem}{Theorem}[section]

\newtheorem{lemma}[theorem]{Lemma}
\newtheorem{prop}[theorem]{Proposition}
\newtheorem{conj}[theorem]{Conjecture}
\newtheorem{question}[theorem]{Question}
\newtheorem{corollary}[theorem]{Corollary}

\theoremstyle{definition}
\newtheorem{remark}[theorem]{Remark}
\newtheorem{example}[theorem]{Example}
\newtheorem{definition}[theorem]{Definition}

\DeclareMathOperator{\Hom}{Hom}
\DeclareMathOperator{\psl}{PSL}
\DeclareMathOperator{\rr}{\mathbb{R}}

\DeclareMathOperator{\wt}{wt}
\DeclareMathOperator{\twt}{twt}
\DeclareMathOperator{\crs}{cross}

\title{Double Dimer Covers on Snake Graphs from Super Cluster Expansions}
\author{Gregg Musiker}

\author{Nicholas Ovenhouse}
\author{Sylvester W. Zhang}
\thanks{School of Mathematics, University of Minnesota, Minneapolis, MN 55455, USA}
\thanks{\emph{Email}:
\href{mailto:musiker@umn.edu}{musiker@umn.edu},
\href{mailto:ovenh001@umn.edu}{ovenh001@umn.edu},
\href{mailto:swzhang@umn.edu}{swzhang@umn.edu}}
\date{}

\begin{document}
\date{October 13, 2021}

\begin{abstract}
    In a recent paper, the authors gave combinatorial formulas for the Laurent expansions
    of super $\lambda$-lengths in a marked disk, generalizing Schiffler's $T$-path formula.
    In the present paper, we give an alternate combinatorial expression for these
    super $\lambda$-lengths in terms of double dimer covers 
    on snake graphs.  This generalizes the dimer formulas of Musiker, Schiffler, and Williams.
\end{abstract}

\maketitle
\setcounter{tocdepth}{1}
\tableofcontents
\setlength{\parindent}{0em}
\setlength{\parskip}{0.618em}

\tableofcontents

\section*{Introduction}

A \emph{marked surface} is a surface $S$ with boundary, with a collection of marked points $M \subset \partial S$ (we require at least one marked point per
boundary component). Although it is common in the literature to also consider interior marked points (i.e. ``punctures''), we will only consider
unpunctured surfaces. It is well-known in the theory of cluster algebras that the \emph{decorated Teichm\"{u}ller space} of $S$ possesses 
a cluster structure \cite{gsv_05} \cite{fg_06} \cite{fst_08}. Given a fixed triangulation, with arcs terminating at marked boundary points,
the collection of \emph{$\lambda$-lengths} gives a set of cluster coordinates.
In particular, this means that the $\lambda$-length of any geodesic arc between marked points is expressible as a Laurent polynomial in these coordinates.
The ``flip'' of a triangulation corresponds to cluster mutation, realized by the hyperbolic version of Ptolemy's relation.

There have been many combinatorial ways to enumerate the terms occurring in these Laurent expansions. 
One of the earliest was the notion of \emph{$T$-paths} \cite{schiffler_08} \cite{st_09}.
A few years later, a dimer (perfect matching) interpretation was given \cite{ms10} \cite{msw_11}. 
There have been other combinatorial models (e.g. \cite{yurikusa_19}, \cite{llm_15}, \cite{claussen20}), but these two will be our main focus
in this article.

In recent years, there have been efforts to find supersymmetric generalizations of both cluster algebras and decorated Teichm\"{u}ller spaces. On the one hand,
there have been a couple different attempts to define super cluster algebras formally (see \cite{lmrs_21} and \cite{o_15, OS19}). There is also a supersymmetric version
of frieze patterns (which are related to cluster combinatorics) \cite{M-GOT15}. On the other hand, in a more geometric setting, Penner and Zeitlin have recently introduced
\emph{decorated super-Teichm\"{u}ller spaces} \cite{pz_19}.

In a recent paper, the current authors took a first step in unifying the algebraic, combinatorial, and geometric viewpoints mentioned above \cite{moz21}. 
In particular, we found an analogue of the $T$-path Laurent formula for decorated super-Teichm\"{u}ller spaces.
We also showed that the collection of super $\lambda$-lengths and $\mu$-invariants fit into a super frieze
pattern, as in \cite{M-GOT15}. Lastly, we explored the connection between our approach and the definitions of super cluster algebras given in \cite{OS19}.

The main purpose of the present paper is to give a super analogue of the dimer formulas from \cite{ms10}. Surprisingly, the terms in the super Laurent polynomials
are more naturally expressed in terms of \emph{double dimer covers} (rather than ordinary dimer covers). We also give double dimer formulas for some (but not all) of
the odd variables corresponding to triangles. Later sections of the paper give other combinatorial interpretations which can be derived from the double dimer model.
The structure of the paper is as follows. 

\vspace{-1.5em}
\begin {itemize}
    \item In Section 1, we review background on decorated super-Teichm\"{u}ller theory of \cite{pz_19}.

    \item In Section 2, we recall some definitions and conventions from our previous paper \cite{moz21} regarding
          triangulations of polygons, and how they are labelled.

    \item Section 3 defines \emph{snake graphs}, and gives examples. These graphs will be important in our main theorem, which says that any super $\lambda$-length
          can be written as a Laurent polynomial whose terms are indexed by double dimer covers on a certain snake graph.

    \item Section 4 introduces general terms and notations related to double dimer covers.

    \item Section 5 gives recurrence formulas for double dimer covers on snake graphs. The lemmas appearing in this section are the main ingredients of the inductive
          proof of the main theorem later in the paper.

    \item Section 6 presents the main theorem (Theorem \ref{thm:main}), which says that the Laurent expansions of super $\lambda$-lengths are sums over weights of double dimer covers
          on a certain snake graph.

    \item In Section 7, we define a modified version of the super $T$-paths from \cite{moz21}, and use these to give an explicit weight-preserving bijection between
          super $T$-paths and double dimer covers.

    \item In Section 8, we discuss \emph{dual snake graphs}. Looking at these dual snake graphs, we get yet another combinatorial description of the Laurent terms,
          this time in terms of lattice paths on the snake graph.

    \item In Section 9, we discuss how the terms in these Laurent polynomials form a distributive lattice, and thus can be identified with order ideals of
          a certain poset.

    \item In Section 10, we give some examples and illustrations.  

    \item Lastly, in Section 11, we consider super $\lambda$-lengths on other surfaces, including the once-punctured torus 
          and the annulus with one marked point on each boundary.  We then discuss super analogues of Fibonacci numbers and end with open problems.
\end {itemize}

\section {Decorated Super-Teichm\"{u}ller Theory} 

In this section, we survey the theory of decorated super-Teichm\"uller spaces recently developed by Penner and Zeitlin \cite{pz_19}.
We start with a brief overview of Penner's classical theory of decorated Teichm\"uller spaces. The readers are referred to \cite{penner_12} for a detailed reference.

Consider a surface $S$ with marked points on its boundary\footnote{In general, a surface can also contain marked points on its interior, known as \emph{punctures}, 
but we will not be concerned with this case, except for a discussion of the once-punctured torus in Section 11.},
such that each boundary component contains at least one marked point. 

The \emph{Teichm\"{u}ller space} of $S$, denoted $\mathcal{T}(S)$, is the space of (equivalence classes of) hyperbolic metrics on $S$ with constant negative
curvature, with cusps at the marked points. 
More formally, the Teichm\"{u}ller space of $S$ is defined to be the quotient space
\[ \mathcal{T}(S)=\Hom(\pi_1(S),\psl(2,\rr))/\psl(2,\rr). \]

The \emph{decorated Teichm\"{u}ller space} of $S$, written $\widetilde{\mathcal{T}}(S)$, is a trivial vector bundle over $\mathcal{T}(S)$,
with fiber $\Bbb{R}_{> 0}^n$. The fibers represent a choice of a positive real number associated to each marked point.
At each marked point, we draw a horocycle whose size (or \emph{height}) is determined by the corresponding positive number.
Truncating the geodesics using these horocycles, it now makes sense to talk about their lengths. If $\ell$ is the truncated
length of one of these geodesic segments, then the $\lambda$-length (or \emph{Penner coordinate}) associated to that geodesic arc
is defined to be
\[ \lambda := \exp(\ell/2) \]

Fixing a triangulation of the marked surface, the collection of $\lambda$-lengths corresponding to the arcs in the triangulation (including segments
of the boundary) form a system of coordinates for $\widetilde{\mathcal{T}}(S)$. Choosing a different triangulation results in a different
system of coordinates, but they are related by simple transformations which are a hyperbolic analogue of Ptolemy's theorem from classical Euclidean geometry.
If two triangulations differ by the flip of a single arc as in \Cref{fig:ptolemy}, 
then the $\lambda$-lengths are related by:
$ef = ac + bd$.

\begin{figure}[h!]
\centering

\begin{tikzpicture}[scale=0.7, baseline, thick]

    \draw (0,0)--(3,0)--(60:3)--cycle;
    \draw (0,0)--(3,0)--(-60:3)--cycle;

    \draw (0,0) -- (3,0);

    \draw node[above]      at (70:1.5){$a$};
    \draw node[above]      at (30:2.8){$b$};
    \draw node[below]      at (-30:2.8){$c$};
    \draw node[below=-0.1] at (-70:1.5){$d$};
    \draw node[above] at (1.5, 0){$e$};

    \draw node[left] at (0,0) {};
    \draw node[above] at (60:3) {};
    \draw node[right] at (3,0) {};
    \draw node[below] at (-60:3) {};

\end{tikzpicture}
\begin{tikzpicture}[baseline]
    \draw[->, thick](0,0)--(1,0);
    \node[above]  at (0.5,0) {};
\end{tikzpicture}
\begin{tikzpicture}[scale=0.7, baseline, thick,every node/.style={sloped,allow upside down}]
    \draw (0,0)--(60:3)--(-60:3)--cycle;
    \draw (3,0)--(60:3)--(-60:3)--cycle;

    \draw node[above]      at (70:1.5)  {$a$};
    \draw node[above]      at (30:2.8)  {$b$};
    \draw node[below]      at (-30:2.8) {$c$};
    \draw node[below=-0.1] at (-70:1.5) {$d$};
    \draw node[left]       at (1.6,0)   {$f$};

    \draw (1.5,-2) --  (1.5,2);

    \draw node[left] at (0,0) {};
    \draw node[above] at (60:3) {};
    \draw node[right] at (3,0) {};
    \draw node[below] at (-60:3) {};

\end{tikzpicture}
\caption{Ptolemy transformation}
\label{fig:ptolemy}
\end{figure}
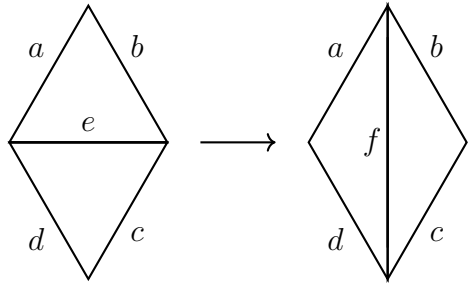

In a recent paper \cite{pz_19}, Penner and Zeitlin gave a supersymmetric analogue of the above mentioned decorated Teichm\"uller theory.
The \emph{decorated super-Teichm\"uller space} of $S$ is a superalgebra, generated by even elements corresponding to the $\lambda$-lengths 
and odd elements, called \emph{$\mu$-invariants}, associated to the triangles. 

To understand the multiplication of the anti-commutative $\mu$-invariants, we need additional combinatorial data 
called a \emph{spin structure}. Connected components of the super-Teichm\"uller space $ST(S)$ are indexed by the set of 
spin structures on $S$, which are formulated as the set of isomorphism classes of Kasteleyn orientations of a certain graph 
embedded on a deformation retract of $S$ \cite{cr_07,cr_08}. These Kasteleyn orientations correspond to equivalence classes of orientations of a fatgraph spine of
$S$. For our purpose, we use the dual of this formalism, where the \emph{spin structures} will be 
identified with the equivalence classes of orientations of edges of a triangulation, under the equivalence relation $\sim$ which we describe right now.

Given a triangulation $T$ with an orientation $\tau$ on its edges. For any triangle $t$,
consider the transformation which reverses the orientation of the three sides of $t$. We say that two orientations are 
equivalent $\tau \sim \tau'$ if they are related by a sequence of these moves. 

Now we are ready to state the super version of the Ptolemy relations.
\begin{align}
    ef      &= ac+bd+\sqrt{abcd} \, \sigma\theta \label{eqn:super_ptolemy_lambda} \\
    \sigma' &= \frac{\sigma\sqrt{bd}-\theta\sqrt{ac}}{\sqrt{ac+bd}}\label{eqn:super_ptolemy_mu_right} \\
    \theta' &= \frac{\theta\sqrt{bd}+\sigma \sqrt{ac}}{\sqrt{ac+bd}}\label{eqn:super_ptolemy_mu_left}
\end{align}

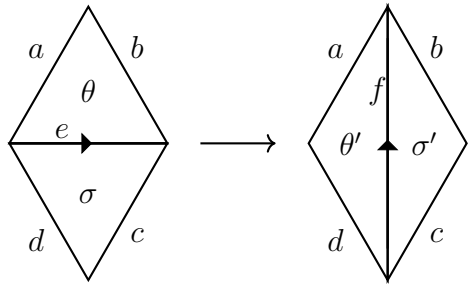
\begin{figure}[h!]
\centering
\begin{tikzpicture}[scale=0.7, baseline, thick]

    \draw (0,0)--(3,0)--(60:3)--cycle;
    \draw (0,0)--(3,0)--(-60:3)--cycle;

    \draw (0,0)--node {\midarrow} (3,0);

    \draw node[above]      at (70:1.5){$a$};
    \draw node[above]      at (30:2.8){$b$};
    \draw node[below]      at (-30:2.8){$c$};
    \draw node[below=-0.1] at (-70:1.5){$d$};
    \draw node[above] at (1,-0.12){$e$};

    \draw node[left] at (0,0) {};
    \draw node[above] at (60:3) {};
    \draw node[right] at (3,0) {};
    \draw node[below] at (-60:3) {};

    \draw node at (1.5,1){$\theta$};
    \draw node at (1.5,-1){$\sigma$};
\end{tikzpicture}
\begin{tikzpicture}[baseline]
    \draw[->, thick](0,0)--(1,0);
    \node[above]  at (0.5,0) {};
\end{tikzpicture}
\begin{tikzpicture}[scale=0.7, baseline, thick,every node/.style={sloped,allow upside down}]
    \draw (0,0)--(60:3)--(-60:3)--cycle;
    \draw (3,0)--(60:3)--(-60:3)--cycle;

    \draw node[above]      at (70:1.5)  {$a$};
    \draw node[above]      at (30:2.8)  {$b$};
    \draw node[below]      at (-30:2.8) {$c$};
    \draw node[below=-0.1] at (-70:1.5) {$d$};
    \draw node[left]       at (1.7,1)   {$f$};

    \draw (1.5,-2) --node {\midarrow} (1.5,2);

    \draw node[left] at (0,0) {};
    \draw node[above] at (60:3) {};
    \draw node[right] at (3,0) {};
    \draw node[below] at (-60:3) {};

    \draw node at (0.8,0){$\theta'$};
    \draw node at (2.2,0){$\sigma'$};
\end{tikzpicture}
\caption{super Ptolemy transformation}
\label{fig:super_ptolemy}
\end{figure}
Note that in \Cref{eqn:super_ptolemy_lambda}, the order of multiplying the two odd variables $\sigma$ and $\theta$ is determined by the orientation of the 
edge being flipped (see the arrow in \Cref{fig:super_ptolemy}).
In addition, it was shown in \cite[Proposition 6.3]{moz21} that \cref{eqn:super_ptolemy_mu_left,eqn:super_ptolemy_mu_right} can be simplified as follows.
\begin{align}
   \tag{2*} \sigma' & = \frac{\sigma\sqrt{bd}-\theta\sqrt{ac}}{\sqrt{ef}} \label{eqn:super_ptolemy_mu_right_1} \\
   \tag{3*} \theta' &= \frac{\theta\sqrt{bd}+\sigma \sqrt{ac}}{\sqrt{ef}} \label{eqn:super_ptolemy_mu_left_1}
\end{align}

In Figure \ref{fig:super_ptolemy}, the orientations of the edges around the boundary of the quadrilateral
are not indicated, but the super Ptolemy transformation 
does affect the boundary orientations, where three of the four edges are unchanged and only the orientation of the edge labelled $b$ is changed 
(see \Cref{fig:flip_spin}). 

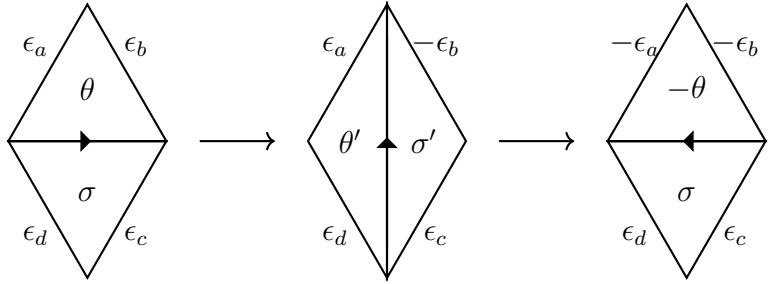
\begin{figure}[h!]

\centering

\begin{tikzpicture}[scale=0.7, baseline, thick]

\draw (0,0)--(3,0)--(60:3)--cycle;

\draw (0,0)--(3,0)--(-60:3)--cycle;

\draw (0,0)--node {\midarrow} (3,0);

\draw node[above] at (70:1.5){$\epsilon_a$};

\draw node[above] at (30:2.8){$\epsilon_b$};

\draw node[below] at (-30:2.8){$\epsilon_c$};

\draw node[below=-0.1] at (-70:1.5){$\epsilon_d$};


\draw node[left] at (0,0) {};

\draw node[above] at (60:3) {};

\draw node[right] at (3,0) {};

\draw node[below] at (-60:3) {};

\draw node at (1.5,1){$\theta$};

\draw node at (1.5,-1){$\sigma$};

\end{tikzpicture}
\begin{tikzpicture}[baseline]

\draw[->, thick](0,0)--(1,0);

\node[above]  at (0.5,0) {};

\end{tikzpicture}
\begin{tikzpicture}[scale=0.7, baseline, thick,every node/.style={sloped,allow upside down}]
\draw (0,0)--(60:3)--(-60:3)--cycle;
\draw (3,0)--(60:3)--(-60:3)--cycle;

\draw node[above] at (70:1.5){$\epsilon_a$};

\draw node[above] at (30:2.8){$-\epsilon_b$};

\draw node[below] at (-30:2.8){$\epsilon_c$};

\draw node[below=-0.1] at (-70:1.5){$\epsilon_d$};

\draw (1.5,-2) --node {\midarrow} (1.5,2);


\draw node[left] at (0,0) {};

\draw node[above] at (60:3) {};

\draw node[right] at (3,0) {};

\draw node[below] at (-60:3) {};

\draw node at (0.8,0){$\theta'$};

\draw node at (2.2,0){$\sigma'$};

\end{tikzpicture}
\begin{tikzpicture}[baseline]

\draw[->, thick](0,0)--(1,0);

\node[above]  at (0.5,0) {};

\end{tikzpicture}
\begin{tikzpicture}[scale=0.7, baseline, thick]

\draw (0,0)--(3,0)--(60:3)--cycle;

\draw (0,0)--(3,0)--(-60:3)--cycle;

\draw (3,0) --node {\reflectbox{\midarrow}} (0,0);

\draw node[above] at (70:1.5){$-\epsilon_a$};

\draw node[above] at (30:2.8){$-\epsilon_b$};

\draw node[below] at (-30:2.8){$\epsilon_c$};

\draw node[below=-0.1] at (-70:1.5){$\epsilon_d$};


\draw node[left] at (0,0) {};

\draw node[above] at (60:3) {};

\draw node[right] at (3,0) {};

\draw node[below] at (-60:3) {};

\draw node at (1.5,1){$-\theta$};

\draw node at (1.5,-1){$\sigma$};

\end{tikzpicture}
\caption{Flip effect on spin structures. Here $\epsilon_x$ denotes the orientation of an edge $x$.}
\label{fig:flip_spin}
\end{figure}

Note the super Ptolemy relation is \emph{not} an involution. As depicted in \Cref{fig:flip_spin}, flipping an edge
twice results in reversing the orientations around the top triangle, and additionally negating the $\mu$-invariant.
Hence $\mu$-invariants are well-defined only up to sign.

In our paper, we mostly consider marked disks with a triangulation such that every triangle has a boundary (see \Cref{sec:3} for an explanation), 
in which case the orientations of the boundary edges will not affect our calculation of super $\lambda$-lengths. 
Therefore we can often ignore the boundary orientations. The following lemma from \cite{moz21} tells us that for any spin structure, 
          the corresponding equivalence class of orientations will exhaust all possible interior orientations after ignoring the boundary.

\begin{lemma}[Proposition 4.1 of \cite{moz21}]\label{lm:spin}
    Fix a triangulation of a polygon in which every triangle has at least one boundary edge. Then there is a unique spin structure after ignoring the boundary edges. 
    In particular, this means that, from any representative orientation of a fixed spin structure, one can obtain all other orientations on the interior diagonals, 
    without changing the spin structure.
\end{lemma}

\begin{remark}[Remarks 4.2 and 4.3 of \cite{moz21}]
    Note that the equivalence relation guarantees that the result after flipping twice represents the same spin structure, 
    but algebraically it has the effect of negating the $\mu$-invariant of that triangle
    ($\theta \mapsto -\theta$ in \Cref{fig:flip_spin}). This means that the specific $\mu$-invariants are not a feature of the triangulation
    and spin structure alone, but also the choice of representative orientation. Choosing a different orientation corresponds to
    changing the sign of some of the $\mu$-invariants.
    On the other hand, the expressions of $\lambda$-lengths in terms of an initial triangulation are independent of the orientation of 
    the arc as part of a spin structure, and of the flip sequence used to obtain a triangulation containing that arc.  
    This is proven in \cite{pz_19} in the case of surfaces without boundaries. In \cite{moz21}, the authors proved the well-definedness 
    of super $\lambda$-lengths for marked disks by deriving a super Pentagon relation (see \cite[Appendix A.]{moz21}).
\end{remark}

\section {Fan Decomposition, Default Orientation, and Positive Order}\label{sec:3}

In this section, we recall some definitions and conventions used in \cite{moz21}. In particular, given a triangulation of a polygon,
and an (oriented) arc $i \to j$ not in the triangulation,
we define an orientation of the arcs in the triangulation called the \emph{default orientation}, and
a total order on the set of triangles (and odd generators of the super algebra) called the \emph{positive order}.

Let $P$ be a polygon with $n+3$ vertices, i.e. a disk with $n+3$ marked points on the boundary.
Let $a$ and $b$ be two non-adjacent vertices on the boundary and let $(a,b)$ be the arc that connects $a$ and $b$. Without loss of generality, we assume that 
$(a,b)$ crosses all internal diagonals of $T$. (For the purpose of $\lambda$-length calculation, 
triangles that do not intersect with the arc $(a,b)$ can simply be removed from the picture.)

\subsection {Decomposing a Triangulation into Fans}

A triangulation is called a \emph{fan} if all the internal diagonals share a common vertex.
We will introduce a canonical way to break any triangulation $T$ of $P$ into smaller ones that are fans. 
For this purpose, certain vertices of $P$ will be distinguished as \emph{fan centers}.

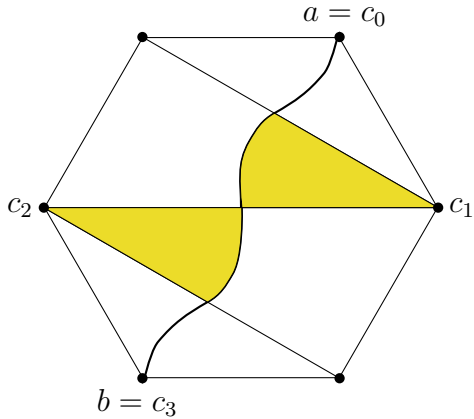
\begin{figure}[h]
	\centering
	\begin{tikzpicture}[scale=1.1]
	\tikzstyle{every path}=[draw] 
		\path
    node[
      regular polygon,
      regular polygon sides=6,
      draw,
      inner sep=1.6cm,
    ] (hexagon) {}
    %
    (hexagon.corner 1) node[above] {$\ a=c_0$}
    (hexagon.corner 3) node[left] {$c_2$}
    (hexagon.corner 4) node[below] {$b=c_3\ $}
    (hexagon.corner 6) node[right] {$c_1$}

  ;
  \coordinate (m1) at (1,1.62);
  \coordinate (m2) at (0.1,0.8);
  \coordinate (m3) at (-0.1,-0.8);
  \coordinate (m4) at (-1,-1.62);
  
  \draw [name path = m ,thin,opacity=0] (hexagon.corner 1) to [out=90,in=50] (m1) to [out=-180+50,in=60] (m2) to [out=-180+60,in=60] (m3) to [out=-180+60,in=50] (m4) to [out=-180+50,in=-180+90](hexagon.corner 4);%
  
  \draw [name path = L1] (hexagon.corner 6) to (hexagon.corner 2);
  \draw [name path = L2] (hexagon.corner 6) to (hexagon.corner 3);
  \draw [name path = L3] (hexagon.corner 5) to (hexagon.corner 3);
  
  \path [name intersections={of= m and L1,by=X1}];
  \path [name intersections={of= m and L2,by=X2}];
  \path [name intersections={of= m and L3,by=X3}];
  
  \draw [name path = m ,thick, black] (hexagon.corner 1) to [out=90,in=50] (m1) to [out=-180+50,in=30] (X1) to [out=-180+30,in=60] (m2) to [out=-180+60,in=90] (X2) to [out=90,in=60] (m3) to [out=-180+60,in=30] (X3) to [out=-180+30,in=50] (m4) to [out=-180+50,in=-180+90](hexagon.corner 4);

  \draw [fill=yellow!90!black] (X1) to [out=-180+30,in=60] (m2) to [out=-180+60,in=90] (X2) to (hexagon.corner 6) to (X1);
  
  \draw[fill=yellow!90!black](X2) to [out=90,in=60] (m3) to [out=-180+60,in=30] (X3) to (hexagon.corner 3)--(X2);

  

  \draw (hexagon.corner 1) node [fill,circle,scale=0.35] {};
  \draw (hexagon.corner 2) node [fill,circle,scale=0.35] {};
  \draw (hexagon.corner 3) node [fill,circle,scale=0.35] {};
  \draw (hexagon.corner 4) node [fill,circle,scale=0.35] {};
  \draw (hexagon.corner 5) node [fill,circle,scale=0.35] {};
  \draw (hexagon.corner 6) node [fill,circle,scale=0.35] {};
  
  \end{tikzpicture}
		\caption{centers of fan segments.}
		\label{fig:fan_centers}
	\end{figure}

\begin{definition}
    The intersections of arc $(a,b)$ and internal diagonals of $T$ create small triangles (colored yellow in \Cref{fig:fan_centers}), 
    whose vertices in $P$ are called \emph{fan centers}.
    We set $a=c_0$ and $b=c_{N+1}$ and as a convention we will name these centers $c_1,\cdots,c_N$ such that 
    \begin{enumerate}[nosep]
        \item The edge $(c_i,c_{i+1})$ is in $T$ which crosses $(a,b)$ for $1\leq i\leq N-1$;
        \item The intersection $(c_i,c_{i+1}) \cap (a,b)$ is closer to $a$ than $(c_j,c_{j+1}) \cap (a,b)$ if $i<j$. 
    \end{enumerate} 
    Now the polygon $P$ is broken into smaller polygons by the edges $(c_i,c_{i+1})$, each of which 
    comes with a fan triangulation induced from $T$. 
    Moreover, the sub-triangulation of $T$ bounded by $c_{i-1}$, $c_{i}$ and $c_{i+1}$ is called the $i$-th \emph{fan segment} of $T$, whose \emph{center} is said to be $c_i$.
\end{definition}

\subsection{Default Orientation and Positive Order}

For a triangulation $T$ of $P$ and the (directed) arc $(a,b)$ which crosses all internal diagonals of $T$, we define a \emph{default orientation}. 
Such orientation determines an ordering of the $\mu$-invariants, which we call the \emph{positive order}, in which the super $\lambda$-length expansion has positive coefficients.
Note that we will omit the orientation of boundary edges because only the orientation of interior edges affects our calculation of super $\lambda$-lengths.

\begin{definition}
    When the triangulation is a single fan with $c_1$ being the center, every interior edge is oriented away from $c_1$. 
    When $T$ is a triangulation with $N>1$ fans, where $c_1,\cdots,c_N$ are the centers, the interior edges within each fan segment are oriented away 
    from its center. The edges where two fans meet each other are oriented as 
    $c_1\rightarrow c_2\rightarrow\cdots\rightarrow c_{N-1}\rightarrow c_N$.
    See \Cref{fig:default_spin}.
\end{definition}

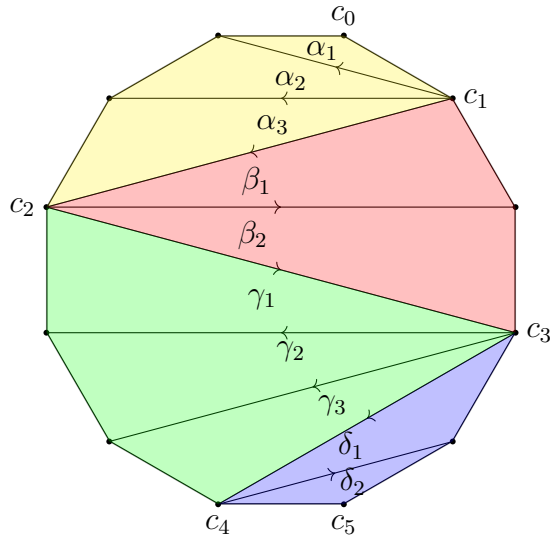
\begin{figure}[h!]
\centering
\begin{tikzpicture}[decoration={
    markings,
    mark=at position 0.5 with {\arrow{>}}},
    scale=1.2] 

\tikzstyle{every path}=[draw] 
		\path
    node[
      regular polygon,
      regular polygon sides=12,
      draw,
      inner sep=2.2cm,
    ] (hexagon) {}
    %
    (hexagon.corner 1) node[above] {$c_0$}
    (hexagon.corner 2) node[above] {}
    (hexagon.corner 3) node[left] {}
    (hexagon.corner 4) node[left] {$c_2$}
    (hexagon.corner 5) node[below] {}
    (hexagon.corner 6) node[right] {}
    (hexagon.corner 7) node[below] {$c_4$}
    (hexagon.corner 8) node[below] {$c_5$}
    (hexagon.corner 9) node[left] {}
    (hexagon.corner 10) node[right] {$c_3$}
    (hexagon.corner 11) node[below] {}
    (hexagon.corner 12) node[right] {$c_1$}
;

\foreach \x in {1,2,...,12}{
\draw (hexagon.corner \x) node [fill,circle,scale=0.2] {};}

\draw[postaction={decorate}] (hexagon.corner 12)--(hexagon.corner 4);
\draw[postaction={decorate}] (hexagon.corner 12)--(hexagon.corner 2);
\draw[postaction={decorate}] (hexagon.corner 12)--(hexagon.corner 3);

\draw[postaction={decorate}](hexagon.corner 4)--(hexagon.corner 10);
\draw[postaction={decorate}](hexagon.corner 4)--(hexagon.corner 11);

\draw[postaction={decorate}](hexagon.corner 10)--(hexagon.corner 7);
\draw[postaction={decorate}](hexagon.corner 10)--(hexagon.corner 5);
\draw[postaction={decorate}](hexagon.corner 10)--(hexagon.corner 6);

\draw[postaction={decorate}](hexagon.corner 7)--(hexagon.corner 9);

\draw[fill=yellow, nearly transparent] (hexagon.corner 12)--(hexagon.corner 4)--(hexagon.corner 3)--(hexagon.corner 2)--(hexagon.corner 1)--cycle;
\draw[fill=red, nearly transparent] (hexagon.corner 12)--(hexagon.corner 4)--(hexagon.corner 10)--(hexagon.corner 11)--cycle;

\draw[fill=green, nearly transparent] (hexagon.corner 10)-- (hexagon.corner 4)--(hexagon.corner 5)--(hexagon.corner 6)--(hexagon.corner 7)--cycle;
\draw[fill=blue, nearly transparent] (hexagon.corner 10)--(hexagon.corner 7)--(hexagon.corner 8)--(hexagon.corner 9)--cycle;

\coordinate (m 1) at  ($0.38*(hexagon.corner 2)+0.38*(hexagon.corner 1)+0.24*(hexagon.corner 12)$);
\coordinate (m 2) at  ($0.28*(hexagon.corner 2)+0.28*(hexagon.corner 3)+0.44*(hexagon.corner 12)$);
\coordinate (m 3) at  ($0.24*(hexagon.corner 4)+0.24*(hexagon.corner 3)+0.52*(hexagon.corner 12)$);
\coordinate (m 4) at  ($0.52*(hexagon.corner 4)+0.24*(hexagon.corner 11)+0.24*(hexagon.corner 12)$);
\coordinate (m 5) at  ($0.56*(hexagon.corner 4)+0.22*(hexagon.corner 11)+0.22*(hexagon.corner 10)$);
\coordinate (m 6) at  ($0.27*(hexagon.corner 4)+0.27*(hexagon.corner 5)+0.46*(hexagon.corner 10)$);
\coordinate (m 7) at  ($0.23*(hexagon.corner 6)+0.23*(hexagon.corner 5)+0.44*(hexagon.corner 10)$);
\coordinate (m 8) at  ($0.26*(hexagon.corner 6)+0.26*(hexagon.corner 7)+0.48*(hexagon.corner 10)$);
\coordinate (m 9) at  ($0.25*(hexagon.corner 9)+0.5*(hexagon.corner 7)+0.25*(hexagon.corner 10)$);
\coordinate (m 10) at  ($0.36*(hexagon.corner 9)+0.24*(hexagon.corner 7)+0.4*(hexagon.corner 8)$);

\node at (m 1) {$\alpha_1$};
\node at (m 2) {$\alpha_2$};
\node at (m 3) {$\alpha_3$};
\node at (m 4) {$\beta_1$};
\node at (m 5) {$\beta_2$};
\node at (m 6) {$\gamma_1$};
\node at (m 7) {$\gamma_2$};
\node at (m 8) {$\gamma_3$};
\node at (m 9) {$\delta_1$};
\node at (m 10) {$\delta_2$};

\end{tikzpicture}
\caption{The default orientation of a generic triangulation where each fan segment is colored differently. The faces are labelled by their $\mu$-invariants.}

\label{fig:default_spin}
	
\end{figure}

\begin{remark}
    As mentioned above, the definition of default orientation depends on the choice of direction $a \to b$.
    In particular, choosing the opposite direction $b \to a$ would change the labelling so that $c_i$ becomes $c_{N-i}$.
    The effect is that the orientation of the diagonals within a fan are unchanged, but the diagonals connecting 
    two fan centers would have the reverse orientation.
\end{remark}

\begin{definition} \label{def:pos_order}
    We define the \emph{positive ordering} inductively, triangle-by-triangle, as follows:
    Let the triangles be labelled $\theta_1, \theta_2, \dots, \theta_n$ in order from $a$ to $b$.
    For each triangle $\theta_k$, we look at the edge separating $\theta_k$ and $\theta_{k+1}$: 
    \begin{enumerate}
    	\item If the edge is oriented so that $\theta_k$ is to the right,
              then we declare that $\theta_k > \theta_i$ for all $i > k$.
        \item Otherwise, if $\theta_k$ is to the left, we declare that $\theta_k < \theta_i$ for all $i > k$.
    \end{enumerate}
\end{definition}

For example, in \Cref{fig:default_spin}, the positive ordering on the faces is
\[\alpha_1>\alpha_2>\alpha_3>\gamma_1>\gamma_2>\gamma_3> \delta_2>\delta_1>\beta_2>\beta_1.\]
       
\begin{remark} \label{rem:pos_order}
    Note that inside each of the fan segments, the triangles are ordered counterclockwise around the fan center. 
    For a more detailed description of the positive order, see Section 4 of \cite{moz21}.
\end{remark}

\bigskip

From now on, we will always assume our triangulation has the default orientation on its interior diagonals, 
and state our theorems under this assumption. This is allowed because of \Cref{lm:spin}. If we start with an arbitrary orientation, 
we can first apply a sequence of equivalence relations (reversing the arrows around a triangle and negating the $\mu$-invariant) 
to get to the default orientation, and apply our theorems therein.

\bigskip

\section {Snake Graphs}

\bigskip

\begin {definition}
    A \emph{snake graph} is a planar graph consisting of a sequence of square \emph{tiles},
    each connected to either the top or right side of the previous tile. Given a snake graph $G$,
    the \emph{word} of $G$, denoted $W(G)$, is a string in the alphabet $\{\text{R,U}\}$ (standing for ``\emph{right}'' and ``\emph{up}'')
    indicating how each tile is connected to the previous. See \Cref{fig:snake_graph_examples} for examples.
\end {definition}

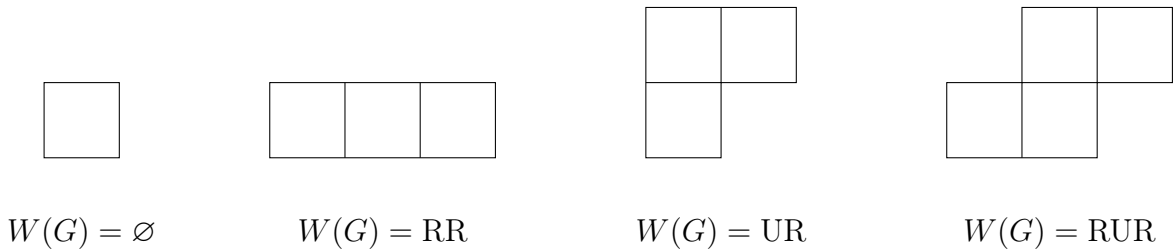
\begin {figure}[h!]
\centering
\begin {tikzpicture}
    \begin {scope}[shift={(-3,0)}]
        \draw (0,0) -- (1,0) -- (1,1) -- (0,1) -- cycle;

        \draw (0.5,-1) node {$W(G) = \varnothing$};
    \end {scope}

    \draw (0,0) -- (3,0) -- (3,1) -- (0,1) -- cycle;
    \draw (1,0) -- (1,1);
    \draw (2,0) -- (2,1);

    \draw (1.5,-1) node {$W(G) = \text{RR}$};

    \begin {scope}[shift={(5,0)}]
        \draw (0,0) -- (0,2) -- (2,2) -- (2,1) -- (1,1) -- (1,0) -- cycle;
        \draw (0,1) -- (1,1) -- (1,2);

        \draw (1,-1) node {$W(G) = \text{UR}$};
    \end {scope}

    \begin {scope}[shift={(9,0)}]
        \draw (0,0) -- (2,0) -- (2,1) -- (3,1) -- (3,2) -- (1,2) -- (1,1) -- (0,1) -- cycle;
        \draw (1,0) -- (1,1) -- (2,1) -- (2,2);

        \draw (1.5,-1) node {$W(G) = \text{RUR}$};
    \end {scope}
\end {tikzpicture}
\caption {Examples of snake graphs}
\label {fig:snake_graph_examples}
\end {figure}

We will now review a construction from \cite{ms10}, which associates to each arc of a triangulated polygon
a labelled snake graph. In the remainder of this section, we assume we are given a polygon with $n$ sides,
and a chosen triangulation. As in \cite{moz21}, we label the edges and triangular faces of the triangulation 
with the corresponding generators of the super algebra.

\bigskip

\begin {definition} \label{def:Ggamma}
    Let $\gamma$ be a diagonal of the polygon. We will construct a snake graph $G_\gamma$.
    Suppose vertices $a$ and $b$ are the endpoints of $\gamma$. If we traverse $\gamma$ from $a$ to $b$,
    let $x_1,\dots,x_k$ be the labels of the diagonals crossed by $\gamma$, in order. Without loss
    of generality, we suppose that $\gamma$ is a \emph{longest edge}, in which case $k=n-3$. We will also
    always draw our polygons so that $a$ is at the bottom and $b$ is at the top (the arc $\gamma$ is oriented bottom-to-top).

    There will be one tile of $G_\gamma$ for each diagonal $x_i$. For each $i$, the diagonal $x_i$
    is the common side of two triangles in the triangulation, and is thus the diagonal of a quadrilateral.
    We make a square tile $T_i$ whose four sides are labelled the same as the sides of this quadrilateral.
    If $i$ is odd, then the orientation of $T_i$ matches that of the triangulation, and if $i$ is even,
    then the orientation of $T_i$ is reversed. Examples are shown in \Cref{fig:snake_tiles}. By convention,
    we draw the first tile $T_1$ so that the endpoint of $\gamma$ is the bottom left corner.

    Although we do not draw the diagonals in the snake graph tiles, we still speak of $x_i$ as being ``\emph{the diagonal of tile $T_i$}''.
    These diagonals, although not drawn, should be thought of as separating the $\theta_i$'s labelling the triangles. These $\theta$
    labels are positioned in the corners of the tiles corresponding to their position in the triangulation.

    \begin {figure}[h!]
    \centering
    \begin {tikzpicture}[scale=1.1, every node/.style = {scale=0.8}]
        \draw (0:1.5) -- (60:1.5) -- (120:1.5) -- (180:1.5) -- (240:1.5) -- (300:1.5) -- cycle;
        \draw (120:1.5) --node[above]{$x_3$} (0:1.5) --node[above]{$x_2$} (180:1.5) --node[right]{$x_1$} (-60:1.5);

        \draw (210:1.5) node {$y_1$};
        \draw (270:1.5) node {$y_2$};
        \draw (330:1.5) node {$y_3$};
        \draw (30:1.5)  node {$y_4$};
        \draw (90:1.5)  node {$y_5$};
        \draw (150:1.5) node {$y_6$};

        \draw (240:1.15) node {$\theta_1$};
        \draw (310:0.65) node {$\theta_2$};
        \draw (130:0.65) node {$\theta_3$};
        \draw (60:1.2)   node {$\theta_4$};

        \draw[blue, line width=1.5] (180:1.5) -- (240:1.5) -- (300:1.5) -- (0:1.5) -- cycle;

        \begin {scope} [shift={(-0.75,-4)}]
            \draw (-1.5,0.75) node {$T_1$};
            \draw (-1,0.75) node {$\colon$};
            \draw (0,0) --node[below]{$y_2$} (1.5,0) --node[right]{$y_3$} (1.5,1.5) --node[above]{$x_2$} (0,1.5) --node[left]{$y_1$} (0,0);
            \draw (0.3,0.3) node {$\theta_1$};
            \draw (1.2,1.2) node {$\theta_2$};
        \end {scope}

        \begin {scope}[shift={(5,0)}]
            \draw (0:1.5) -- (60:1.5) -- (120:1.5) -- (180:1.5) -- (240:1.5) -- (300:1.5) -- cycle;
            \draw (120:1.5) --node[above]{$x_3$} (0:1.5) --node[above]{$x_2$} (180:1.5) --node[right]{$x_1$} (-60:1.5);

            \draw (210:1.5) node {$y_1$};
            \draw (270:1.5) node {$y_2$};
            \draw (330:1.5) node {$y_3$};
            \draw (30:1.5)  node {$y_4$};
            \draw (90:1.5)  node {$y_5$};
            \draw (150:1.5) node {$y_6$};

            \draw (240:1.15) node {$\theta_1$};
            \draw (310:0.65) node {$\theta_2$};
            \draw (130:0.65) node {$\theta_3$};
            \draw (60:1.2)   node {$\theta_4$};

            \draw[red, line width=1.5] (180:1.5) -- (-60:1.5) -- (0:1.5) -- (120:1.5) -- cycle;

            \begin {scope} [shift={(-0.75,-4)}]
                \draw (-1.5,0.75) node {$T_2$};
                \draw (-1,0.75) node {$\colon$};
                \draw (0,0) --node[below]{$x_1$} (1.5,0) --node[right]{$y_6$} (1.5,1.5) --node[above]{$x_3$} (0,1.5) --node[left]{$y_3$} (0,0);
                \draw (0.3,0.3) node {$\theta_2$};
                \draw (1.2,1.2) node {$\theta_3$};
            \end {scope}
        \end {scope}
    \end {tikzpicture}
    \caption {Tiles corresponding to diagonals in a triangulation}
    \label {fig:snake_tiles}
    \end {figure}
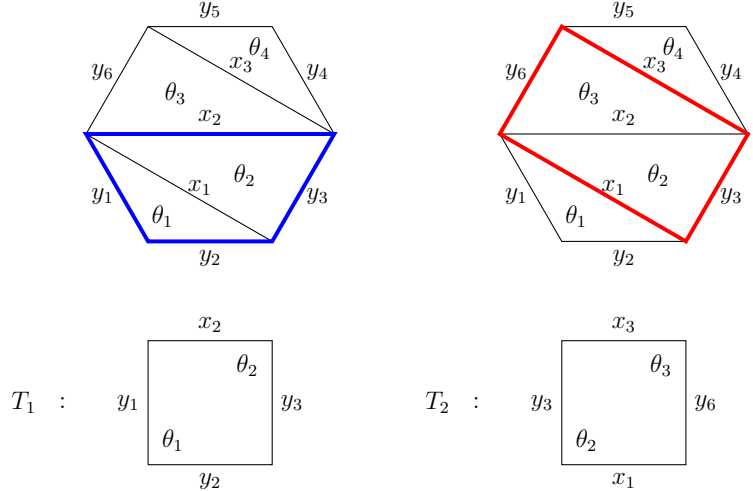

    By construction, tiles $T_i$ and $T_{i+1}$ share a triangle in common. Two of the sides of this triangle are the diagonals $x_i$ and $x_{i+1}$.
    Let $e$ denote the label on the third side of this triangle. The tiles $T_i$ and $T_{i+1}$ will both have a side labelled $e$ (which will be
    either the top or right edge of $T_i$). We will glue the tiles along the edge $e$.

    Examples of snake graphs are pictured in \Cref{fig:labelled_snake_graphs}.
\end {definition}

\begin {figure}[h!]
\centering
\begin {tikzpicture}[scale=1.1, every node/.style = {scale=0.8}]
    \draw (0:1.5) -- (60:1.5) -- (120:1.5) -- (180:1.5) -- (240:1.5) -- (300:1.5) -- cycle;
    \draw (120:1.5) --node[above]{$x_3$} (0:1.5) --node[above]{$x_2$} (180:1.5) --node[right]{$x_1$} (-60:1.5);

    \draw (210:1.5) node {$y_1$};
    \draw (270:1.5) node {$y_2$};
    \draw (330:1.5) node {$y_3$};
    \draw (30:1.5)  node {$y_4$};
    \draw (90:1.5)  node {$y_5$};
    \draw (150:1.5) node {$y_6$};

    \draw (240:1.15) node {$\theta_1$};
    \draw (310:0.65) node {$\theta_2$};
    \draw (130:0.65) node {$\theta_3$};
    \draw (60:1.2)   node {$\theta_4$};

    \draw[dashed] (240:1.5) -- (60:1.5);

    \begin {scope} [shift={(5,-0.75)}]
        \draw (-1.5,0.75) node {$G_\gamma$};
        \draw (-1,0.75) node {$\colon$};
        \draw (0,0)   --node[below]{$y_2$} (1.5,0) --node[right]{$y_3$} (1.5,1.5) --node[above]{$x_2$} (0,1.5) --node[left]{$y_1$} (0,0);
        \draw (1.5,0) --node[below]{$x_1$} (3,0)   --node[right]{$y_6$} (3,1.5)   --node[above]{$x_3$} (1.5,1.5);
        \draw (3,0)   --node[below]{$x_2$} (4.5,0) --node[right]{$y_4$} (4.5,1.5) --node[above]{$y_5$} (3,1.5);

        \draw (0.3,0.3) node {$\theta_1$};
        \draw (1.2,1.2) node {$\theta_2$};
        \draw (1.8,0.3) node {$\theta_2$};
        \draw (2.7,1.2) node {$\theta_3$};
        \draw (3.3,0.3) node {$\theta_3$};
        \draw (4.2,1.2) node {$\theta_4$};
    \end {scope}

    \begin {scope}[shift={(0,-4)}]
        \draw (0:1.5) -- (60:1.5) -- (120:1.5) -- (180:1.5) -- (240:1.5) -- (300:1.5) -- cycle;
        \draw (0:1.5) --node[above]{$x_2$} (180:1.5) --node[right]{$x_1$} (-60:1.5);
        \draw (60:1.5) --node[above]{$x_3$} (180:1.5);

        \draw (210:1.5) node {$y_1$};
        \draw (270:1.5) node {$y_2$};
        \draw (330:1.5) node {$y_3$};
        \draw (30:1.5)  node {$y_4$};
        \draw (90:1.5)  node {$y_5$};
        \draw (150:1.5) node {$y_6$};

        \draw (240:1.15) node {$\theta_1$};
        \draw (310:0.65) node {$\theta_2$};
        \draw (40:0.65)  node {$\theta_3$};
        \draw (120:1.2)  node {$\theta_4$};

        \draw[dashed] (240:1.5) -- (120:1.5);

        \begin {scope} [shift={(5,-0.75)}]
            \draw (-1.5,0.75) node {$G_\gamma$};
            \draw (-1,0.75) node {$\colon$};
            \begin {scope} [shift={(0,-0.75)}]
                \draw (0,0)   --node[below]{$y_2$} (1.5,0) --node[right]{$y_3$} (1.5,1.5) --node[above]{$x_2$} (0,1.5) --node[left]{$y_1$} (0,0);
                \draw (1.5,0) --node[below]{$x_1$} (3,0)   --node[right]{$x_3$} (3,1.5)   --node[above]{$y_4$} (1.5,1.5);
                \draw (3,1.5) --node[right]{$y_5$} (3,3)   --node[above]{$y_6$} (1.5,3)   --node[left]{$x_2$} (1.5,1.5);

                \draw (0.3,0.3) node {$\theta_1$};
                \draw (1.2,1.2) node {$\theta_2$};
                \draw (1.8,0.3) node {$\theta_2$};
                \draw (2.7,1.2) node {$\theta_3$};
                \draw (1.8,1.8) node {$\theta_3$};
                \draw (2.7,2.7) node {$\theta_4$};
            \end {scope}
        \end {scope}
    \end {scope}
\end {tikzpicture}
\caption {Snake graph corresponding to a diagonal $\gamma$ (drawn dashed)}
\label {fig:labelled_snake_graphs}
\end {figure}
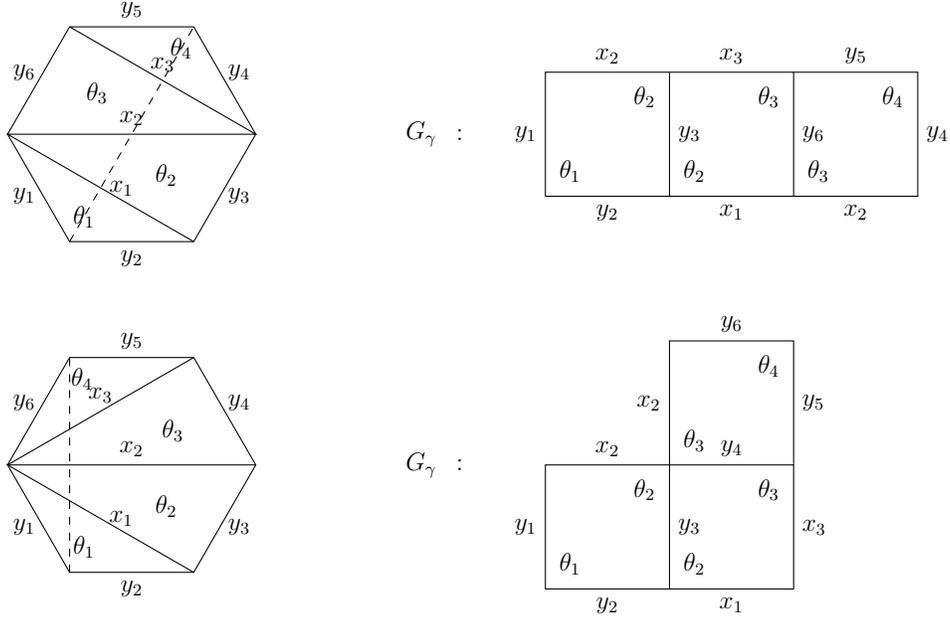

\bigskip

\begin {definition}
    Let $\gamma = (i,j)$ be a diagonal of a triangulated polygon. Define $\mathrm{cross}(\gamma) = \mathrm{cross}(i,j)$ to be
    the monomial $x_1 x_2 \cdots x_k$, the product of all diagonals which cross $\gamma$.
\end {definition}

\bigskip

\begin {theorem} \cite{ms10} \label{thm:MS10}
    Let $\mathcal{A}$ be the cluster algebra coming from a polygon, with initial seed given by a chosen triangulation,
    and let $\gamma=(i,j)$ be a diagonal of the polygon. Then
    \[ x_{ij} = \frac{1}{\mathrm{cross}(\gamma)} \sum_{M \in D(G_\gamma)} \mathrm{wt}(M) \]
\end {theorem}

Here, $D(G_\gamma)$ denotes the collection of dimer covers on the snake graph $G_\gamma$.  
In particular, each $M$ is a subset of edges on $G_\gamma$ such that every vertex of $G_\gamma$ is incident to exactly one edge of $M$.  
The weight of $M$, denoted as $\mathrm{wt}(M)$, is the product of all edge labels for all edges in $M$.  
In the next section, we extend these definitions to the case of \emph{double} dimer covers.

\bigskip

\section {Double Dimer Covers}

\begin {definition}
    If $G$ is a planar bipartite graph, a \emph{double dimer cover} of $G$ is a multiset $M$ of edges
    such that each vertex of $G$ is incident to exactly two edges from $M$ (which are allowed to be two
    copies of the same edge). Given a double dimer cover $M$, we call each element of $M$ a \emph{dimer}.
    Dimers will be pictured as wavy orange lines, and two overlapping dimers (called a \emph{double dimer}) 
    will be pictured as a solid blue line.
\end {definition}

\bigskip

\begin {definition}
    Let $G$ be a snake graph, and let $T$ denote the last tile of $G$. Then we define:
    
    \begin {itemize}
        \item $D(G)$ is the set of all double dimer covers  
        on $G$.
        \item $D_R(G)$ is the set of $M \in D(G)$ which use two dimers on the right edge of $T$.
        \item $D_T(G)$ is the set of $M \in D(G)$ which use two dimers on the top edge of $T$.
        \item $D_{tr}(G)$ is the set of $M \in D(G)$ which use single dimers on the top and right edges of $T$.
        \item $D_r(G)$ is the set of $M \in D(G)$ which use at least one dimer on the right edge of $T$.
        \item $D_t(G)$ is the set of $M \in D(G)$ which use at least one dimer on the top edge of $T$.
    \end {itemize}
\end {definition}

\bigskip

Looking at the top-right vertex of the last tile, we see that $D(G)$ is the disjoint union 
\[ D(G) = D_R(G) \cup D_T(G) \cup D_{tr}(G) \]

\bigskip

\begin {remark}
    Note that $D_r(G) = D_R(G) \cup D_{tr}(G)$ and $D_t(G) = D_T(G) \cup D_{tr}(G)$.
\end {remark}

\bigskip

\begin {definition} \label{def:weight-M}
    If $G$ is a snake graph, and $M \in D(G)$, we will define the \emph{weight} of $M$, denoted $\mathrm{wt}(M)$,
    as a monomial of the super algebra. The weight of a dimer (an element of the multiset $M$)
    is the square root of the label of the corresponding edge of $G$. Let $c(M)$ be the set of cycles formed by the
    edges of $M$. For a cycle $C \in c(M)$, let $\theta_i$ be the odd variable corresponding to the triangle in the
    bottom-left corner of $C$, and $\theta_j$ for the top-right corner. We define the weight of the cycle to be $\mathrm{wt}(C) = \theta_i \theta_j$.
    Finally, we define the weight of $M$ as
    \[ \mathrm{wt}(M) :=\prod_{e \in M} \mathrm{wt}(e) \prod_{C \in c(M)} \mathrm{wt}(C)\]
where the products are taken under the positive order of the underlying triangulation.
\end {definition}

Note that we can map any dimer cover of $G$ to a double dimer cover of $G$ by turning each dimer edge into a double dimer, i.e. a doubled edge.  
For such $M$'s, the weight of $M$, as defined here, versus following the definition from \Cref{thm:MS10} coincide.  
Thus there is no abuse in notation for using $\mathrm{wt}(M)$ in both settings.

\bigskip

\section {Double Dimer Recurrences}

\bigskip

In this section, we give some recurrences for double dimer covers  
on snake graphs. 
In particular, we establish bijections between each of the sets $D_R(G)$, $D_T(G)$, and $D_{tr}(G)$
with certain subsets of double dimer covers  
on smaller snake graphs. We also give explicit
expressions for how the weights transform under these bijections.

\begin {remark}
    Although we usually start with a triangulation, and build a snake graph from it (as in \Cref{def:Ggamma}),
    in this section we instead start with a labelled snake graph. It is not hard to see that the construction
    of \Cref{def:Ggamma} can be reversed, and the triangulation can be reconstructed from the snake graph.
    As such, we will still speak of the ``\emph{diagonal}'' of a tile, meaning the label
    on the corresponding diagonal of the triangulation.
\end {remark}

\bigskip

\begin {definition}
    Define $G^{(-k)}$ to be the snake graph $G$ with the last $k$ tiles removed. 
\end {definition}

\bigskip

\begin {lemma} \label{lem:lemma1}
    \ \\
    \begin {itemize}
        \item[$(a)$] If $W(G)$ ends in  ``$R$'', then there is a bijection $f \colon D(G^{(-1)}) \to D_R(G)$,
        satisfying $\mathrm{wt}(f(M)) = a \cdot \mathrm{wt}(M)$, where ``$a$'' is the label on the right edge
        of the last tile. In particular,
        \[ \sum_{M \in D(G^{(-1)})} \mathrm{wt}(M) = \frac{1}{a} \sum_{M \in D_R(G)} \mathrm{wt}(M) \]
        \item[$(b)$] If $W(G)$ ends in  ``$U$'', then there is a bijection $f \colon D(G^{(-1)}) \to D_T(G)$,
        satisfying $\mathrm{wt}(f(M)) = a \cdot \mathrm{wt}(M)$, where ``$a$'' is the label on the top edge
        of the last tile. In particular,
        \[ \sum_{M \in D(G^{(-1)})} \mathrm{wt}(M) = \frac{1}{a} \sum_{M \in D_T(G)} \mathrm{wt}(M) \]
    \end {itemize}
\end {lemma}

\begin{proof}
A double dimer cover in $D(G^{(-1)})$ can be uniquely extended to a double dimer cover in $D_R(G)$ (resp. $D_T(G)$) by adjoining two copies of the edge labeled $a$.  See \Cref{fig:proof_c}.
\end{proof}

\begin {lemma} \label{lem:lemma2}
    \ \\
    \begin {itemize}
        \item[$(a)$] If $W(G)$ ends in either ``$RRURUR \cdots R$'' or ``$UURURU \cdots R$'' (of length $k \geq 2$), 
        then there is a bijection $f \colon D(G^{(-k)}) \to D_T(G)$ satisfying $\mathrm{wt}(f(M)) = b e_2 \cdots e_k \cdot \mathrm{wt}(M)$,
        where $e_1,\dots,e_k$ are the diagonals of the last $k$ tiles in reverse order,
        and $b$ is the label on the top side of the last tile. In particular,
        \[ \sum_{M \in D(G^{(-k)})} \mathrm{wt}(M) = \frac{1}{be_2 \cdots e_k} \sum_{M \in D_T(G)} \mathrm{wt}(G) \]
        \item[$(b)$] If $W(G)$ ends in either ``$UURURU \cdots U$'' or ``$RRURUR \cdots U$'' (of length $k \geq 2$), 
        then there is a bijection $f \colon D(G^{(-k)}) \to D_R(G)$ satisfying $\mathrm{wt}(f(M)) = b e_2 \cdots e_k \cdot \mathrm{wt}(M)$,
        where $e_1,\dots,e_k$ are the diagonals of the last $k$ tiles in reverse order,
        and $b$ is the label on the right side of the last tile. In particular,
        \[ \sum_{M \in D(G^{(-k)})} \mathrm{wt}(M) = \frac{1}{be_2 \cdots e_k} \sum_{M \in D_R(G)} \mathrm{wt}(G) \]
    \end {itemize}
\end {lemma}
\begin {proof}
    Supposing the snake graph ends in a staircase which ends with ``R'', then the last $k$ tiles of $G$ look
    like the picture in \Cref{fig:lemma2_proof_NEW}. Note that because of the snake graph, i.e. see \Cref{def:Ggamma}, if $e_i$ is the label of the diagonal of
    the $i$th tile counting from the end (e.g. $e_1$ is the label of the diagonal of the last tile), then the W edge or S edge (depending on whichever is a
    boundary edge) of the succeeding tile is labeled $e_i$.  Similarly, the N edge or E edge of the previous tile is labeled $e_i$.  It is easy to see that if the top
    edge of the last tile is occupied by two dimers, then the opposite edge (the bottom) must also have a double edge. Consequently, the left side of the
    previous tile must use a double edge. Going on, one can see that this uniquely determines a double dimer cover on the rest of the staircase
    segment where one edge of each corner is used. This is pictured in \Cref{fig:lemma2_proof_NEW} (Left). 
    It is easily seen that the weight of these dimers is the product $b e_2 \cdots e_k$.  Removing the last $k$ tiles leads to $G^{(-k)}$ as pictured in   
    \Cref{fig:lemma2_proof_NEW} (Right).
\end {proof}

   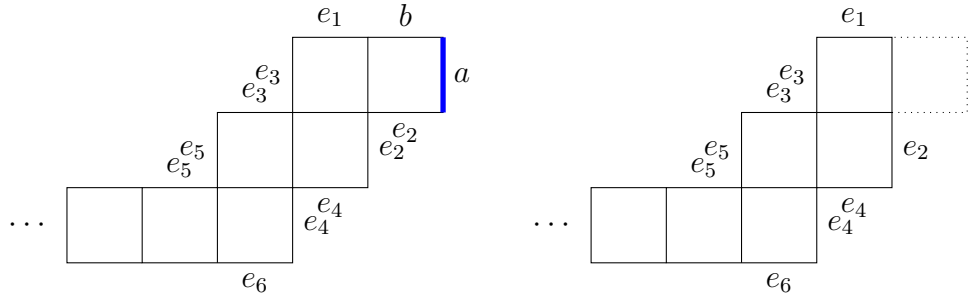
\begin{figure}[h]
   \begin {tikzpicture}[scale=1]
    \draw (0,0) -- (-1,0) -- (-1,1) -- (0,1);
    \draw (0,0) -- ++(2,0) -- ++(0,1) -- ++(1,0) -- ++(0,2) -- ++(-1,0) -- ++(0,-1) -- ++(-1,0) -- ++(0,-1) -- ++(-1,0) -- cycle;
    \draw (1,0) -- ++(0,1) -- ++(1,0) -- ++(0,1) -- ++(1,0);
    \draw (3,2) -- ++(1,0) -- ++(0,1) -- ++(-1,0);

    \draw (-1.5,0.5) node {$\cdots$};

    \draw (2.5,3) node[above] {$e_1$};
    
    \draw [blue,line width =2] (4,3)-- (4,2);
    \draw [draw=none] (4,3)--node [right] {$a$} (4,2);

    \draw (3.5,3) node[above] {$b$};

    \draw (3.5,2) node[below] {$e_2$};
    \draw (3,1.5) node[right] {$e_2$};

    \draw (1.5,2) node[above] {$e_3$};
    \draw (2,2.5) node[left]  {$e_3$};

    \draw (2.5,1) node[below] {$e_4$};
    \draw (2,0.5) node[right] {$e_4$};

    \draw (1,1.5) node[left]  {$e_5$};
    \draw (0.5,1) node[above] {$e_5$};

    \draw (1.5,0) node[below] {$e_6$};
\end {tikzpicture}
\quad
\begin {tikzpicture}[scale=1]
    \draw (0,0) -- (-1,0) -- (-1,1) -- (0,1);
    \draw (0,0) -- ++(2,0) -- ++(0,1) -- ++(1,0) -- ++(0,2) -- ++(-1,0) -- ++(0,-1) -- ++(-1,0) -- ++(0,-1) -- ++(-1,0) -- cycle;
    \draw (1,0) -- ++(0,1) -- ++(1,0) -- ++(0,1) -- ++(1,0);
    \draw [dotted] (3,2) -- ++(1,0) -- ++(0,1) -- ++(-1,0);

    \draw (-1.5,0.5) node {$\cdots$};

    \draw (2.5,3) node[above] {$e_1$};

    \draw (3,1.5) node[right] {$e_2$};

    \draw (1.5,2) node[above] {$e_3$};
    \draw (2,2.5) node[left]  {$e_3$};

    \draw (2.5,1) node[below] {$e_4$};
    \draw (2,0.5) node[right] {$e_4$};

    \draw (1,1.5) node[left]  {$e_5$};
    \draw (0.5,1) node[above] {$e_5$};

    \draw (1.5,0) node[below] {$e_6$};

\end {tikzpicture}
   \caption{Parts of double dimer covers  
   which use the edge $a$ twice}
   \label{fig:proof_c}
   \end{figure}

\begin {figure}[h]
\centering
\begin {tikzpicture}[scale=1]
    \draw (0,0) -- (-1,0) -- (-1,1) -- (0,1);
    \draw (0,0) -- ++(2,0) -- ++(0,1) -- ++(1,0) -- ++(0,2) -- ++(-1,0) -- ++(0,-1) -- ++(-1,0) -- ++(0,-1) -- ++(-1,0) -- cycle;
    \draw (1,0) -- ++(0,1) -- ++(1,0) -- ++(0,1) -- ++(1,0);
    \draw (3,2) -- ++(1,0) -- ++(0,1) -- ++(-1,0);

    \draw (-1.5,0.5) node {$\cdots$};

    \draw (3.5,3) node[above] {$b$};
    \draw[blue, line width = 2] (3,3) -- (4,3);

    \draw [draw=none] (4,3)--node [right] {$a$} (4,2);

    \draw (2.5,3) node[above] {$e_1$};

    \draw (3,1.5) node[right] {$e_2$};
    \draw (3.5,2) node[below] {$e_2$};
    \draw [blue, line width = 2] (3,2) -- (4,2);

    \draw (1.5,2) node[above] {$e_3$};
    \draw (2,2.5) node[left]  {$e_3$};
    \draw [blue, line width = 2] (2,2) -- (2,3);

    \draw (2.5,1) node[below] {$e_4$};
    \draw (2,0.5) node[right] {$e_4$};
    \draw [blue, line width = 2] (2,1) -- (3,1);

    \draw (1,1.5) node[left]  {$e_5$};
    \draw (0.5,1) node[above] {$e_5$};
    \draw[blue, line width = 2] (1,1) -- (1,2);

    \draw (1.5,0) node[below] {$e_6$};
    \draw[blue, line width = 2] (1,0) -- (2,0);    

\end {tikzpicture}
\quad
\begin {tikzpicture}[scale=1]
    \draw (0,0) -- (-1,0) -- (-1,1) -- (0,1) -- (0,0);
    \draw [dotted] (0,0) -- ++(2,0) -- ++(0,1) -- ++(1,0) -- ++(0,2) -- ++(-1,0) -- ++(0,-1) -- ++(-1,0) -- ++(0,-1) -- ++(-1,0) -- cycle;
    \draw [dotted] (1,0) -- ++(0,1) -- ++(1,0) -- ++(0,1) -- ++(1,0);
    \draw [dotted] (3,2) -- ++(1,0) -- ++(0,1) -- ++(-1,0);

    \draw (-1.5,0.5) node {$\cdots$};

    \draw (2.5,3) node[above] {$e_1$};
    
    \draw [draw=none] (4,3)--node [right] {$a$} (4,2);

    \draw (3.5,3) node[above] {$b$};

    \draw (3.5,2) node[below] {$e_2$};
    \draw (3,1.5) node[right] {$e_2$};

    \draw (1.5,2) node[above] {$e_3$};
    \draw (2,2.5) node[left]  {$e_3$};

    \draw (2.5,1) node[below] {$e_4$};
    \draw (2,0.5) node[right] {$e_4$};

    \draw (1,1.5) node[left]  {$e_5$};
    \draw (0.5,1) node[above] {$e_5$};

    \draw (1.5,0) node[below] {$e_6$};
\end {tikzpicture}

\caption {Parts of double dimer covers ending in a ``staircase'' which use the edge $b$ twice}
\label {fig:lemma2_proof_NEW}
\end {figure}
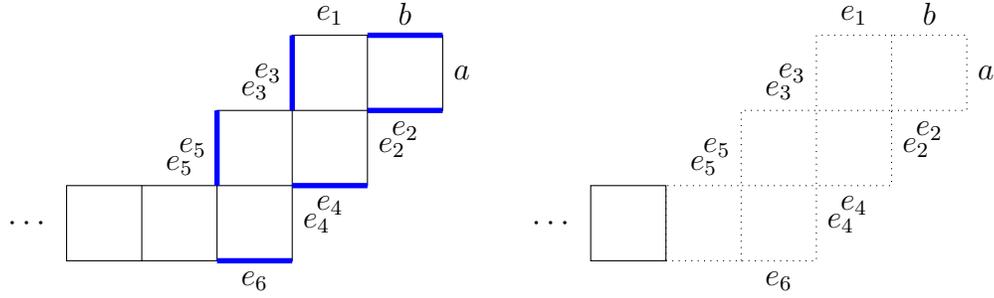

\bigskip

\begin {lemma} \label{lem:lemma3}
    \ \\
    \begin {itemize}
        \item[$(a)$] If $W(G) = (RU)^nR$ or $W(G) = (UR)^n$, 
        then $D_T(G)$ contains a unique double dimer cover  
        with weight $bc e_2 \cdots e_k$ (with $k=2n+2$ or $k=2n+1$),
        where $e_1,\dots,e_k$ are the diagonals of the last $k$ tiles in 
        reverse order, 
        $b$ is the label on the top side of the last tile, and $c$
        is the label on the left or bottom edge of the first tile.
        \item[$(b)$] If $W(G) = (UR)^nU$ or $W(G) = (RU)^n$, 
        then $D_R(G)$ contains a unique double dimer cover  
        with weight $bc e_2 \cdots e_k$ (with $k=2n+2$ or $k=2n+1$),
        where $e_1,\dots,e_k$ are the diagonals of the last $k$ tiles in 
        reverse order, 
        $b$ is the label on the right side of the last tile, and
        $c$ is the label on the left or bottom edge of the first tile.
    \end {itemize}
\end {lemma}
\begin {proof}
    The proof is essentially the same as for \Cref{lem:lemma2}, using \Cref{fig:lemma2_proof} which illustrates the case of part (a).
\end {proof}

\bigskip

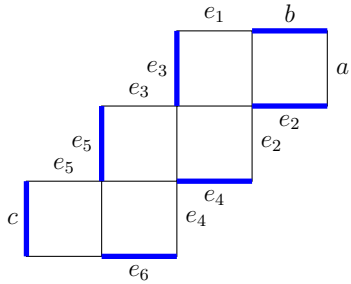
\begin {figure}[h]
\centering
\begin {tikzpicture}[scale=1.0, every node/.style = {scale=0.8}]
    \draw (0,0) -- ++(2,0) -- ++(0,1) -- ++(1,0) -- ++(0,2) -- ++(-1,0) -- ++(0,-1) -- ++(-1,0) -- ++(0,-1) -- ++(-1,0) -- cycle;
    \draw (1,0) -- ++(0,1) -- ++(1,0) -- ++(0,1) -- ++(1,0);
    \draw (3,2) -- ++(1,0) -- ++(0,1) -- ++(-1,0);

    \draw (4,2.5) node[right] {$a$};

    \draw (2.5,3) node[above] {$e_1$};

    \draw (3.5,3) node[above] {$b$};
    \draw[blue, line width = 2] (3,3) -- (4,3);

    \draw (3.5,2) node[below] {$e_2$};
    \draw (3,1.5) node[right] {$e_2$};
    \draw[blue, line width = 2] (3,2) -- (4,2);

    \draw (1.5,2) node[above] {$e_3$};
    \draw (2,2.5) node[left]  {$e_3$};
    \draw[blue, line width = 2] (2,2) -- (2,3);

    \draw (2.5,1) node[below] {$e_4$};
    \draw (2,0.5) node[right] {$e_4$};
    \draw[blue, line width = 2] (2,1) -- (3,1);

    \draw (1,1.5) node[left]  {$e_5$};
    \draw (0.5,1) node[above] {$e_5$};
    \draw[blue, line width = 2] (1,1) -- (1,2);

    \draw (1.5,0) node[below] {$e_6$};
    \draw[blue, line width = 2] (1,0) -- (2,0);

    \draw (0,0.5) node[left] {$c$};
    \draw[blue, line width = 2] (0,0) -- (0,1);
\end {tikzpicture}
\caption {The unique double dimer cover  
on a ``staircase'' using the edge $b$ twice}
\label {fig:lemma2_proof}
\end {figure}

\bigskip

We will now define an involution on the set of monomials of the super-algebra $A = \Bbb{R}[x_{ij}^{\pm 1/2} ~|~ \theta_k]$, given
by \emph{toggling} $\theta_n$.

\bigskip

\begin {definition} \label{def:star_involutions}
    Let $x \in A$ be a monomial, written in the positive order.
    If $x$ contains $\theta_n$, then $x^*$ is defined as $x$ with $\theta_n$ removed. If $x$ does not
    contain $\theta_n$, then $x^*$ is defined by inserting $\theta_n$ such that the positive order is preserved.
\end {definition}

\bigskip

\begin {example}
    Let $n=4$, and suppose the positive order is $\theta_1 > \theta_3 > \theta_4 > \theta_2$. 
    Then $(\theta_1 \theta_3 \theta_2)^* = \theta_1\theta_3 \theta_4\theta_2$, and $(\theta_1\theta_4)^* = \theta_1$.
\end {example}

\bigskip

\begin {lemma} \label{lem:lemma4}
    Suppose that $G$ contains $(n-1)$ tiles, so that $\theta_{n-1}$ (resp. $\theta_n$) labels the bottom-left (resp. top-right) triangle of the last tile.
    Let $x^\dagger$ and $x^\ast$ denote the involutions corresponding to $\theta_{n-1}$ and $\theta_n$, respectively.
    \begin {itemize}
        \item[$(a)$] If $W(G)$ ends in ``R'' , 
        then there is a bijection $f \colon D_r(G^{(-1)}) \to D_{tr}(G)$ satisfying
        $\mathrm{wt}(f(M)) = \sqrt{\frac{abc}{d}} \, \left( \mathrm{wt}(M)^\dagger \right)^\ast$,
        where ``$d$'' is the label on the left edge of the last tile, and $a,b,c$ are the labels on the other three sides.
        In particular,
        \[ \sum_{M \in D_r(G^{(-1)})} \mathrm{wt}(M) = \sqrt{\frac{d}{abc}} \sum_{M \in D_{tr}(G)} \left( \mathrm{wt}(M)^\ast \right)^\dagger \]
        \item[$(b)$] If $W(G)$ ends in ``U'' , 
        then there is a bijection $f \colon D_t(G^{(-1)}) \to D_{tr}(G)$ satisfying
        $\mathrm{wt}(f(M)) = \sqrt{\frac{abc}{d}} \, \left( \mathrm{wt}(M)^\dagger \right)^*$,
        where ``$d$'' is the label on the bottom edge of the last tile, and $a,b,c$ are the labels on the other three sides.
        In particular,
        \[ \sum_{M \in D_t(G^{(-1)})} \mathrm{wt}(M) = \sqrt{\frac{d}{abc}} \sum_{M \in D_{tr}(G)} \left( \mathrm{wt}(M)^\ast \right)^\dagger \]
    \end {itemize}
\end {lemma}
\begin {proof}
    Suppose $W(G)$ ends in ``R''. Let $M \in D_r(G^{(-1)})$. Then we define $f(M)$ by first erasing a dimer from the right edge of the last tile of $G^{(-1)}$,
    and then add single dimers on the top, right, and bottom edges of the last tile of $G$. This is illustrated in \Cref{fig:proof_lemma4}. Clearly,
    this has the effect of multiplying the weight of $M$ by a factor $\sqrt{\frac{abc}{d}}$. But it also changes the odd variables, since the cycles at the
    end surround a different set of tiles.

    There are two cases, corresponding to the two pictures in \Cref{fig:proof_lemma4}. First consider the top picture, corresponding to $M \in D_R(G^{(-1)})$.
    Since $M$ does not contain a cycle surrounding the last tile, $\mathrm{wt}(M)$ does not include the variable $\theta_{n-1}$. Looking at the figure,
    we see that $f(M)$ contains a cycle around just the last tile, which means we must multiply by $\theta_{n-1}\theta_n$. This is the same as applying both
    the $\dagger$ and $\ast$ involutions (the $\dagger$ multiplies by $\theta_{n-1}$ and the $\ast$ multiplies by $\theta_n$).  Following \Cref{def:star_involutions}, the odd variables, including $\theta_{n-1}$ and $\theta_n$, are written in the positive order.

    In the second case, $M$ has a cycle surrounding the last tile of $G^{(-1)}$, and so $\mathrm{wt}(M)$ includes a factor of $\theta_{n-1}$.
    The effect of $f$ is to extend the cycle around the last tile of $G$. So we need to exchange $\theta_{n-1}$ for $\theta_n$. Again applying both
    involutions acheives this effect, since $\dagger$ will remove $\theta_{n-1}$, and then $\ast$ will add $\theta_n$ in the proper place relative to positive order.

    \begin {figure}[h!]
    \centering
    \begin {tikzpicture}[scale=1.4, every node/.style = {scale=0.9}]
        \tikzset {decoration={snake, amplitude=0.5mm}}

        \draw (0,0) -- (2,0) -- (2,1) -- (0,1) -- cycle;
        \draw (1,0) -- (1,1);

        \draw (-0.5,0.5) node {$\cdots$};
        \draw (1.5,1) node[above] {$a$};
        \draw (2,0.5) node[right] {$b$};
        \draw (1.5,0) node[below] {$c$};
        \draw (1,0.5) node[left]  {$d$};

        \draw (1.8,0.8) node {\tiny $\theta_n$};
        \draw (1.3,0.2) node {\tiny $\theta_{n-1}$};
        \draw (0.7,0.8) node {\tiny $\theta_{n-1}$};

        \draw[blue, line width = 2] (1,0) -- (1,1);

        \draw[-latex] (2.5,0.5) -- (5,0.5);
        \draw (3.75,0.5) node[above] {$f$};

        \begin {scope} [shift={(6,0)}]
            \draw (0,0) -- (2,0) -- (2,1) -- (0,1) -- cycle;
            \draw (1,0) -- (1,1);
            \draw (-0.5,0.5) node {$\cdots$};
            \draw (1.5,1) node[above] {$a$};
            \draw (2,0.5) node[right] {$b$};
            \draw (1.5,0) node[below] {$c$};
            \draw (1,0.5) node[left]  {$d$};

            \draw (1.8,0.8) node {\tiny $\theta_n$};
            \draw (1.3,0.2) node {\tiny $\theta_{n-1}$};
            \draw (0.7,0.8) node {\tiny $\theta_{n-1}$};

            \draw[orange, line width = 2, decorate] (1,0) -- (1,1); 
            \draw[orange, line width = 2, decorate] (1,1) -- (2,1); 
            \draw[orange, line width = 2, decorate] (2,1) -- (2,0); 
            \draw[orange, line width = 2, decorate] (2,0) -- (1,0); 
        \end {scope}

        \begin {scope} [shift={(0,-2)}]
            \draw (0,0) -- (2,0) -- (2,1) -- (0,1) -- cycle;
            \draw (1,0) -- (1,1);
            \draw (-0.5,0.5) node {$\cdots$};
            \draw (1.5,1) node[above] {$a$};
            \draw (2,0.5) node[right] {$b$};
            \draw (1.5,0) node[below] {$c$};
            \draw (1,0.5) node[left]  {$d$};

            \draw (1.8,0.8) node {\tiny $\theta_n$};
            \draw (1.3,0.2) node {\tiny $\theta_{n-1}$};
            \draw (0.7,0.8) node {\tiny $\theta_{n-1}$};

            \draw[orange, line width = 2, decorate] (0,1) -- (1,1); 
            \draw[orange, line width = 2, decorate] (1,1) -- (1,0); 
            \draw[orange, line width = 2, decorate] (1,0) -- (0,0); 

            \draw[-latex] (2.5,0.5) -- (5,0.5);
            \draw (3.75,0.5) node[above] {$f$};
        \end {scope}

        \begin {scope} [shift={(6,-2)}]
            \draw (0,0) -- (2,0) -- (2,1) -- (0,1) -- cycle;
            \draw (1,0) -- (1,1);
            \draw (-0.5,0.5) node {$\cdots$};
            \draw (1.5,1) node[above] {$a$};
            \draw (2,0.5) node[right] {$b$};
            \draw (1.5,0) node[below] {$c$};
            \draw (1,0.5) node[left]  {$d$};

            \draw (1.8,0.8) node {\tiny $\theta_n$};
            \draw (1.3,0.2) node {\tiny $\theta_{n-1}$};
            \draw (0.7,0.8) node {\tiny $\theta_{n-1}$};

            \draw[orange, line width = 2, decorate] (0,1) -- (2,1); 
            \draw[orange, line width = 2, decorate] (2,1) -- (2,0); 
            \draw[orange, line width = 2, decorate] (2,0) -- (0,0); 
        \end {scope}
    \end {tikzpicture}
    \caption {Proof of \Cref{lem:lemma4}}
    \label {fig:proof_lemma4}
    \end {figure}
\end {proof}

\bigskip

\section {The Main Formula}

\bigskip

In this section we give a double dimer interpretation for the terms of the Laurent expansion of $\lambda$-lengths and certain
types of $\mu$-invariants. The proof will induct on the number of triangles in the polygon (equivalently, the number of tiles
of the corresponding snake graph). There are several cases to consider, with subtle differences, depending on the following factors:
\begin {itemize}
    \item Whether there are an even or odd number of triangles/tiles. This affects whether the last tile of the snake graph has 
          either normal or reversed orientation.
    \item Whether the word $W(G)$ of the snake graph ends in ``R'' or ``U''. This affects which versions (parts $(a)$ or $(b)$) of
          Lemmas \ref{lem:lemma1} -- \ref{lem:lemma4} must be used in the induction. This will correspond to interchanging $D_T(G)$ and $D_R(G)$
          in the recursion formulas.
    \item Whether the top-most fan center is on the left (as in \Cref{fig:proof}(a)) or the right. This affects which edge ($c$ or $d$
          in \Cref{fig:proof}) crosses more diagonals.
\end {itemize}

Rather than go through the proof for all possible cases, we will make particular choices for the factors listed above,
so that we are in the case pictured in \Cref{fig:proof}. Namely, we assume that the top-most fan center (the vertex labelled $j$ in \Cref{fig:proof})
is on the left (adjacent to edge $a$, and not $b$), and that the triangulation has an odd number of triangles.
This actually determines the third condition (whether $W(G)$ ends with ``R'' or ``U''),
which we now explain in the following lemma.

\begin {lemma} \label{lem:parity} \
    \begin {enumerate}
         \item[$(a)$] Suppose a triangulated polygon is oriented as in \Cref{fig:proof}, so that the top-most fan center (vertex $j$) is on the left side of the polygon.
                      Then $W(G)$ ends in ``R'' if and only if there are an odd number of triangles (even number of tiles in $G$), and
                      $W(G)$ ends in ``U'' if and only if there are an even number of triangles (odd number of tiles in $G$).
         \item[$(b)$] Suppose a triangulated polygon is oriented opposite of \Cref{fig:proof}, so that the top-most fan center is on the right side of the polygon.
                      Then $W(G)$ ends in ``R'' if and only if there are an even number of triangles (odd number of tiles in $G$), and
                      $W(G)$ ends in ``U'' if and only if there are an odd number of triangles (even number of tiles in $G$).
    \end {enumerate}
\end {lemma}
\begin {proof}
    In the case of a fan triangulation (i.e. $j$ is the only fan center), this is easy to see. By the construction outlined in
    \Cref{def:Ggamma}, the assumption that the fan center $j$ is on the left implies that $W(G)$ always begins with a ``U''.
    The result then follows for the case of a single fan segment. The argument for part $(b)$ in a single fan is the same,
    since $W(G)$ will always begin with ``R'' in that case.

    Now we will induct on the number of fan segments. Suppose parts $(a)$ and $(b)$ are true for all triangulations with $n-1$ fan segments.
    We will consider the case of $n$ fan segments. Let $G'$ be the smaller snake graph corresponding
    to the longest arc in the polygon containing only the first $n-1$ fan segments. Decompose $W(G)$ as $W(G') W'$, where $W'$ is the remainder of $W(G)$.
    Note that by the construction in \Cref{def:Ggamma}, the first letter of $W'$ is the same as the last letter of $W(G')$. This means if $W'$ has odd
    length, then $W(G)$ and $W(G')$ end in the same letter, and if $W'$ has even length, then $W(G)$ and $W(G')$ end in the opposite letter.

    Also note that if $W'$ has even length, then the parity of the numbers of tiles of $G$ and $G'$ are the same, while if $W'$ has odd length,
    the parity of the numbers of tiles in $G$ and $G'$ are opposite. Combining these observations with the induction assumption gives the result.
\end {proof}

\bigskip

\begin{figure}[h]
\centering
\begin{tikzpicture}
\tikzstyle{every path}=[draw] 
		\path
    node[
      regular polygon,
      regular polygon sides=13,
      draw=none,
      inner sep=1.8cm,
    ] (T) {}
    %
    (T.corner 1) node[above] {$l$}
    (T.corner 2) node[above] {$k$}
    (T.corner 3) node[left] {$j$}
    (T.corner 4) node[left] {}
    (T.corner 5) node[below] {}
    (T.corner 6) node[below] {}
    (T.corner 7) node[below] {}
    (T.corner 8) node[below] {$i$}
    (T.corner 9) node[right] {}
    (T.corner 10) node[right] {$m$}
    (T.corner 11) node[right] {}
    (T.corner 12) node[right ]{}
    (T.corner 13) node [right] {}
;
    \begin {scope} [rotate=15]
        \foreach \t [evaluate=\t as \s using int(\t-1)] in {1,...,13} {\coordinate (\s) at (T.corner \t);}
        \draw [-] (0)--(12);
        \foreach \t [evaluate=\t as \s using int(\t+1)] in {2,...,12}{\draw [-] (T.corner \t)--(T.corner \s);}

        \draw [-] (0)--node [above] {$h$} (12);

        \draw [-] (1)--node [above] {$a$} (2);
        \draw [-] (1)--node [above] {$b$} (0);
        \draw [dashed] (2) --node [left, inner sep=0.2em] {$d$} (7) -- node [left] {$c$} (0);
        \draw [-] (2)--node [midway] {$e$} (0);
        \draw [-](2)--node [midway] {$e_2$} (12);
        \draw [dotted](2)--(11);
        \draw[-](2)--node[midway] {$e_k$}(10);
        \draw [-] (2)--(9);
        \draw [dotted] (2)--(10);
        \draw (5)--(9)--(3);
        \draw [dotted] (9)--(4);
        \draw [dotted](5)--(8)--(6);
    \end {scope}
    \draw (0,-4) node {$(a)$};

    \begin {scope} [scale=0.7, shift={(8,0)}]
        \draw[dashed] (-2,0) -- (0,-2) -- (2,0);
        \draw         (2,0) -- (0,2) -- (-2,0);
        \draw[->-]    (-2,0) -- (2,0);

        \draw (45:1.77)   node {$b$};
        \draw (135:1.77)  node {$a$};
        \draw (-135:1.77) node {$d$};
        \draw (-45:1.77)  node {$c$};
        \draw (0,0)       node[above] {$e$};

        \draw (0,1) node {$\theta$};
        \draw (0,-1) node {$\sigma$};

        \begin {scope} [shift={(6,0)}]
            \draw[dashed] (-2,0) -- (0,-2) -- (2,0);
            \draw         (2,0) -- (0,2) -- (-2,0);
            \draw[->-]    (0,-2) -- (0,2);

            \draw (45:1.77)   node {$b$};
            \draw (135:1.77)  node {$a$};
            \draw (-135:1.77) node {$d$};
            \draw (-45:1.77)  node {$c$};
            \draw (0,0)       node[right] {$f$};

            \draw (-1,0) node {$\varphi$};
        \end {scope}
    \end {scope}
    \draw (8,-4) node {$(b)$};
\end {tikzpicture}
\caption{$(a)$ Triangulation $T$ with $n$ triangles. $(b)$ Ptolemy relation in the quadrilateral $ijkl$.}
\label{fig:proof}
\end{figure}
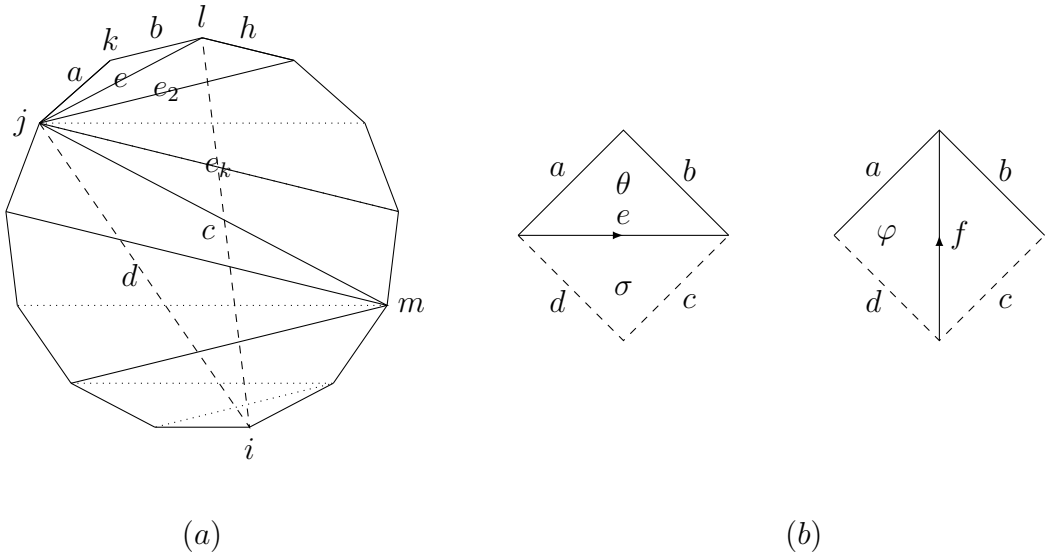
\begin {theorem} \label{thm:main}
    Consider a triangulation as pictured in \Cref{fig:proof} where $f=(i,k)$ is the longest arc, with corresponding snake graph $G$.
    In particular, we assume that $a$ and $b$ are boundary edges, the top-most fan center $j$ is on the left. Then
    \begin {enumerate}
        \item[$(a)$] $\displaystyle f = \frac{1}{\mathrm{cross}(f)} \sum_{M \in D(G)} \mathrm{wt}(M)$
        \item[$(b)$] If the polygon has an odd number of triangles\footnote{By \Cref{lem:parity}, this implies that $W(G)$ ends with ``R''.}, then 
                     \[ \sqrt{df} \; \varphi = \frac{1}{\mathrm{cross}(f)} \; \sqrt{\frac{e}{b}} \sum_{M \in D_t(G)} \mathrm{wt}(M)^* \]
                     where $x \mapsto x^*$ is the involution toggling $\theta_n$ (the top-most triangle). If there are an even number of triangles, then
                     the sum is over {$D_r(G)$} instead of {$D_t(G)$}\footnote{This is because, by \Cref{lem:parity}, $W(G)$ ends with ``U'' in this case.}.
    \end {enumerate}
\end {theorem}
\begin {proof}
    We will prove this by induction on the number of triangles.

    \paragraph{\textbf{Base case} }
    
    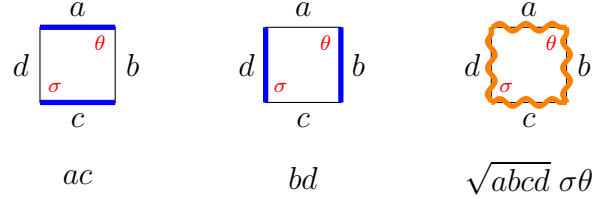
\begin {figure}[h]
    \centering
    \begin {tikzpicture}
        \tikzset {decoration={snake, amplitude=0.5mm}}

        \draw (0,0) -- (1,0) -- (1,1) -- (0,1) -- cycle;
        \draw (0.5,0) node[below] {$c$};
        \draw (0,0.5) node[left]  {$d$};
        \draw (0.5,1) node[above] {$a$};
        \draw (1,0.5) node[right] {$b$};

        \draw (0.2,0.2) node {\color{red} \tiny $\sigma$};
        \draw (0.8,0.8) node {\color{red} \tiny $\theta$};

        \draw[blue, line width = 2] (0,0) -- (1,0);
        \draw[blue, line width = 2] (0,1) -- (1,1);

        \draw (0.5,-1) node {$ac$};

        \begin {scope}[shift={(3,0)}]
            \draw (0,0) -- (1,0) -- (1,1) -- (0,1) -- cycle;
            \draw (0.5,0) node[below] {$c$};
            \draw (0,0.5) node[left]  {$d$};
            \draw (0.5,1) node[above] {$a$};
            \draw (1,0.5) node[right] {$b$};

            \draw (0.2,0.2) node {\color{red} \tiny $\sigma$};
            \draw (0.8,0.8) node {\color{red} \tiny $\theta$};

            \draw[blue, line width = 2] (0,0) -- (0,1);
            \draw[blue, line width = 2] (1,0) -- (1,1);

            \draw (0.5,-1) node {$bd$};
        \end {scope}

        \begin {scope}[shift={(6,0)}]
            \draw (0,0) -- (1,0) -- (1,1) -- (0,1) -- cycle;
            \draw (0.5,0) node[below] {$c$};
            \draw (0,0.5) node[left]  {$d$};
            \draw (0.5,1) node[above] {$a$};
            \draw (1,0.5) node[right] {$b$};

            \draw (0.2,0.2) node {\color{red} \tiny $\sigma$};
            \draw (0.8,0.8) node {\color{red} \tiny $\theta$};

            \draw[orange, line width = 2, decorate] (0,0) -- (0,1);
            \draw[orange, line width = 2, decorate] (0,0) -- (1,0);
            \draw[orange, line width = 2, decorate] (1,0) -- (1,1);
            \draw[orange, line width = 2, decorate] (0,1) -- (1,1);

            \draw (0.5,-1) node {$\sqrt{abcd} \; \sigma \theta$};
        \end {scope}
    \end {tikzpicture}
    \caption {The three double dimer covers on a single-tile snake graph}
    \label{fig:thm_main_base_case}
    \end {figure}
    
    The base case is a quadrilateral, in which case $W(G)$ is the empty word. Assume the quadrilateral is labelled as
    in \Cref{fig:super_ptolemy}. Then the snake graph $G$ (corresponding to the diagonal $f$) is just a single tile. It is easy to see that
    there are only three double dimer covers  
    on the square graph. These are pictured in \Cref{fig:thm_main_base_case}, along with their weights.
    There is only one diagonal in this triangulation (the edge $e$), and so $\mathrm{cross}(f) = e$. We see then that the weights of these
    three double dimer covers  
    (after being divided by $e$) are precisely the three terms in the super Ptolemy relation (\Cref{eqn:super_ptolemy_lambda}). 
    This confirms part $(a)$ in the base case.

    Next, we consider part $(b)$ for the base case. We are interested in what is called $\theta'$ in the super Ptolemy relation (\Cref{eqn:super_ptolemy_mu_left}).
    Recall that $D_r(G) = D_R(G) \cup D_{tr}(G)$. In the case of a single tile, $D_r(G)$ consists of the second and third double dimer covers 
    in \Cref{fig:thm_main_base_case}. In this case the $x^*$ involution corresponds to $\theta$, and we have
    \[ (bd)^* = bd \; \theta \quad \text{ and } \quad (\sqrt{abcd} \; \sigma \theta)^* = \sqrt{abcd} \; \sigma \]
    As noted above, $\mathrm{cross}(f) = e$, and so the right-hand side of part $(b)$ in the case of a single tile is
    \[ \frac{1}{\mathrm{cross}(f)} \; \sqrt{\frac{e}{b}} \sum_{M \in D_r(G)} \mathrm{wt}(M)^* = \frac{1}{\sqrt{be}} \left( bd \; \theta + \sqrt{abcd} \; \sigma \right) \]
    On the other hand, the left-hand side of part $(b)$ (using the notations of \Cref{fig:super_ptolemy}) is
    \begin {align*} 
        \sqrt{df} \; \theta' &= \frac{\sqrt{df}}{\sqrt{ef}} \left( \sqrt{bd} \; \theta + \sqrt{ac} \; \sigma \right) \\
                             &= \frac{1}{\sqrt{e}} \left( d \sqrt{b} \; \theta + \sqrt{acd} \; \sigma \right) \\
                             &= \frac{1}{\sqrt{be}} \left( bd \; \theta + \sqrt{abcd} \; \sigma \right) \\
    \end {align*}
    This confirms the base case for part $(b)$.
    
    \bigskip

    \paragraph{\textbf{Induction Assumptions}}

    Now we assume the formulas in parts $(a)$ and $(b)$ for triangulations with fewer than $n$ triangles, for some $n>2$.

    As mentioned above, we assume the triangulation to be pictured as in \Cref{fig:proof}, so that the top-most fan center $j$ is on the left.
    The argument for the other case (that the top-most fan center is on the right) is similar: we would need to use part $(b)$ of \Cref{lem:parity}
    rather than part $(a)$, and the roles of the edges labelled $c$ and $d$ in \Cref{fig:proof} would be swapped.

    We will also assume that the polygon has an odd number of triangles (equivalently, the snake graph has an even number of tiles). By \Cref{lem:parity}, this implies that the snake graph ends with ``R''.
    Because $G$ has an even number of tiles, this means the last tile of $G$ has orientation opposite of the polygon. These assumptions imply that the end of the snake graph looks as follows:
    
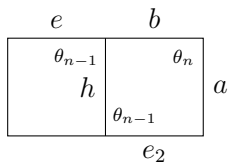
\begin {center}
    \begin {tikzpicture}[scale=1.3, every node/.style={scale=0.9}]
        \draw (0,0) -- (2,0) -- (2,1) -- (0,1) -- cycle;
        \draw (1,0) -- (1,1);

        \draw (1,0.5) node[left]  {$h$};
        \draw (1.5,1) node[above] {$b$};
        \draw (2,0.5) node[right] {$a$};
        \draw (1.5,0) node[below] {$e_2$};
        \draw (0.5,1) node[above] {$e$};

        \draw (1.8,0.8) node {\tiny $\theta_n$};
        \draw (1.3,0.2) node {\tiny $\theta_{n-1}$};
        \draw (0.7,0.8) node {\tiny $\theta_{n-1}$};
    \end {tikzpicture}
    \end {center}
    On the other hand, if the triangulation has an even number of tiles (the snake graph has an odd number of tiles), then the snake graph ends with ``U'' by \Cref{lem:parity},
    and the end of the snake graph would instead look as follows:
    
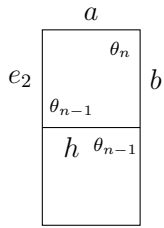
\begin {center}
    \begin {tikzpicture}[scale=1.3, every node/.style={scale=0.9}]
        \draw (0,0) -- (1,0) -- (1,2) -- (0,2) -- cycle;
        \draw (0,1) -- (1,1);

        \draw (0.5,2) node[above] {$a$};
        \draw (1,1.5) node[right] {$b$};
        \draw (0.3,1) node[below] {$h$};
        \draw (0,1.5) node[left]  {$e_2$};

        \draw (0.8,1.8) node {\tiny $\theta_n$};
        \draw (0.3,1.2) node {\tiny $\theta_{n-1}$};
        \draw (0.75,0.8) node {\tiny $\theta_{n-1}$};
    \end {tikzpicture}
    \end {center}
    We will assume the former case for the remainder of the proof. The argument for the latter case is similar: occurrences of $D_R(G)$ and $D_r(G)$ would
    need to be replaced by $D_T(G)$ and $D_t(G)$.

    \bigskip
    
    \paragraph{\textbf{Induction for part ${\bf (b)}$}}
    
    First we prove part $(b)$. Looking at \Cref{fig:proof},
    the super Ptolemy relation gives the following expression for $\varphi$:
    \[ \varphi = \frac{1}{\sqrt{ef}} \left(\sqrt{bd} \, \theta + \sqrt{ac} \, \sigma \right) \]
    Multiplying by $\sqrt{df}$ gives
    \[ \sqrt{df} \, \varphi = \sqrt{\frac{d}{e}} \left(\sqrt{bd} \, \theta + \sqrt{ac} \, \sigma \right) = d \, \sqrt{\frac{b}{e}} \, \theta + \sqrt{\frac{acd}{e}} \, \sigma \]
    Recall that {$D_t(G) = D_T(G) \cup D_{tr}(G)$}. We claim that the two terms on the right-hand side correspond to {$D_T(G)$} and $D_{tr}(G)$.
    Consider the first term, $d \, \sqrt{\frac{b}{e}} \, \theta$. The elements $b,e,\theta$ are all in the initial triangulation, but $d$ might not be.
    Since $d$ is a diagonal in a smaller triangulation, we may use induction and say that
    \[ d = \frac{1}{\mathrm{cross}(d)} \, \sum_{M \in D(G^{(-k)})} \mathrm{wt}(M), \]
    where $k$ is the number of diagonals in the top fan segment.

    Note that $G^{(-k)}$ is $G$ with the last maximal ``\emph{staircase}'' removed.
    Let $e_1,\dots,e_k$ be the last $k$ internal diagonals of the last fan segment (with $e_1 = e$). These appear on the snake graph in the
    positions indicated in \Cref{fig:lemma2_proof_NEW} (Left). Then $\mathrm{cross}(f) = e_1 \cdots e_k \cdot \mathrm{cross}(d)$.
    Substituting the expression for $d$ into the expression we had before, we get that the first term is
    \[ 
        d \, \sqrt{\frac{b}{e}} \, \theta = \frac{e_1 \cdots e_k}{\mathrm{cross}(f)} \, \sqrt{\frac{b}{e}} \, \sum_{M \in D(G^{(-k)})} \mathrm{wt}(M) \, \theta 
        = \frac{e_2 \cdots e_k}{\mathrm{cross}(f)} \, \sqrt{be} \sum_{M \in D(G^{(-k)})} \mathrm{wt}(M) \, \theta 
    \]
{Here we need to check that the multiplication puts $\theta$ at the correct place in the positive order. By the induction hypothesis, the $\mu$-invariants in the monomials $\wt(M)$ are already written in the positive order with respect to $T$. Depending on the orientation of edge $e$, the $\mu$-invariant $\theta$ can possibly be at either the beginning or end of the order (\Cref{def:pos_order}).  Because of the construction in \Cref{def:weight-M}, the monomial $\mathrm{wt}(M)$ contains an even number of odd variables for any choice of $M$.  Thus $\theta$ commutes with $\wt(M)$, and the resulting product can always be written in the correct order without changing its sign.}
   
Now we use \Cref{lem:lemma2} to conclude that
    {
    \[ \sum_{M \in D(G^{(-k)})} \mathrm{wt}(M) = \frac{1}{be_2 \cdots e_k} \sum_{M \in D_T(G)} \mathrm{wt}(M) \]
    }
    
    Making this substitution gives 
    {
    \[ d \sqrt{\frac{b}{e}} \, \theta = \frac{1}{\mathrm{cross}(f)} \, \sqrt{\frac{e}{b}} \sum_{M \in D_T(G)} \, \mathrm{wt}(M) \, \theta \]
    }
    
    Note that since elements of {$D_T(G)$} cannot have a cycle around the last tile, $\mathrm{wt}(M) \, \theta = \mathrm{wt}(M)^*$. 
    So we get that the first term is the contribution from {$D_T(G)$}, as claimed. 

    Next, we examine the second term, $\sqrt{\frac{acd}{e}} \, \sigma$. 
    Recall that we assume $c$ is the second-longest edge (crosses all but the last diagonal), as in \Cref{fig:proof}.
    Also let $h$ be the boundary side of the polygon adjacent to the endpoint of $c$, as in the figure.
    Note that $\sqrt{cd} \, \sigma$ is precisely the left-hand side of the expression in part $(b)$ for the triangle $\sigma$ (in the role of $\varphi$).
    Let $x \mapsto x^\dagger$ denote the toggle involution for $\theta_{n-1}$. Then by induction, since $\sigma$ is in a smaller polygon, 
    we will assume the formula from part $(b)$. 
    {
        Since there is one less triangle in this smaller polygon, the sum is over $D_r$ rather than $D_t$.
    }
    So we get
    \begin {align*} 
        \sqrt{\frac{acd}{e}} \, \sigma &= \sqrt{\frac{a}{e}} \, \frac{1}{\mathrm{cross}(c)} \, \sqrt{\frac{e_2}{h}} \sum_{M \in D_r(G^{(-1)})} \mathrm{wt}(M)^\dagger \\
                                        &= \frac{\sqrt{ae}}{\mathrm{cross}(f)} \, \sqrt{\frac{e_2}{h}} \sum_{M \in D_r(G^{(-1)})} \mathrm{wt}(M)^\dagger
    \end {align*}
    {where the products are already in the correct positive order by induction.}
    Finally, we use \Cref{lem:lemma4} to substitute the summation over $D_r(G^{(-1)})$ for a summation over $D_{tr}(G)$:
    \[ \sum_{M \in D_r(G^{(-1)})} \mathrm{wt}(M)^\dagger = \sum_{M \in D_{tr}(G)} \sqrt{\frac{h}{abe_2}} \mathrm{wt}(M)^* \]
    Making this substitution cancels the $\sqrt{\frac{ae_2}{h}}$ factors, and gives the result:
    \[ \sqrt{\frac{acd}{e}} \, \sigma = \frac{1}{\mathrm{cross}(f)} \, \sqrt{\frac{e}{b}} \sum_{M \in D_{tr}(G)} \mathrm{wt}(M)^* \]

    \paragraph{\textbf{Induction for part ${\bf (a)}$}} 
       
    Consider the Ptolemy relation \Cref{eqn:super_ptolemy_lambda} on the quadrilateral $(i,j,k,l)$:
    \begin{equation} \label{eq:proof_flip}
        f = {ac \over e} + {bd \over e} + {\sqrt{abcd} \over e} \sigma \theta
    \end{equation}
   
    First we examine the first term in \Cref{eq:proof_flip}: $ac/e$. 
    Since $c$ is the longest arc in the smaller polygon with one less triangle, we have by induction that
    \begin{equation}\label{eq:c}
        c = \frac{1}{\crs(c)} \sum_{M\in D(G^{(-1)})} \wt(M)
    \end{equation}
    where $G^{(-1)}$ is the snake graph for $c$, which is obtained by removing the last tile from $G$. 
    Now multiply \Cref{eq:c} by $a/e$, noting that $\crs(f)=e\cdot\crs(c)$, to get
    \[{ac\over e}={1\over \crs(f)}\sum_{M\in D(G^{(-1)})}\wt(M)a\]
  	
    By \Cref{lem:lemma1}, double dimer covers 
     in $G^{(-1)}$ are exactly those in $D$ which have two dimers on edge $a$ of the last tile (see 
    \Cref{fig:lemma2_proof_NEW}). 
    Hence we have:
    \begin{equation}
    \label{eq:ac_e}{ac\over e}=\frac{1}{\crs(f)}\sum_{M\in D_R(G)}\wt(M)
    \end{equation}
  	
    Next we examine the term $bd/e$ in \Cref{eq:proof_flip}.
  
    Recall that by induction, $d$ is given by
    \begin{equation}\label{eq:d}
        d={1\over \crs(d)}\sum_{M\in D(G^{(-k)})}\wt(M)
    \end{equation}
    Multiply \Cref{eq:d} by $b/e$ to get
    \[{bd\over e}={1\over \crs(d)}\cdot{b\over e} \sum_{M\in D(G^{(-k)})}\wt(M)\]
    Since $\crs(f)=\crs(d)ee_2e_3\cdots e_k$, we have
    \[{bd\over e}={1\over \crs(f)}\sum_{M\in D(G^{(-k)})}\wt(M)be_2e_3\cdots e_k\]
    Now by \Cref{lem:lemma2}, double dimer covers  
    in $D(G^{(-k)})$ are in bijection with those in $D_T(G)$ up to multiplying by the weight 
    of edges in the first $k$ tiles $ee_2e_3\cdots e_k$ (see \Cref{fig:lemma2_proof_NEW}). Hence we have:
    \begin{equation} \label{eq:bd_e}
        {bd \over e} = {1\over \crs(f)} \sum_{M\in D_T(G)}\wt(M)
    \end{equation}
    
    {Note that in \Cref{eq:ac_e,eq:bd_e}, the products of $\mu$-invariants are taken under the positive order by induction.}

    Lastly we examine the term ${\sqrt{abcd}\over e}\sigma\theta$ in \Cref{eq:proof_flip}. By induction hypothesis of part (b) we have
    \begin{equation} \label{eq:mu_induction}
        \sqrt{cd}\sigma={1\over\crs(c)}\sqrt{e_2\over h}\sum_{M\in D_r (G^{(-1)})}\wt(M)^\dagger 
    \end{equation}
    {Here $G^{(-1)}$ has an even number of tiles, therefore the sum is over $D_r(G^{(-1))}$ rather than $D_t(G^{(-1)})$.}
    Right multiply \Cref{eq:mu_induction} by ${\sqrt{ab}\over e}\theta$ we get
    \begin{align*}
        {\sqrt{abcd}\over e}\sigma \theta &={1\over\crs(c)}\cdot{\sqrt{ab}\over e}\cdot\sqrt{e_2\over h}\sum_{M\in D_r (G^{(-1)})}\wt(M)^\dagger\theta \\
                                          &={1\over\crs(f)}\sqrt{abe_2\over h}\sum_{M\in D_r (G^{(-1)})}\wt(M)^\dagger\theta
    \end{align*}
    
    Since we assumed that the direction of edge $e$ is from left to right as depicted in \Cref{fig:proof}, $\theta$ will appear at the end of the positive order of $T$, hence right-multiplication of $\theta$ will result in the correct order. In the other situation where the edge $e$ goes from right to left, the product in \Cref{eq:proof_flip} would be $\theta\sigma$ instead, giving the positive order in that case.

    Note that for $M\in D_r (G^{(-1)})$, we have $\wt(M)^\dagger\theta=(\wt(M)^\dagger)^*$. Hence applying the bijection in \Cref{lem:lemma4}, we get
    \begin{equation} \label{eq:mu_term}
        {\sqrt{abcd}\over e}\sigma\theta={1\over\crs(f)}\sum_{M\in D_{tr}( G)}\wt(M).
    \end{equation}

    Finally, combining \Cref{eq:ac_e,eq:bd_e,eq:mu_term} gives us
    \[ f = \frac{1}{\mathrm{cross}(f)} \sum_{M \in D(G)} \mathrm{wt}(M),\]
    which completes the proof.
\end {proof}

\bigskip

\section{Relationship Between Different Combinatorial Models for Super $\lambda$-Lengths}

In this section, we relate the combinatorial interpretation of super $\lambda$-lengths given by double dimer covers of snake graphs, 
as in Theorem \ref{thm:main} (a), to our previous combinatorial interpretation given in \cite{moz21} by super $T$-paths.
As a by-product of our efforts towards this comparison, we introduce an additional family of combinatorial objects, which we call \emph{twisted} super $T$-paths, 
that yield the same elements of the super algebra $A = \Bbb{R}[x_{ij}^{\pm 1/2} ~|~ \theta_k]$.

\subsection{Twisted Super $T$-Paths}

Recall from Section 4 of \cite{moz21} that we distinguished certain vertices of a triangulation $T$ to be fan centers and consequently defined an auxiliary graph from this data.
To match our notation to that of \cite{moz21}, we fix a choice of arc $(a,b)$, noting this is the arc whose super $\lambda$-length we wish to compute, i.e. $\lambda_{ab}$.  
Then recall by restricting to a sub-triangulation $T(a,b) \subset T$ if necessary, we assume $(a,b)$ is the longest arc in $T(a,b)$.
	
For a triangulation $T$ and a pair of vertices $a$ and $b$, the auxiliary graph is the graph of the triangulation $T$ with some additional vertices and edges. 
\begin{enumerate}[]
    \item For each face of the triangulation $T$, we place an internal vertex, which lies on the arc $(a,b)$. 
          We denote the internal vertices using $\theta_1,\cdots,\theta_{n+1}$, such that $\theta_i$ is closer to $a$ than $\theta_j$ if and only if $i<j$. 
    
    \item For each face of $T$, we add an edge $\sigma_i := (\theta_i, c_j)$ connecting the internal vertex $\theta_i$ to the center of the fan segment which contains $\theta_i$. 
          We denote by $\sigma$ the set of all such edges.
    
    \item For each $\theta_i$ and $\theta_j$ with $i<j$, we add an edge connecting $\theta_i$ and $\theta_j$. 
          We denote the collection of these edges as $\tau = \{\tau_{ij}:i<j\}$. For simplicity the $\tau$-edges are drawn to be overlapping.
\end{enumerate}

See \Cref{fig:auxiliary_graph} for an example.  Note that we continue to follow the convention of \Cref{def:Ggamma} and assume that the 
longest edge $(a,b)$ traverses the triangulation $T$ from bottom-to-top.

\begin{figure}[h]
	\centering
	\begin{tikzpicture}[]
	\tikzstyle{every path}=[draw] 
		\path
    node[
      regular polygon,
      regular polygon sides=6,
      draw,
      inner sep=1.6cm,
    ] (hexagon) {}
    %
    (hexagon.corner 1) node[above] {$\ \ b$}
    (hexagon.corner 2) node[above] {$y\ \ $}
    (hexagon.corner 3) node[left ] {$c_1$}
    (hexagon.corner 4) node[below] {$a\ \ $}
    (hexagon.corner 5) node[below] {$\ \ x$}
    (hexagon.corner 6) node[right] {$c_2$}

  ;
  \coordinate (m1) at (1,1.62);
  \coordinate (m2) at (0.1,0.8);
  \coordinate (m3) at (-0.1,-0.8);
  \coordinate (m4) at (-1,-1.62);
  
  \draw [name path = m ,opacity=1] (m1) to [out=-180+50,in=60] (m2) to [out=-180+60,in=60] (m3) to [out=-180+60,in=50] (m4);%
  
  \draw [name path = L1] (hexagon.corner 6) to (hexagon.corner 2);
  \draw [name path = L2] (hexagon.corner 6) to (hexagon.corner 3);
  \draw [name path = L3] (hexagon.corner 5) to (hexagon.corner 3);

  \draw (m1) node [fill,circle,scale=0.3] {};
  \draw (m2) node [fill,circle,scale=0.3] {};
  \draw (m3) node [fill,circle,scale=0.3] {};
  \draw (m4) node [fill,circle,scale=0.3] {};
  
  \node[left] at (m1) {$\theta_4$};
  \node[left] at (m2) {$\theta_3$};
  \node[right] at (m3) {$\theta_2$}; 
  \node[left] at (m4) {$\theta_1$};
  
  \node[right] at (1.2,1.4) {$\sigma_4$};
  \node[right] at (0.4,0.4) {$\sigma_3$};
  \node[right] at (-1,-0.4) {$\sigma_2$};
  \node[right] at (-1.7,-1) {$\sigma_1$};

  \draw (hexagon.corner 1) node [fill,circle,scale=0.35] {};
  \draw (hexagon.corner 2) node [fill,circle,scale=0.35] {};
  \draw (hexagon.corner 3) node [fill,circle,scale=0.35] {};
  \draw (hexagon.corner 4) node [fill,circle,scale=0.35] {};
  \draw (hexagon.corner 5) node [fill,circle,scale=0.35] {};
  \draw (hexagon.corner 6) node [fill,circle,scale=0.35] {};
  
  \draw [] (m1)--(hexagon.corner 6);
  \draw [] (m2)-- (hexagon.corner 6);
  \draw [] (m3)--(hexagon.corner 3);
  \draw [] (m4)-- (hexagon.corner 3);

	\end{tikzpicture}
		\caption{The auxiliary graph as defined in \cite{moz21}.}
		\label{fig:auxiliary_graph}
	\end{figure}
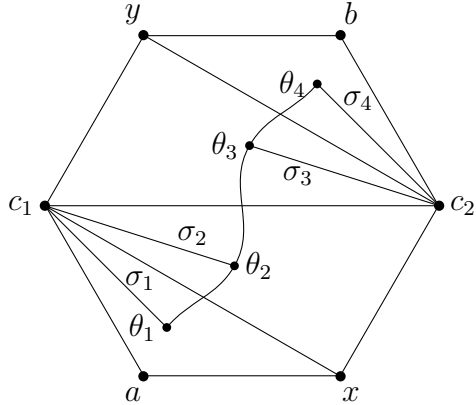

We now define a twisted auxiliary graph $\Gamma_T^{a,b}$ that we use to define twisted super $T$-paths.  
For our alternative twisted definition we introduce new $\sigma$-edges by using the complements of the fan centers.  
More precisely, given a triangulation $T$ admitting longest edge $(a,b)$ as above, for every triangle we associate two $\sigma$-edges 
(rather than one) and note that one of these two $\sigma$-edges is closer to the starting vertex $a$ while the other is closer to the ending vertex $b$.  
We color the former in  {\bf\textcolor{blue}{thick blue}} and call it a $\sigma^A$-edge and color the latter in {\textcolor{cyan}{cyan} 
and call it a $\sigma^B$-edge.  For each $\theta_i$ and $\theta_j$ with $i<j$, $\Gamma_{T}^{a,b}$ contains the edge $\tau_{ij}:i<j$ just as before.
See \Cref{fig:auxiliary_graph2} for an example.

Just as super $T$-paths follow edges of the auxiliary graph, as defined in \cite{moz21}, we will define twisted super $T$-paths to 
follow edges of this newly defined twisted auxiliary graph $\Gamma_T^{a,b}$.

\begin{figure}[h]
	\centering
	\begin{tikzpicture}[]
	\tikzstyle{every path}=[draw] 
		\path
    node[
      regular polygon,
      regular polygon sides=6,
      draw,
      inner sep=1.6cm,
    ] (hexagon) {}
    %
    (hexagon.corner 1) node[above] {$\ \ b$}
    (hexagon.corner 2) node[above] {$y\ \ $}
    (hexagon.corner 3) node[left ] {$c_1$}
    (hexagon.corner 4) node[below] {$a\ \ $}
    (hexagon.corner 5) node[below] {$\ \ x$}
    (hexagon.corner 6) node[right] {$c_2$}

  ;
  \coordinate (m1) at (1,1.62);
  \coordinate (m2) at (0.1,0.8);
  \coordinate (m3) at (-0.1,-0.8);
  \coordinate (m4) at (-1,-1.62);
  
  \draw [name path = m ,opacity=1] (m1) to [out=-180+50,in=60] (m2) to [out=-180+60,in=60] (m3) to [out=-180+60,in=50] (m4);%
  
  \draw [name path = L1] (hexagon.corner 6) to (hexagon.corner 2);
  \draw [name path = L2] (hexagon.corner 6) to (hexagon.corner 3);
  \draw [name path = L3] (hexagon.corner 5) to (hexagon.corner 3);

  \draw (m1) node [fill,circle,scale=0.3] {};
  \draw (m2) node [fill,circle,scale=0.3] {};
  \draw (m3) node [fill,circle,scale=0.3] {};
  \draw (m4) node [fill,circle,scale=0.3] {};
  
  \node[right] at (1.2,1.4) {$\theta_4$};
  \node[right] at (0.4,0.4)  {$\theta_3$};
  \node[right] at (-1,-0.4)  {$\theta_2$}; 
  \node[left] at (-1,-1.22) {$\theta_1$};
  
  \node[left] at (1.1,1.4) {$\sigma^A_4$};
  \node[left] at (1.3,2.0) {$\sigma^B_4$};
  \node[right] at (-1.0, 0.4) {$\sigma^A_3$};
   \node[right] at (-1.2, 1.2) {$\sigma^B_3$}; 
  \node[right] at (0.3, -1.2) {$\sigma^A_2$};
  \node[right] at (0.2,-0.3) {$\sigma^B_2$};  
  \node[right] at (-1.1,-2.0) {$\sigma^A_1$};
  \node[right] at (-0.8,-1.5) {$\sigma^B_1$};

  \draw (hexagon.corner 1) node [fill,circle,scale=0.35] {};
  \draw (hexagon.corner 2) node [fill,circle,scale=0.35] {};
  \draw (hexagon.corner 3) node [fill,circle,scale=0.35] {};
  \draw (hexagon.corner 4) node [fill,circle,scale=0.35] {};
  \draw (hexagon.corner 5) node [fill,circle,scale=0.35] {};
  \draw (hexagon.corner 6) node [fill,circle,scale=0.35] {};
  
  \draw [cyan,thick] (m1)--(hexagon.corner 1);
  \draw [blue,ultra thick] (m1)--(hexagon.corner 2);  

  \draw [cyan,thick] (m2)--(hexagon.corner 2);
  \draw [blue,ultra thick] (m2)--(hexagon.corner 3);  
  
  \draw [cyan,thick] (m3)--(hexagon.corner 6);
  \draw [blue,ultra thick] (m3)--(hexagon.corner 5);  
  
  \draw [cyan,thick] (m4)--(hexagon.corner 5);
  \draw [blue,ultra thick] (m4)--(hexagon.corner 4);  
  

	\end{tikzpicture}
		\caption{The twisted auxiliary graph $\Gamma_T^{a,b}$.}
		\label{fig:auxiliary_graph2}
	\end{figure}
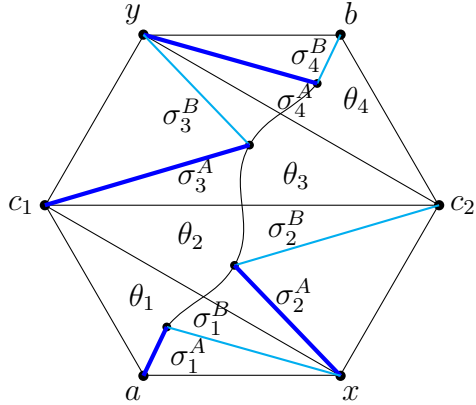

We define twisted super $T$-paths from $a$ to $b$  axiomatically as follows.

\begin{definition} \label{def:super-T} 
    A \emph{twisted super $T$-path} $t$ from $a$ to $b$ is a sequence
    \[t=(a_0,a_1,\cdots,a_{\ell(t)} | t_1,t_2,\cdots,t_{\ell(t)})\]
    such that
    \begin{enumerate}
        \item[(T1)] $a=a_0,a_1,\cdots,a_{\ell(t)}=b$ are vertices on $\Gamma_T^{a,b}$.
        \item[(T2)] For each $1\leq i\leq \ell(t)$, $t_i$ is an edge in $\Gamma_T^{a,b}$ connecting $a_{i-1}$ and $a_i$.
        \item[(T3)] $t_i\neq t_j$ if $i\neq j$.
        \item[(T4)] $\ell(t)$ is odd.
        \item[(T5')] $t_i$ crosses $(a,b)$ if $i$ is even. The $\tau$-edges are considered to cross $(a,b)$, and 
                     any step along a $\tau$-edge must end further from endpoint $a$ and closer to endpoint $b$.
        \item[(T6')] $t_i\in\sigma$ only if $i$ is odd, $t_i\in\tau$ only if $i$ is even.
        \item[(T7)] If $i<j$ and both $t_i$ and $t_j$ cross the arc $(a,b)$, then the intersection $t_i\cap (a,b)$ 
                    is closer to the vertex $a$ than the intersection $t_j\cap (a,b)$.
    \end{enumerate}

     We define $\mathcal T_{a,b}$ to be the set of twisted super $T$-paths from $a$ to $b$. 
\end{definition}

\begin{remark}
This is a slight variant of Definition 4.2 of \cite{moz21} with the only difference being that we replaced the previous axioms 
\begin{enumerate}
    \item[(T5)] $t_i$ crosses $(a,b)$ if $i$ is even. The $\sigma$-edges are considered to cross $(a,b)$.            
    \item[(T6)] $t_i\in\sigma$ only if $i$ is even, $t_i\in\tau$ only if $i$ is odd
\end{enumerate}
with the axioms (T5') and (T6'), thereby switching the parity of $\sigma$-steps and $\tau$-steps.
\end{remark}

We also define new weights for such twisted super $T$-paths.  (Note that the weights of $\sigma$-edges differ from 
Definition 4.8 of \cite{moz21} since $\sigma$-edges are now assumed to be odd steps rather than even steps.)

\begin{definition} \label{def:twisted_weight}
    Let $t \in \mathcal{T}_{ab}$ be a twisted super $T$-path which uses edges $t_1,t_2,\dots$ in the twisted auxiliary graph $\Gamma^T_{a,b}$.
    We will assign to each edge $t_i$ a twisted weight, which will be an element in the super algebra
    $\mathbb R[x_1^{\pm\frac{1}{2}},\cdots,x_{2n+3}^{\pm\frac{1}{2}} ~|~ \theta_1,\cdots,\theta_{n+1}]$ (where $\theta_i$'s are the odd generators) as follows.  
    For the parity of edges $t_i \in \sigma$ or $\tau$, we recall axiom (T6') of \Cref{def:super-T}.

    \[
        \twt(t_i) := \begin{cases}  
            x_j      &\text{ if } t_i \in T \text{, with }\lambda\text{-length }x_j \text{, and }i \text{ is odd} \\
            x_j^{-1} &\text{ if } t_i \in T \text{, with }\lambda\text{-length }x_j \text{, and }i \text{ is even} \\
            \sqrt{\frac{x_kx_l}{x_j}} \, \theta_s &\text{ if } t_i \in\sigma \text{ and the face containing }t_i \text{ is as pictured below} ~ (i \text{ must be odd}) \\
            1&\text{ if }t_i\in\tau~ (i \text{ must be even})
    \end{cases}  \]
    
    
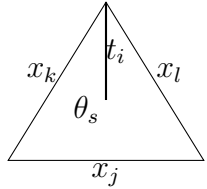
\begin{center}
        \begin{tikzpicture}[scale=1.3]
                \coordinate (1) at (1,0);
                \coordinate (2) at (-1,0);
                \coordinate (3) at (0,1.618);
                \coordinate (4) at (0, 0.618);
                \draw (1)--(2)--(3)--cycle;
                \draw [thick](3)--(4);
                \node at (-0.2, 0.518) {$\theta_s$};
                \node at (0.1,1.118) {$t_i$};
                \node at (0.65,0.9) {$x_l$};
                \node at (-0.65,0.9) {$x_k$};
                \node at (0,-0.15) {$x_j$};
        \end{tikzpicture}
    \end{center}
    Here, $\theta_s$ is the $\mu$-invariant associated to the face containing $t_i$.
	Finally, we define the (twisted) weight of a twisted super $T$-path to be the product of the (twisted) weights of its edges
    \[\twt(t)= \prod_{t_i\in t}\twt(t_i) \]
    where the product of $\mu$-invariants is taken under the positive order.
\end{definition}

We use this to arrive at an analogue of Theorem 4.9 of \cite{moz21}.

\begin{theorem}\label{thm:main_twisted} Under the default orientation,
    the $\lambda$-length of $(a,b)$ is given by
    $$\lambda_{a,b} = \sum_{t\in\mathcal T_{a,b}} \twt(t)$$
    where we use twisted super $T$-paths on the twisted auxiliary graph, and with twisted weights of each of the steps.
\end{theorem}

Before proving this theorem, we construct a weight-preserving bijection between super $T$-paths, as defined in \cite{moz21}, and twisted super $T$-paths, 
as introduced in \Cref{def:super-T}.  For this bijection, it will also be useful to define \emph{twisted super steps} 
in analogy to super steps in \cite{moz21}.

\begin{definition} \label{def:twisted_super_steps}
    A triple of steps along a twisted super $T$-path which consist of a $\sigma$-edge, followed by a $\tau$-edge, 
    and then a second $\sigma$-edge will be called a \emph{twisted super step}.
\end{definition}

\begin{lemma} \label{lem:twisted_bijection}
    Given a triangulation $T$ of a polygon such that $(a,b)$ is its unique longest arc, there is a weight-preserving bijection $\psi$ 
    that maps every super $T$-path $t = (t_1,t_2,\dots, t_{\ell(t)})$ to a \emph{twisted} super $T$-path
    $\psi(t) = (t_1',t_2',\dots, t_{\ell(t)'}')$ through the auxiliary graph $\Gamma_T^{a,b}$.
    Furthermore, for each super $T$-path $t$, the weight of $t$, i.e. $wt(t) = \prod_{i=1}^{\ell(t)} wt(t_i)$, and the twisted 
    weight of $\psi(t) = t'$, i.e.  $twt(t') = \prod_{i=1}^{\ell(t')} twt(t_i')$, coincide.
\end{lemma}

\begin{proof}
    We construct a bijection between super $T$-paths and twisted super $T$-paths via the following:

    Recall that a non-ordinary super $T$-path, as defined in \cite{moz21}, involves super steps, which consists of the three-step combination of 
    $\sigma_i \tau_{ij} \sigma_j$ where $\sigma_i$ is the incoming edge to the face center corresponding to $\theta_i$, $\tau_{ij}$ is the teleportation 
    between the face centers for $\theta_i$ and $\theta_j$, and
    $\sigma_j$ is the outgoing edge to the face center corresponding to $\theta_j$.  Further, we assume that $\sigma_i$ and $\sigma_j$ 
    are even steps while $\tau_{ij}$ is an odd step.

    In our bijection, we either replace two out of three steps of a super step with either a single step or a three step combination.  
    The consequence of this map switches the parity of $\sigma$-steps as needed while adding (or deleting) an ordinary step.  
    The parity of the $\tau_{ij}$ steps is similarly switched from odd to even.  These switches are also local in the sense that the 
    remainder of the super $T$-paths may be left alone.  Locally, there are four possible cases, and they are each illustrated 
    in \Cref{fig:old_new_bijection}.  Traversing the twisted auxiliary graph $\Gamma_T^{a,b}$ from bottom-to-top, in cases (i) and (ii), 
    we illustrate the two ways a super $T$-path (resp. super $T$-path) may enter the triangle with face center $\theta_i$ followed by a $\tau_{ij}$-step.  
    (Blue or cyan steps correspond to odd-indexed steps while red steps correspond to even-indexed steps.)  Cases (iii) and (iv) illustrate 
    the analogous options for exiting the triangle with face center $\theta_j$ proceeding after a $\tau_{ij}$-step.

    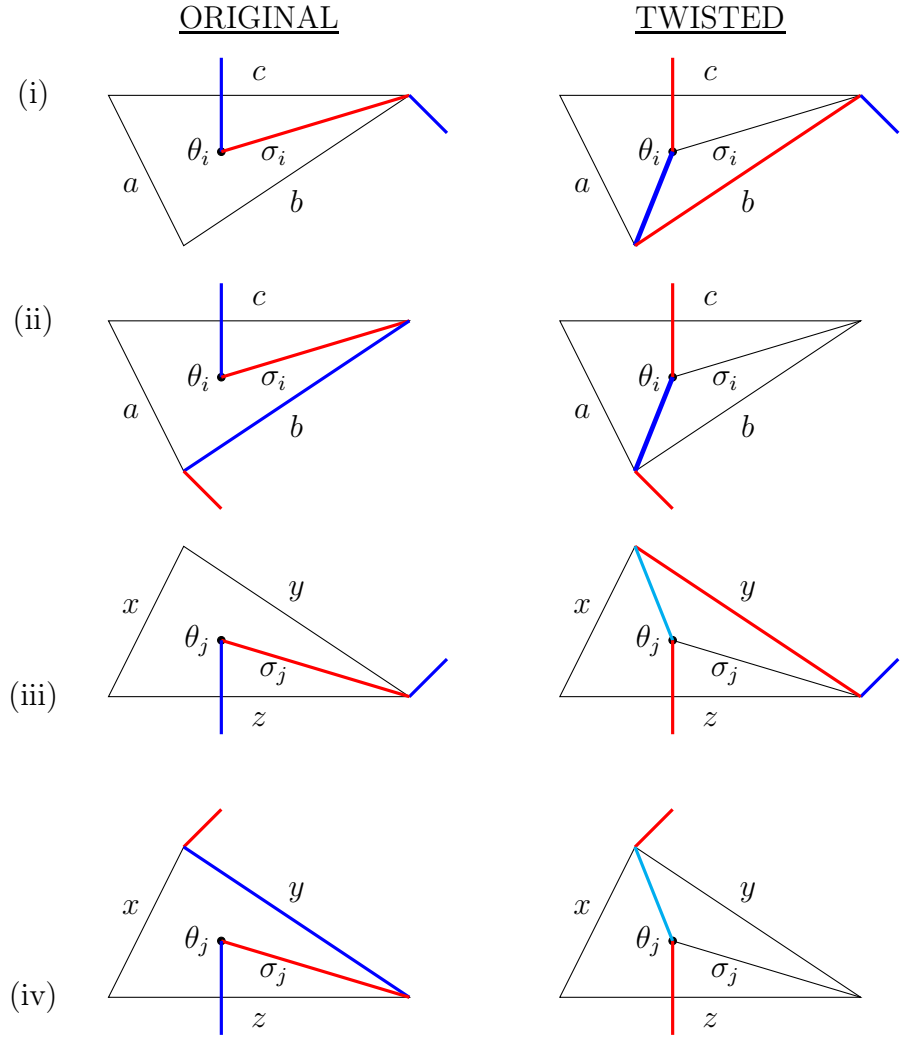
\begin {figure}[h]
    \centering
    \begin {tikzpicture}
    \draw (2,3) node {\underline{ORIGINAL}};
    \draw (8,3) node {\underline{TWISTED}};

    \begin {scope} [shift={(0,-6)}]
    	\draw (-1,-0) node {(iii)};
    	\draw (0,0) -- (4,0) -- (1,2) -- cycle;
    	\draw (4,0) -- (1.5,0.75);
    	\draw[fill=black] (1.5,0.75) circle (0.05);

    	\draw (0.3,1.2) node {$x$};
    	\draw (2.5,1.4) node {$y$};
    	\draw (2, -0.3) node {$z$};
    	\draw (2.2,0.3) node {$\sigma_j$};
    	\draw (1.5,0.75) node[left] {$\theta_j$};

    	\draw [blue, line width = 1.2] (4.5,0.5) -- (4,0);
    	\draw [red,  line width = 1.2] (4,0) -- (1.5,0.75);
    	\draw [blue, line width = 1.2] (1.5,0.75) -- (1.5,-0.5);
    \end {scope}

    \begin {scope} [shift={(6,-6)}]
        \draw (0,0) -- (4,0) -- (1,2) -- cycle;
        \draw (4,0) -- (1.5,0.75);
        \draw[fill=black] (1.5,0.75) circle (0.05);

        \draw (0.3,1.2) node {$x$};
        \draw (2.5,1.4) node {$y$};
        \draw (2, -0.3) node {$z$};
        \draw (2.2,0.3) node {$\sigma_j$};
        \draw (1.5,0.75) node[left] {$\theta_j$};

        \draw [blue, line width = 1.2] (4.5,0.5) -- (4,0);
        \draw [red,  line width = 1.2] (4,0) -- (1,2);
        \draw [cyan, line width = 1.2] (1,2) -- (1.5,0.75);
        \draw [red,  line width = 1.2] (1.5,0.75) -- (1.5,-0.5);
    \end {scope}

    \begin {scope} [shift={(0,-10)}]
    	\draw (-1,-0) node {(iv)};
        \draw (0,0) -- (4,0) -- (1,2) -- cycle;
        \draw (4,0) -- (1.5,0.75);
        \draw[fill=black] (1.5,0.75) circle (0.05);

        \draw (0.3,1.2) node {$x$};
        \draw (2.5,1.4) node {$y$};
        \draw (2, -0.3) node {$z$};
        \draw (2.2,0.3) node {$\sigma_j$};
        \draw (1.5,0.75) node[left] {$\theta_j$};

        \draw [red,  line width = 1.2] (1.5,2.5) -- (1,2);
        \draw [blue, line width = 1.2] (1,2) -- (4,0);
        \draw [red,  line width = 1.2] (4,0) -- (1.5,0.75);
        \draw [blue, line width = 1.2] (1.5,0.75) -- (1.5,-0.5);
    \end {scope}

    \begin {scope} [shift={(6,-10)}]
        \draw (0,0) -- (4,0) -- (1,2) -- cycle;
        \draw (4,0) -- (1.5,0.75);
        \draw[fill=black] (1.5,0.75) circle (0.05);

        \draw (0.3,1.2) node {$x$};
        \draw (2.5,1.4) node {$y$};
        \draw (2, -0.3) node {$z$};
        \draw (2.2,0.3) node {$\sigma_j$};
        \draw (1.5,0.75) node[left] {$\theta_j$};

        \draw [red,  line width = 1.2] (1.5,2.5) -- (1,2);
        \draw [cyan, line width = 1.2] (1,2) -- (1.5,0.75);
        \draw [red,  line width = 1.2] (1.5,0.75) -- (1.5,-0.5);
    \end {scope}

    \begin {scope} [shift={(0,2)}]
        \draw (-1,-0) node {(i)};
        \draw (0,0) -- (4,0) -- (1,-2) -- cycle;
        \draw (4,0) -- (1.5,-0.75);
        \draw[fill=black] (1.5,-0.75) circle (0.05);

        \draw (0.3,-1.2) node {$a$};
        \draw (2.5,-1.4) node {$b$};
        \draw (2, 0.3)   node {$c$};
        \draw (2.2,-0.8) node {$\sigma_i$};
        \draw (1.5,-0.75) node[left] {$\theta_i$};

        \draw [blue, line width = 1.2] (1.5,0.5) -- (1.5,-0.75);
        \draw [red,  line width = 1.2] (1.5,-0.75) -- (4,0);
        \draw [blue, line width = 1.2] (4,0) -- (4.5,-0.5);
    \end {scope}

    \begin {scope} [shift={(6,2)}]
        \draw (0,0) -- (4,0) -- (1,-2) -- cycle;
        \draw (4,0) -- (1.5,-0.75);
        \draw[fill=black] (1.5,-0.75) circle (0.05);

        \draw (0.3,-1.2) node {$a$};
        \draw (2.5,-1.4) node {$b$};
        \draw (2, 0.3)   node {$c$};
        \draw (2.2,-0.8) node {$\sigma_i$};
        \draw (1.5,-0.75) node[left] {$\theta_i$};

        \draw [red,  line width = 1.2] (1.5,0.5) -- (1.5,-0.75);
        \draw [blue, line width = 1.8] (1.5,-0.75) -- (1,-2);
        \draw [red,  line width = 1.2] (1,-2) -- (4,0);
        \draw [blue, line width = 1.2] (4,0) -- (4.5,-0.5);
    \end {scope}

    \begin {scope} [shift={(0,-1)}]
        \draw (-1,-0) node {(ii)};
        \draw (0,0) -- (4,0) -- (1,-2) -- cycle;
        \draw (4,0) -- (1.5,-0.75);
        \draw[fill=black] (1.5,-0.75) circle (0.05);

        \draw (0.3,-1.2) node {$a$};
        \draw (2.5,-1.4) node {$b$};
        \draw (2, 0.3)   node {$c$};
        \draw (2.2,-0.8) node {$\sigma_i$};
        \draw (1.5,-0.75) node[left] {$\theta_i$};

        \draw [blue, line width = 1.2] (1.5,0.5) -- (1.5,-0.75);
        \draw [red,  line width = 1.2] (1.5,-0.75) -- (4,0);
        \draw [blue, line width = 1.2] (4,0) -- (1,-2);
        \draw [red,  line width = 1.2] (1,-2) -- (1.5,-2.5);
    \end {scope}

    \begin {scope} [shift={(6,-1)}]
        \draw (0,0) -- (4,0) -- (1,-2) -- cycle;
        \draw (4,0) -- (1.5,-0.75);
        \draw[fill=black] (1.5,-0.75) circle (0.05);

        \draw (0.3,-1.2) node {$a$};
        \draw (2.5,-1.4) node {$b$};
        \draw (2, 0.3)   node {$c$};
        \draw (2.2,-0.8) node {$\sigma_i$};
        \draw (1.5,-0.75) node[left] {$\theta_i$};

        \draw [red,  line width = 1.2] (1.5,0.5) -- (1.5,-0.75);
        \draw [blue, line width = 1.8] (1.5,-0.75) -- (1,-2);
        \draw [red,  line width = 1.2] (1,-2) -- (1.5,-2.5);
    \end {scope}
\end {tikzpicture}
    \caption {Illustration of the bijection between super $T$-paths and twisted super $T$-paths; here the vertical steps through edge $c$ (or respectively $z$) indicate the beginning (resp. ending) of a $\tau$-step, i.e. teleportation}
    \label {fig:old_new_bijection}
    \end {figure}

    Furthermore, comparing weights to twisted weights, this transformation is weight-preserving.  
    We verify this case-by-case following Figure \ref{fig:old_new_bijection}.  In particular, we see the weights of the $\sigma$-step 
    (or corresponding two steps) on the left-hand-side (respectively right-hand-side) are as in \Cref{table:weights}.

    {\centering
    \begin{table}
    \begin{tabular}{lccccccc}
    & \underline{ORIGINAL}\quad\quad &~&~&~&~&~& \underline{TWISTED}  \\
    &~&~&~&~&~&~&~ \\
    (i)&$\displaystyle \sqrt{\frac{a}{bc}}\theta_i$&~&~&~&~& ~ &$\displaystyle \frac{1}{b} \cdot \sqrt{\frac{ab}{c}} \theta_i$\\
    (ii)& $\displaystyle b \cdot \sqrt{\frac{a}{bc}}\theta_i$ &~&~&~&~& ~ & $\displaystyle \sqrt{\frac{ab}{c}} \theta_i$\\
    (iii) & $\displaystyle \sqrt{\frac{x}{yz}}\theta_j$ &~&~&~&~&~& $\displaystyle \frac{1}{y} \cdot \sqrt{\frac{xy}{z}} \theta_j$\\
    (iv)&$\displaystyle y \cdot \sqrt{\frac{x}{yz}}\theta_j$&~&~&~&~& ~ &$\displaystyle \sqrt{\frac{xy}{z}} \theta_j$ \\
    &~&~&~&~&~&~&~
    \end{tabular}
    \caption{Weights of pieces of a super $T$-path (Left) or twisted super $T$-path (Right), respectively}
    \label{table:weights}
    \end{table}
    }

    \vspace{1em}

    Comparing these expressions, we see that these weights on the left- and right-hand sides indeed agree.  
    Since $\tau$-steps are unweighted, notice that changing the parity of the $\tau$-step does not affect the weight of the super 
    or twisted super $T$-path.  We thus first apply this transformation to a $\sigma$-step in the triangle corresponding to the $\mu$-invariant 
    $\theta_i$, as illustrated in case (i) or (ii).  Then second, we switch the parity of the $\tau$-step $\tau_{ij}$ without changing the weight.  
    Finally, third, we apply the same transformation to a $\sigma$-step in the triangle corresponding to the $\mu$-invariant $\theta_j$, as illustrated 
    in case (iii) or (iv).  Thus we have replaced a super step with a twisted super step, or vice-versa.

    These transformations can be applied either from left-to-right (original to twisted super $T$-paths) or instead 
    from right-to-left, thus providing the inverse map for this bijection.  
\end{proof}

As a consequence of this bijection, we obtain important constraints on $\sigma$-steps of a twisted super $T$-paths that go beyond the axioms listed above.

\begin{corollary} \label{Cor:super_steps} 
    In addition to axioms (T1), (T2), (T3), (T4), (T5'), (T6'), and (T7), we can deduce that in any twisted super $T$-path, we have the following two properties:
    \begin {itemize}
        \item[$(i)$]  we only allow edges moving from a boundary vertex to a face center along a {\bf\textcolor{blue}{$\sigma^A$-edge}} 
                      or allow edges moving from a face center to a boundary vertex along a \textcolor{cyan}{$\sigma^B$-edge}. 
        \item[$(ii)$] any $\tau$-edge, e.g. $\tau_{ij}$ between face center $\theta_i$ and face center $\theta_j$, must be immediately 
                      preceded by the edge $\sigma^A_i$ and immediately followed by the edge $\sigma^B_j$.  
    \end {itemize}
\end{corollary}

As a consequence of this corollary, twisted super steps must specifically have the form of the edge $\sigma^A_i$ followed by $\tau_{ij}$ and then $\sigma^B_j$.  
Thus we may be more precise than in \Cref{def:twisted_super_steps}.

\begin{proof} [Proof of \Cref{thm:main_twisted}]
This is now an immediate consequence of combining \Cref{lem:twisted_bijection} and Theorem 4.9 of \cite{moz21}.
\end{proof}

\subsection{Bijection between twisted super $T$-paths and double dimer covers}

In this subsection we use the new combinatorial objects presented in the previous subsection to give a weight-preserving bijection between double dimer covers  
on snake graphs and super $T$-paths.  In light of the bijection from \Cref{lem:twisted_bijection}, it is sufficient to provide a weight-preserving 
bijection between \emph{twisted} super $T$-paths and double dimer covers.  

\begin{theorem}
    Given the hypotheses of Lemma \ref{lem:twisted_bijection}, i.e. a triangulation $T$ with longest arc $(a,b)$, we build the corresponding 
    snake graph $G$ following the construction in \Cref{def:Ggamma}.  Then we have a weight-preserving bijection between 
    twisted super $T$-paths and double dimer covers of $G$.
\end{theorem}

\begin{proof}
    As indicated in Corollary \ref{Cor:super_steps}, every teleportation step, i.e. $\tau$-edge, of a twisted super $T$-path is 
    preceded by a $\sigma^A$-edge and followed by a $\sigma^B$-edge.  
    We thus are able to decompose any twisted super $T$-path into building blocks of three different possible types:
    \begin {itemize}
        \item[$(i)$] Odd-indexed edges $t_i$ that travel along either a boundary edge or internal diagonal of triangulation $T$.  
                     We will refer to such edges as {blue steps}.
        \item[$(ii)$] Even indexed edges $t_i$ that must travel along an internal diagonal of triangulation $T$.  We will refer to such edges as red steps.
        \item[$(iii)$] Twisted super steps, which are triples $(\sigma^A_i, \tau_{ij}, \sigma^B_j)$ such that $\sigma^A_i$ (resp. $\sigma^B_j$) 
                       is an entering (resp. exiting) $\sigma$-step, and $\tau_{ij}$ is a teleportation step between 
                       the two associated face centers $\theta_i$ and $\theta_j$.  By Axiom (T6') in \Cref{def:super-T}, 
                       the steps $\sigma^A_i$ and $\sigma^B_j$ are odd-indexed while $\tau_{ij}$ is even-indexed.
    \end {itemize}

    We map twisted super $T$-paths to double dimer covers in $G$, the snake graph associated to arc $\gamma$ and triangulation $T$, via these building blocks.  
    First off, we map twisted super steps to pieces in double dimer covers that correspond to a cycle around a connected sub-snake graph.  
    More precisely, consider a twisted super step $(\sigma^A_i, \tau_{ij}, \sigma^B_j)$ such that $\tau_{ij}$ teleports from the face center 
    associated to $\theta_i$ to the face center associated to $\theta_j$.  Then there is a unique tile of the snake graph $G$ that contains a 
    triangle in the lower left corner labeled by $\theta_i$ and a unique tile of $G$ that contains a triangle in the upper right corner labeled 
    by $\theta_j$.  Further, this second tile is to the northeast of the first tile since the $j$th face center is closer to the end of $\gamma$ than 
    the $i$th face center, due to axiom (T5').  The cycle circumscribes the connected sub-snake graph beginning and ending at these two tiles.  

    Since a twisted super step must begin and end with an odd-indexed step, a twisted super $T$-path cannot contain two twisted super steps 
    immediately following one another.  Instead, unless they lie at the beginning or end of the twisted super $T$-path, they must be immediately preceded 
    and followed by a red edge.  Excluding twisted super steps and these accompanying red edges, the remaining connected pieces of a twisted super $T$-path 
    all consist entirely of blue and red edges, and are each of odd-length.  Due to the axioms of \Cref{def:super-T}, these connected pieces are ordinary $T$-paths 
    of a subtriangulation $S$ of $T$ with longest diagonal $\gamma'$.  It thus suffices to apply the bijection from \cite{ms10} to obtain dimer covers of sub-snake 
    graphs associated to such ordinary $T$-paths on subtriangulation $S$ and longest diagonal $\gamma'$.  If $\gamma'$ starts on the bottom of the triangle with face 
    center $\theta_i$ and ends at the top of the triangle with face center $\theta_j$, we get a dimer cover on the connected sub-snake graph between the tile with lower-left 
    corner $\theta_i$ and upper-right corner $\theta_j$.  We duplicate each of the edges in this dimer cover to obtain a double dimer cover.

    Lastly, we note that the red edges immediately preceding or following a twisted super step correspond to blank tiles in the snake graph $G$ such that 
    if the red edge is on the diagonal $\tau_k$, then the blank tile in $G$ is the unique one with diagonal $x_k$.   

    Every diagonal of $T$ is either crossed by one of these ordinary $T$-paths, by one of the red edges preceding or following a twisted super step, 
    or by the $\tau$-step as part of a twisted super step.  Hence every tile of $G$ is respectively either covered by doubled edges, a blank tile, 
    or a cycle of single edges.  The inverse map combines the inverse map in \cite{ms10} to get ordinary $T$-paths from double dimer covers on a 
    connected sub-snake graph consisting exclusively of doubled edges and sending cycles to the unique twisted super step defined by the face centers associated 
    to the beginning and ending tiles.  

    With this combinatorial bijection defined, we wish to also verify that the weights of twisted super $T$-paths agree with weights of the associated double dimer covers.  
    Again, we utilize the decomposition of a twisted super $T$-path into blue edges, red edges and twisted super steps.  The weights of blue edges correspond to the weights 
    of the associated dimer in $G$, using the fact that we map blue edges in a twisted super $T$-path to a doubled-edge in $G$ but then weight it by a square-root.

    Similarly, the weights of red edges do not contribute anything to the numerator but either correspond to a piece of an ordinary $T$-path or to a blank 
    tile associated to the edge preceding or following a twisted super step.  Either way, the weights of such red edges agree with the contributions to the weight of the 
    double dimer cover given by the crossing monomial in the denominator.  

    Finally, a twisted super step consists of a $\sigma^A_i$-step, an unweighted $\tau$-step, and a $\sigma^B_j$-step.  As illustrated in the right-hand-side of 
    Figure \ref{fig:old_new_bijection}, in cases (ii) and (iv), the weight of the $\sigma^A_i$-edge is $\sqrt{\frac{ab}{c}} \theta_i$ while the weight of 
    the $\sigma^B_j$-edge is $\sqrt{\frac{xy}{z}}\theta_j$.  

    Letting $\gamma_{AB}$ denote the arc starting from the source of the $\sigma^A_i$-edge and ending at the target of the $\sigma^B_j$-edge, 
    let $G_{\gamma_{AB}}$ be the corresponding snake graph, as defined in Definition \ref{def:Ggamma}.  By construction, $a$ and $b$ are the weights of the 
    SW edges of the first tile in $G_{\gamma_{AB}}$ while $x$ and $y$ are the weights of the NE edges of the last tile.

    Following Theorem \ref{thm:MS10}, we consider the denominator cross$(\gamma_{AB})$ which is the product of the weights of all diagonals 
    in $G_{\gamma_{AB}}$.  However, the snake graph and its edge-weighting is also constructed so that the following is true:

    \vspace{0.5em}

    (*) For every\footnote{Exluding the diagonal edges corresponding to the first and last tiles of $G_{\gamma_{AB}}$.} 
    diagonal edge in $G_{\gamma_{AB}}$, the N or E edge of the previous tile has the same weight, as does the S or W edge of succeeding tile.  
    In both cases, we pick the one edge out of the two which is a boundary (as opposed to internal) edge of the snake graph.  

    \vspace{0.5em}

    \noindent As a consequence, the product of the square root of edge weights of all boundary edges of $G_{\gamma_{AB}}$ exactly 
    equals $\sqrt{ab \cdot \frac{\mathrm{cross}(\gamma_{AB})^2}{cz} \cdot xy}$ where $c$ (respectively $z$) denotes the diagonal of the first (resp. last) 
    tile of $G_{\gamma_{AB}}$.  We include $\theta_i$ and $\theta_j$ in this product, corresponding to the beginning and end of the twisted super step; or the 
    first and last triangles of the corresponding snake graph.

    In conclusion, the weight contributed to a twisted super $T$-path by a twisted super step consisting of ($\sigma^A_i, \tau_{ij}, \sigma^B_j$) agrees 
    with the weight contributed to a double dimer cover via a cycle of single edges around the associated sub-snake graph.
\end{proof}

\bigskip

\section {Lattice Paths and Dual Snake Graphs}

\bigskip

\subsection {Duality of Snake Graphs}

In this section, we describe an involution on the set of all snake graphs which was described in \cite{propp05}, and also discussed
extensively in \cite{claussen20}. Under this involution, dimer covers are taken to lattice paths. By a \emph{lattice path} on a snake graph,
we mean a path from the bottom-left corner of the first tile to the top-right corner of the last tile, such that every step goes
either right or up.

\bigskip

\begin {definition}
    Define an involution $x \mapsto \overline{x}$ on the set $\{\text{R},\text{U}\}$ given by $\overline{\text{R}} = \text{U}$ and 
    $\overline{\text{U}} = \text{R}$.
\end {definition}

\bigskip

\begin {definition}
    Define an involution $w \mapsto \overline{w}$ on the set of words in the alphabet $\{\text{R},\text{U}\}$ as follows.
    If $w = w_1 w_2 w_3 \cdots w_{2n}$ or $w = w_1 w_2 w_3 \cdots w_{2n+1}$, then define
    $\overline{w} = \overline{w}_1 w_2 \overline{w}_3 \cdots w_{2n}$ or $\overline{w}_1 w_2 \overline{w}_3 \cdots w_{2n} \overline{w}_{2n+1}$.
    In other words, the involution toggles the odd-numbered letters.
\end {definition}

\bigskip

\begin {definition}
    The involution on snake graphs is defined by applying the involution to the word $W(G)$. In other words, for a snake graph $G$, 
    we define the \emph{dual snake graph} $\overline{G}$ so that $W(\overline{G}) = \overline{W(G)}$.
\end {definition}

\bigskip

If the snake graph $G$ is labelled, then $\overline{G}$ inherets a labelling in the following way. Recall the tiles $T_i$ of $G$
alternate in orientation, so that $T_{2k+1}$ has usual orientation, and $T_{2k}$ has reversed orientation. For the dual snake graph,
if $i$ is odd, then the bottom and left sides of $\overline{T}_i$ are labelled the same as $T_i$, but the top and right labels are swapped.
If $i$ is even, then the top and right sides of $\overline{T}_i$ are the same as $T_i$, and the bottom and left are swapped. This is illustrated
in \Cref{fig:dual_tiles}. An example of two dual snake graphs is pictured in \Cref{fig:dual_snakes}.

\begin {figure}[h]
\centering
\begin {tikzpicture}
    \draw (-1.5,0.5) node {$T_{2k+1}$};
    \draw (0,0) --node[below]{$b$} (1,0) --node[right]{$c$} (1,1) --node[above]{$d$} (0,1) --node[left]{$a$} (0,0);

    \begin {scope} [shift={(5,0)}]
        \draw (-1.5,0.5) node {$\overline{T}_{2k+1}$};
        \draw (0,0) --node[below]{$b$} (1,0) --node[right]{$d$} (1,1) --node[above]{$c$} (0,1) --node[left]{$a$} (0,0);
    \end {scope}

    \begin {scope} [shift={(0,-3)}]
        \draw (-1.5,0.5) node {$T_{2k}$};
        \draw (0,0) --node[below]{$b$} (1,0) --node[right]{$c$} (1,1) --node[above]{$d$} (0,1) --node[left]{$a$} (0,0);

        \begin {scope} [shift={(5,0)}]
            \draw (-1.5,0.5) node {$\overline{T}_{2k}$};
            \draw (0,0) --node[below]{$a$} (1,0) --node[right]{$c$} (1,1) --node[above]{$d$} (0,1) --node[left]{$b$} (0,0);
        \end {scope}
    \end {scope}
\end {tikzpicture}
\caption {Relation between tile labeling in $G$ and $\overline{G}$}
\label {fig:dual_tiles}
\end {figure}
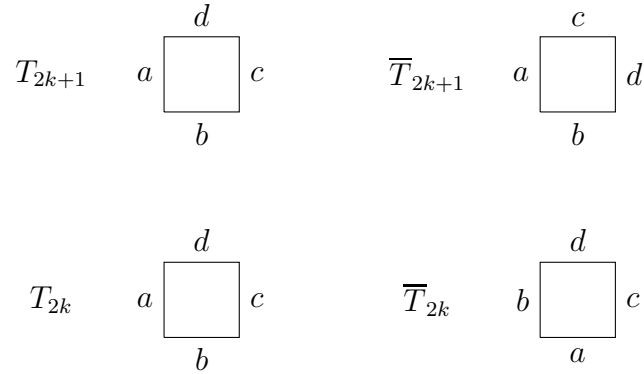

\begin {figure}[h]
\centering
\begin {tikzpicture}
    \draw (0,1) --node[left]{$a$}  (0,0) --node[below]{$b$} (1,0) --node[right]{$x$} (1,1) --node[above]{$j$} (0,1);
    \draw (1,0) --node[below]{$c$} (2,0) --node[right]{$y$} (2,1) --node[above]{$i$} (1,1);
    \draw (2,0) --node[below]{$d$} (3,0) --node[right]{$z$} (3,1) --node[above]{$h$} (2,1);
    \draw (3,0) --node[below]{$e$} (4,0) --node[right]{$f$} (4,1) --node[above]{$g$} (3,1);

    \draw (2,-1.5) node {$W(G) = \text{RRR}$};

    \begin {scope} [shift={(6,0)}]
        \draw (0,1) --node[left]{$a$}  (0,0) --node[below]{$b$} (1,0) --node[right]{$j$} (1,1) --node[above]{$x$} (0,1);
        \draw (1,1) --node[right]{$y$} (1,2) --node[above]{$i$} (0,2) --node[left]{$c$} (0,1);
        \draw (1,1) --node[below]{$d$} (2,1) --node[right]{$h$} (2,2) --node[above]{$z$} (1,2);
        \draw (2,2) --node[right]{$f$} (2,3) --node[above]{$g$} (1,3) --node[left]{$e$} (1,2);

        \draw (1,-1.5) node {$W(G) = \text{URU}$};
    \end {scope}
\end {tikzpicture}
\caption {Two dual snake graphs}
\label {fig:dual_snakes}
\end {figure}
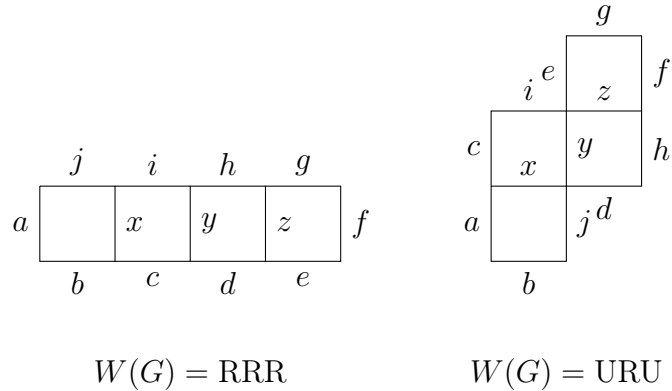

There is a bijection, described in \cite{propp05} and \cite{claussen20}, between perfect matchings (dimer covers) of $G$ and lattice paths in $\overline{G}$
going from the bottom left to the top right corner. With the labelling convention described above for $\overline{G}$, this bijection is weight-preserving,
where the weight of a lattice path is the product of the weights of the edges in the path. An example of this correspondence is pictured in \Cref{fig:dimers_lattice_paths}.

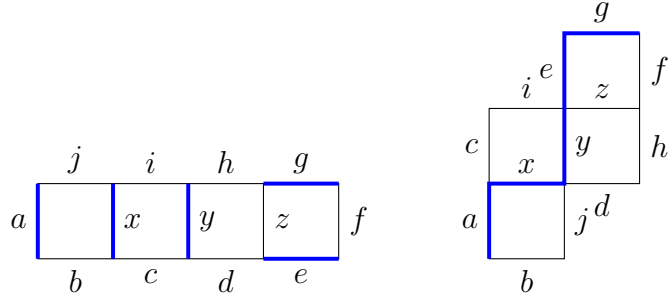
\begin {figure}[h]
\centering
\begin {tikzpicture}
    \draw (0,1) --node[left]{$a$}  (0,0) --node[below]{$b$} (1,0) --node[right]{$x$} (1,1) --node[above]{$j$} (0,1);
    \draw (1,0) --node[below]{$c$} (2,0) --node[right]{$y$} (2,1) --node[above]{$i$} (1,1);
    \draw (2,0) --node[below]{$d$} (3,0) --node[right]{$z$} (3,1) --node[above]{$h$} (2,1);
    \draw (3,0) --node[below]{$e$} (4,0) --node[right]{$f$} (4,1) --node[above]{$g$} (3,1);

    \draw[blue, line width = 1.5] (0,0) -- (0,1);
    \draw[blue, line width = 1.5] (1,0) -- (1,1);
    \draw[blue, line width = 1.5] (2,0) -- (2,1);
    \draw[blue, line width = 1.5] (3,0) -- (4,0);
    \draw[blue, line width = 1.5] (3,1) -- (4,1);

    \begin {scope} [shift={(6,0)}]
        \draw (0,1) --node[left]{$a$}  (0,0) --node[below]{$b$} (1,0) --node[right]{$j$} (1,1) --node[above]{$x$} (0,1);
        \draw (1,1) --node[right]{$y$} (1,2) --node[above]{$i$} (0,2) --node[left]{$c$} (0,1);
        \draw (1,1) --node[below]{$d$} (2,1) --node[right]{$h$} (2,2) --node[above]{$z$} (1,2);
        \draw (2,2) --node[right]{$f$} (2,3) --node[above]{$g$} (1,3) --node[left]{$e$} (1,2);

        \draw[blue, line width = 1.5] (0,0) -- (0,1) -- (1,1) -- (1,3) -- (2,3);
    \end {scope}
\end {tikzpicture}
\caption {Correspondence between dimer covers of $G$ and lattice paths in $\overline{G}$}
\label {fig:dimers_lattice_paths}
\end {figure}

\subsection {Double Lattice Paths}

Recall that a double dimer cover is a multiset of edges such that each vertex is incident to exactly two edges from the multiset.
Another way of describing it is as follows. Let $C(G)$ be the set of dimer covers of a snake graph $G$. There is a natural map
$\pi \colon C(G) \times C(G) \to D(G)$. For each pair of dimer covers, $\pi$ maps the pair to the corresponding multiset. This map
is not in general injective, and different pairs can result in the same double dimer cover. Despite this non-uniqueness, it is often
convenient to think of double dimer covers as simply pairs of dimer covers.

In the previous section, we described a bijection which takes dimer covers on $G$ to lattice paths on $\overline{G}$. Let $\ell(\overline{G})$ be the
set of lattice paths on $\overline{G}$, and let $f \colon C(G) \to \ell(\overline{G})$ be this bijection. We then naturally
get a map $f \times f \colon C(G) \times C(G) \to \ell(\overline{G}) \times \ell(\overline{G})$. As before, we have a map
$\pi'$ from $\ell(\overline{G}) \times \ell(\overline{G})$ to the set of multisets of edges on $\overline{G}$. 

\bigskip

\begin {definition}
    Given a snake graph $G$, define the set $L(\overline{G})$ to be the image of the map $\pi'$.
    Its elements are called \emph{double lattice paths} on $\overline{G}$.
\end {definition}

\bigskip

In other words, we want to think of a double
lattice path as the superposition of two lattice paths (despite the fact that, as with double dimer covers, this is not always unique).

As with double dimer covers, we draw edges of a lattice path as wavy orange lines, where a solid blue line indicates two edges superimposed.

The following is a convenient way of specifying a double lattice path.

\bigskip

\begin {definition}
    Given a snake graph $G$, consider all ways to label each tile with one of the numbers $0$, $1$, or $2$.
    Let $a_i$ be the label of tile $T_i$. Let $X(G)$ be the subset of labelings such that:
    \begin {itemize}
        \item If $T_{i+1}$ is to the right of $T_i$, then $a_i \leq a_{i+1}$
        \item If $T_{i+1}$ is above $T_i$, then $a_i \geq a_{i+1}$
    \end {itemize}
\end {definition}

\bigskip

\begin {theorem}
    There is a bijection between $X(G)$ and $L(G)$.
\end {theorem}
\begin {proof}
    Given a double lattice path, choose a pre-image of $\pi'$ in $\ell(G) \times \ell(G)$. In other words, choose a pair of lattice paths
    which represents the double lattice path. Label each tile according to how many of the two lattice paths go above this tile. Because
    the lattice paths can only go up and right, this ensures the inequalities, so this gives an element of $X(G)$. Conversely, given a
    choice of labels $0$, $1$, or $2$ for each tile, the double lattice path can uniquely be reconstructed.
\end {proof}

\bigskip

\begin {definition}
    The weight of a double lattice path $P$ is the product of the square roots of all edge labels in $P$, times the product of odd variables corresponding to the
    beginning and end of each cycle in $P$. The product of odd variables is taken in the positive order.
\end {definition}

\bigskip

The following is immediate from the weight-preserving bijection mentioned in the previous section.

\bigskip

\begin {theorem}
    Let $G$ be the snake graph of a diagonal $\gamma=(i,j)$ in a triangulated polygon. Then
    \[ x_{ij} = \frac{1}{\mathrm{cross}(i,j)} \sum_{P \in L(\overline{G})} \mathrm{wt}(P) \]
\end {theorem}

\bigskip

\section {Order Ideals and Distributive Lattices}

\bigskip

It was shown in \cite{propp02} that the set of dimer covers of a planar graph forms a distributive lattice,
whose cover relations are local moves called \emph{twists}. By Birkhoff's theorem, every distributive lattice
is isomorphic to the lattice of order ideals of some poset. For a poset $P$, let $J(P)$ denote the lattice of
order ideals in $P$ (ordered by inclusion). For the example of the lattice of dimer covers of
a snake graph $G$, a description of this poset was given in \cite{msw13} (and studied also in \cite{claussen20} and \cite{mss21}).

Using the bijection described earlier between dimer covers of $G$ and lattice paths of $\overline{G}$, it will be
easier to use the latter description. We will describe a bijection between lattice paths of $\overline{G}$ 
and lower order ideals of a certain poset $P(\overline{G})$. The Hasse diagram of $P(\overline{G})$ is constructed as follows. Put a vertex
inside each tile of $\overline{G}$, and connect two vertices with an edge (a cover relation in $P(\overline{G})$) if the tiles share a side.
Finally, rotate this picture clockwise $45^\circ$ to get the Hasse diagram of $P(\overline{G})$. This is pictured in \Cref{fig:hasse_diagram}.

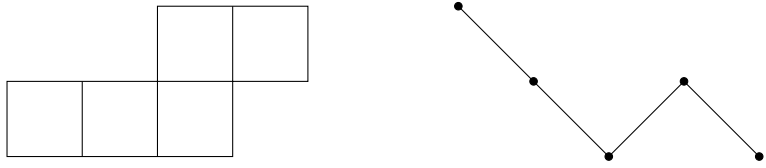
\begin {figure}[h]
\centering
\begin {tikzpicture}
    \draw (0,0) -- ++(3,0) -- ++(0,1) -- ++(1,0) -- ++(0,1) -- ++(-2,0) -- ++(0,-1) -- ++(-2,0) -- cycle;
    \draw (1,0) -- ++(0,1);
    \draw (2,0) -- ++(0,1) -- ++(1,0) -- ++(0,1);

    \begin {scope} [shift={(6,0)}]
        \draw (0,2) -- ++(2,-2) -- ++(1,1) -- ++(1,-1);
        \draw[fill=black] (0,2) circle (0.05);
        \draw[fill=black] (1,1) circle (0.05);
        \draw[fill=black] (2,0) circle (0.05);
        \draw[fill=black] (3,1) circle (0.05);
        \draw[fill=black] (4,0) circle (0.05);
    \end {scope}
\end {tikzpicture}
\caption {The Hasse diagram of the poset $P(\overline{G})$}
\label {fig:hasse_diagram}
\end {figure}

We define an order relation on the set of lattice paths in $\overline{G}$ as follows. The minimal lattice path
is the one which follows the bottom-right edges of the boundary of $\overline{G}$, and the maximal path is the one wich
uses the top-left edges of the boundary. If a lattice path uses both the bottom and right edge of some tile, then this path
is covered by the path which swaps those for the top and left edges of the same tile.

The bijection between lattice paths and order ideals is as follows. As described above, the elements of $P(\overline{G})$ correspond to
the tiles of $\overline{G}$. For each lattice path, there is some subset of tiles which are underneath the path. The corresponding
subset of vertices of the Hasse diagram of $P(\overline{G})$ is an order ideal. The minimal lattice path has no tiles underneath it,
and so corresponds to the empty set. An example is pictured in \Cref{fig:order_ideals}.

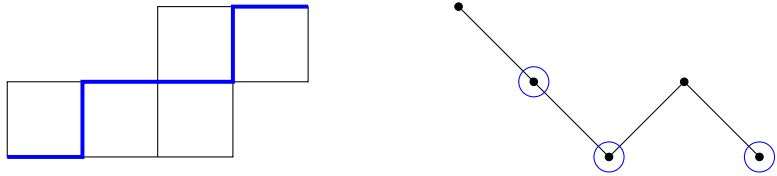
\begin {figure}[h]
\centering
\begin {tikzpicture}
    \draw (0,0) -- ++(3,0) -- ++(0,1) -- ++(1,0) -- ++(0,1) -- ++(-2,0) -- ++(0,-1) -- ++(-2,0) -- cycle;
    \draw (1,0) -- ++(0,1);
    \draw (2,0) -- ++(0,1) -- ++(1,0) -- ++(0,1);

    \draw[blue, line width = 1.5] (0,0) -- (1,0) -- (1,1) -- (3,1) -- (3,2) -- (4,2);

    \begin {scope} [shift={(6,0)}]
        \draw (0,2) -- ++(2,-2) -- ++(1,1) -- ++(1,-1);
        \draw[fill=black] (0,2) circle (0.05);
        \draw[fill=black] (1,1) circle (0.05);
        \draw[fill=black] (2,0) circle (0.05);
        \draw[fill=black] (3,1) circle (0.05);
        \draw[fill=black] (4,0) circle (0.05);

        \draw[blue, fill=none] (1,1) circle (0.2);
        \draw[blue, fill=none] (2,0) circle (0.2);
        \draw[blue, fill=none] (4,0) circle (0.2);
    \end {scope}
\end {tikzpicture}
\caption {A lattice path in $\overline{G}$ and its corresponding order ideal in $P(\overline{G})$}
\label {fig:order_ideals}
\end {figure}

We can extend this idea to the current situation of double dimer covers and double lattice paths.
Let $\Bbb{P}(\overline{G})$ be the product of $P(\overline{G})$ with a chain of length 2 (with elements $0 < 1$).
The following is the double dimer analogue of the situation described above.

\bigskip

\begin {theorem} \label{thm:poset_isomorphism}
    There is a poset isomorphism $L(G) \cong J(\Bbb{P}(G))$, between the set of double lattice paths in $G$ and the
    lattice of lower order ideals in $\Bbb{P}(G)$.
\end {theorem}
\begin {proof}
    Since $\Bbb{P}(G) = P(G) \times \{0,1\}$,
    we will write elements as $(x,0)$ or $(x,1)$, where $x \in P(G)$. 

    We already have a bijection between $L(G)$ and $X(G)$. So we will give a bijection between $X(G)$
    and $J(\Bbb{P}(G))$. The tiles in the snake graph correspond to the elements of the poset $P(G)$.
    Let $p_i$ be the element of the poset corresponding to the tile $T_i$.
    Given a labelling from $X(G)$, we will construct an order ideal. 
    If $T_i$ is labelled $0$, then neither $(p_i,0)$ nor $(p_i,1)$ are in the order ideal.
    If $T_i$ is labelled $1$, then $(p_i,0)$ is in the order ideal, but $(p_i,1)$ is not.
    If $T_i$ is labelled $2$, then both $(p_i,0)$ and $(p_i,1)$ are in the order ideal.
    The inequalities defining $X(G)$ guarantee that this is an order ideal.
    The inverse to this map is obvious, and it is clearly a bijection.

    To see that this bijection is order-preserving, note that the cover relations in terms of $X(G)$ are simply changing the label
    of a tile by increasing its value by $1$. Under the bijection described above, this corresponds to adding a single element (either $(p_i,0)$ or $(p_i,1)$)
    to the order ideal, which is a cover relation in the lattice of order ideals.
\end {proof}

\bigskip

\begin {example}
    An example of the poset isomorphism from \Cref{thm:poset_isomorphism} is shown in \Cref{fig:poset_isomorphism}.

    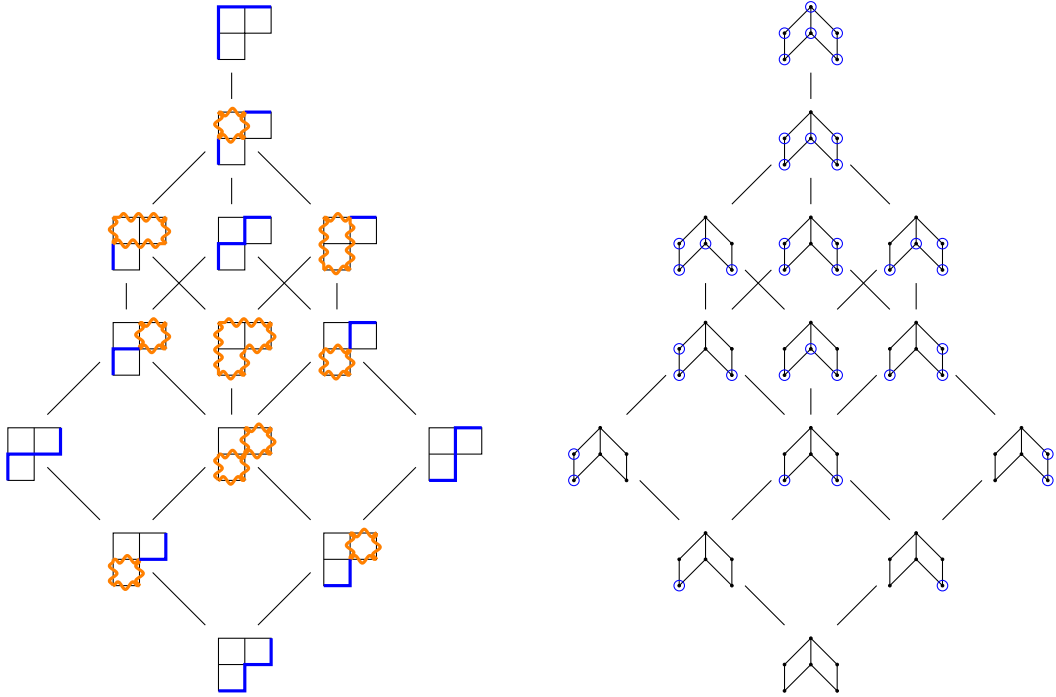
\begin {figure}[h]
    \centering
    \begin {tikzpicture}[scale=0.35]
    \tikzset {decoration={snake, amplitude=0.5mm, segment length = 2mm}}
    \newcommand{\drawsnake}{
        \draw (0,0) -- (0,1) -- (3,1) -- (3,0) -- cycle;
        \draw (1,0) -- (1,1);
        \draw (2,0) -- (2,1);
    }

    \newcommand{\drawdualsnake}{
        \draw (0,0) -- (0,2) -- (2,2) -- (2,1) -- (1,1) -- (1,0) -- cycle;
        \draw (0,1) -- (1,1) -- (1,2);
    }

    \newcommand{\drawposet}{
        \begin {scope}[shift={(-0.5,0)}]
        \draw[fill=black] (0,0) circle (0.05);
        \draw[fill=black] (1,1) circle (0.05);
        \draw[fill=black] (2,0) circle (0.05);
        \draw[fill=black] (0,1) circle (0.05);
        \draw[fill=black] (1,2) circle (0.05);
        \draw[fill=black] (2,1) circle (0.05);

        \draw (0,0) -- (1,1) -- (2,0);
        \draw (0,1) -- (1,2) -- (2,1);
        \draw (0,0) -- (0,1);
        \draw (1,1) -- (1,2);
        \draw (2,0) -- (2,1);
        \end {scope}
    }


    \drawdualsnake

    \draw[blue, line width = 1.2] (0,0) -- (1,0) -- (1,1) -- (2,1) -- (2,2);

    \draw (-0.5,2.5) -- (-2.5,4.5);
    \draw (1.5,2.5) -- (3.5,4.5);


    \begin {scope}[shift={(-4,4)}]
        \drawdualsnake

        \draw[orange, decorate, line width = 1.2] (0,0) -- (1,0);
        \draw[orange, decorate, line width = 1.2] (1,0) -- (1,1);
        \draw[orange, decorate, line width = 1.2] (1,1) -- (0,1);
        \draw[orange, decorate, line width = 1.2] (0,0) -- (0,1);

        \draw[blue, line width = 1.2] (1,1) -- (2,1) -- (2,2);

        \draw (-0.5,2.5) -- (-2.5,4.5);
        \draw (1.5,2.5) -- (3.5,4.5);
    \end {scope}

    \begin {scope}[shift={(4,4)}]
        \drawdualsnake

        \draw[orange, decorate, line width = 1.2] (1,1) -- (2,1);
        \draw[orange, decorate, line width = 1.2] (2,1) -- (2,2);
        \draw[orange, decorate, line width = 1.2] (2,2) -- (1,2);
        \draw[orange, decorate, line width = 1.2] (1,2) -- (1,1);

        \draw[blue, line width = 1.2] (0,0) -- (1,0) -- (1,1);

        \draw (-0.5,2.5) -- (-2.5,4.5);
        \draw (1.5,2.5) -- (3.5,4.5);
    \end {scope}


    \begin {scope}[shift={(-8,8)}]
        \drawdualsnake

        \draw[blue, line width = 1.2] (0,0) -- (0,1) -- (2,1) -- (2,2);

        \draw (1.5,2.5) -- (3.5,4.5);
    \end {scope}

    \begin {scope}[shift={(0,8)}]
        \drawdualsnake

        \draw[orange, decorate, line width = 1.2] (0,0) -- (1,0);
        \draw[orange, decorate, line width = 1.2] (1,0) -- (1,1);
        \draw[orange, decorate, line width = 1.2] (1,1) -- (0,1);
        \draw[orange, decorate, line width = 1.2] (0,0) -- (0,1);

        \draw[orange, decorate, line width = 1.2] (1,1) -- (2,1);
        \draw[orange, decorate, line width = 1.2] (2,1) -- (2,2);
        \draw[orange, decorate, line width = 1.2] (2,2) -- (1,2);
        \draw[orange, decorate, line width = 1.2] (1,2) -- (1,1);

        \draw (0.5,2.5) -- (0.5,3.5);
        \draw (-0.5,2.5) -- (-2.5,4.5);
        \draw (1.5,2.5) -- (3.5,4.5);
    \end {scope}

    \begin {scope}[shift={(8,8)}]
        \drawdualsnake

        \draw[blue, line width = 1.2] (0,0) -- (1,0) -- (1,2) -- (2,2);

        \draw (-0.5,2.5) -- (-2.5,4.5);
    \end {scope}


    \begin {scope}[shift={(-4,12)}]
        \drawdualsnake

        \draw[orange, decorate, line width = 1.2] (1,1) -- (2,1);
        \draw[orange, decorate, line width = 1.2] (2,1) -- (2,2);
        \draw[orange, decorate, line width = 1.2] (2,2) -- (1,2);
        \draw[orange, decorate, line width = 1.2] (1,2) -- (1,1);

        \draw[blue, line width = 1.2] (0,0) -- (0,1) -- (1,1);

        \draw (0.5,2.5) -- (0.5,3.5);
        \draw (1.5,2.5) -- (3.5,4.5);
    \end {scope}

    \begin {scope}[shift={(0,12)}]
        \drawdualsnake

        \draw[orange, decorate, line width = 1.2] (0,0) -- (1,0);
        \draw[orange, decorate, line width = 1.2] (1,0) -- (1,1);
        \draw[orange, decorate, line width = 1.2] (1,1) -- (2,1);
        \draw[orange, decorate, line width = 1.2] (2,1) -- (2,2);
        \draw[orange, decorate, line width = 1.2] (2,2) -- (0,2);
        \draw[orange, decorate, line width = 1.2] (0,2) -- (0,0);

        \draw (-0.5,2.5) -- (-2.5,4.5);
        \draw (1.5,2.5) -- (3.5,4.5);
    \end {scope}

    \begin {scope}[shift={(4,12)}]
        \drawdualsnake

        \draw[orange, decorate, line width = 1.2] (0,0) -- (1,0);
        \draw[orange, decorate, line width = 1.2] (1,0) -- (1,1);
        \draw[orange, decorate, line width = 1.2] (1,1) -- (0,1);
        \draw[orange, decorate, line width = 1.2] (0,0) -- (0,1);

        \draw[blue, line width = 1.2] (1,1) -- (1,2) -- (2,2);

        \draw (-0.5,2.5) -- (-2.5,4.5);
        \draw (0.5,2.5) -- (0.5,3.5);
    \end {scope}


    \begin {scope}[shift={(-4,16)}]
        \drawdualsnake

        \draw[blue, line width = 1.2] (0,0) -- (0,1);

        \draw[orange, decorate, line width = 1.2] (0,1) -- (2,1);
        \draw[orange, decorate, line width = 1.2] (2,1) -- (2,2);
        \draw[orange, decorate, line width = 1.2] (2,2) -- (0,2);
        \draw[orange, decorate, line width = 1.2] (0,2) -- (0,1);

        \draw (1.5,2.5) -- (3.5,4.5);
    \end {scope}

    \begin {scope}[shift={(0,16)}]
        \drawdualsnake

        \draw[blue, line width = 1.2] (0,0) -- (0,1) -- (1,1) -- (1,2) -- (2,2);

        \draw (0.5,2.5) -- (0.5,3.5);
    \end {scope}

    \begin {scope}[shift={(4,16)}]
        \drawdualsnake

        \draw[blue, line width = 1.2] (1,2) -- (2,2);

        \draw[orange, decorate, line width = 1.2] (0,0) -- (1,0);
        \draw[orange, decorate, line width = 1.2] (1,0) -- (1,2);
        \draw[orange, decorate, line width = 1.2] (1,2) -- (0,2);
        \draw[orange, decorate, line width = 1.2] (0,2) -- (0,0);

        \draw (-0.5,2.5) -- (-2.5,4.5);
    \end {scope}


    \begin {scope}[shift={(0,20)}]
        \drawdualsnake

        \draw[blue, line width = 1.2] (0,0) -- (0,1);
        \draw[blue, line width = 1.2] (1,2) -- (2,2);

        \draw[orange, decorate, line width = 1.2] (0,1) -- (1,1);
        \draw[orange, decorate, line width = 1.2] (1,1) -- (1,2);
        \draw[orange, decorate, line width = 1.2] (1,2) -- (0,2);
        \draw[orange, decorate, line width = 1.2] (0,2) -- (0,1);

        \draw (0.5,2.5) -- (0.5,3.5);
    \end {scope}


    \begin {scope}[shift={(0,24)}]
        \drawdualsnake

        \draw[blue, line width = 1.2] (0,0) -- (0,2) -- (2,2);
    \end {scope}


    \begin {scope} [shift = {(22,0)}]

    \drawposet

    \draw (-0.5,2.5) -- (-2,4);
    \draw (1.5,2.5) -- (3,4);


    \begin {scope}[shift={(-4,4)}]
        \drawposet

        \draw[blue] (-0.5,0) circle (0.2);

        \draw (-0.5,2.5) -- (-2,4);
        \draw (1.5,2.5) -- (3,4);
    \end {scope}

    \begin {scope}[shift={(4,4)}]
        \drawposet

        \draw[blue] (1.5,0) circle (0.2);

        \draw (-0.5,2.5) -- (-2,4);
        \draw (1.5,2.5) -- (3,4);
    \end {scope}


    \begin {scope}[shift={(-8,8)}]
        \drawposet

        \draw[blue] (-0.5,0) circle (0.2);
        \draw[blue] (-0.5,1) circle (0.2);

        \draw (1.5,2.5) -- (3,4);
    \end {scope}

    \begin {scope}[shift={(0,8)}]
        \drawposet

        \draw[blue] (-0.5,0) circle (0.2);
        \draw[blue] (1.5,0) circle (0.2);

        \draw (-0.5,2.5) -- (-2,4);
        \draw (1.5,2.5) -- (3,4);
        \draw (0.5,2.5) -- (0.5,3.5);
    \end {scope}

    \begin {scope}[shift={(8,8)}]
        \drawposet

        \draw[blue] (1.5,0) circle (0.2);
        \draw[blue] (1.5,1) circle (0.2);

        \draw (-0.5,2.5) -- (-2,4);
    \end {scope}


    \begin {scope}[shift={(-4,12)}]
        \drawposet

        \draw[blue] (-0.5,0) circle (0.2);
        \draw[blue] (-0.5,1) circle (0.2);
        \draw[blue] (1.5,0) circle (0.2);

        \draw (0.5,2.5) -- (0.5,3.5);
        \draw (1.5,2.5) -- (3,4);
    \end {scope}

    \begin {scope}[shift={(0,12)}]
        \drawposet

        \draw[blue] (-0.5,0) circle (0.2);
        \draw[blue] (0.5,1) circle (0.2);
        \draw[blue] (1.5,0) circle (0.2);

        \draw (-0.5,2.5) -- (-2,4);
        \draw (1.5,2.5) -- (3,4);
    \end {scope}

    \begin {scope}[shift={(4,12)}]
        \drawposet

        \draw[blue] (-0.5,0) circle (0.2);
        \draw[blue] (1.5,1) circle (0.2);
        \draw[blue] (1.5,0) circle (0.2);

        \draw (-0.5,2.5) -- (-2,4);
        \draw (0.5,2.5) -- (0.5,3.5);
    \end {scope}


    \begin {scope}[shift={(-4,16)}]
        \drawposet

        \draw[blue] (-0.5,0) circle (0.2);
        \draw[blue] (-0.5,1) circle (0.2);
        \draw[blue] (0.5,1) circle (0.2);
        \draw[blue] (1.5,0) circle (0.2);

        \draw (1.5,2.5) -- (3,4);
    \end {scope}

    \begin {scope}[shift={(0,16)}]
        \drawposet

        \draw[blue] (-0.5,0) circle (0.2);
        \draw[blue] (-0.5,1) circle (0.2);
        \draw[blue] (1.5,1) circle (0.2);
        \draw[blue] (1.5,0) circle (0.2);

        \draw (0.5,2.5) -- (0.5,3.5);
    \end {scope}

    \begin {scope}[shift={(4,16)}]
        \drawposet
        \draw[blue] (-0.5,0) circle (0.2);
        \draw[blue] (0.5,1) circle (0.2);
        \draw[blue] (1.5,1) circle (0.2);
        \draw[blue] (1.5,0) circle (0.2);

        \draw (-0.5,2.5) -- (-2,4);
    \end {scope}


    \begin {scope}[shift={(0,20)}]
        \drawposet

        \draw[blue] (-0.5,0) circle (0.2);
        \draw[blue] (-0.5,1) circle (0.2);
        \draw[blue] (0.5,1) circle (0.2);
        \draw[blue] (1.5,1) circle (0.2);
        \draw[blue] (1.5,0) circle (0.2);

        \draw (0.5,2.5) -- (0.5,3.5);
    \end {scope}


    \begin {scope}[shift={(0,24)}]
        \drawposet

        \draw[blue] (-0.5,0) circle (0.2);
        \draw[blue] (-0.5,1) circle (0.2);
        \draw[blue] (0.5,1) circle (0.2);
        \draw[blue] (0.5,2) circle (0.2);
        \draw[blue] (1.5,1) circle (0.2);
        \draw[blue] (1.5,0) circle (0.2);
    \end {scope}
    \end {scope}

\end {tikzpicture}

    \caption {Poset isomorphism between $L(G)$ and $J(\Bbb{P}(G))$}
    \label {fig:poset_isomorphism}
    \end {figure}
\end {example}

\bigskip

\section{Examples}

\begin{example}

\input{example_snake_twisted.tex}

\end{example}

\begin{example} \label{ex:expansions}

Consider the following triangulation of a hexagon. Its snake graph and dual snake graph are also pictured.
We remind the reader of our convention of going bottom-to-top (i.e. the first tile of the snake graph corresponds
the bottom triangles in the hexagon).        

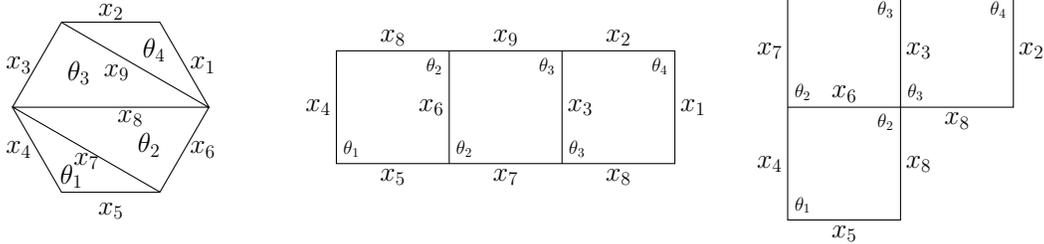
\begin {center}
\begin {tikzpicture}[scale=0.5, every node/.style={scale=0.5}]
      \path (0,0)
      node[
        regular polygon,
        regular polygon sides=6,
        draw,
        inner sep=1.6cm,
      ] (hexagon) {}
      ;

      \def\r{2.6};

      \coordinate (v6) at ($\r*(1,0)$);
      \coordinate (v1) at ($\r*(0.5,{sqrt(3)*0.5})$);
      \coordinate (v2) at ($\r*(-0.5,{sqrt(3)/2})$);
      \coordinate (v3) at ($\r*(-1,0)$);
      \coordinate (v4) at ($\r*(-0.5,{-sqrt(3)/2})$);
      \coordinate (v5) at ($\r*(0.5,{-sqrt(3)/2})$);

      \draw (v2) -- (v6) -- (v3) -- (v5);

      \draw ($0.5*(v6)+0.5*(v1) + (0.5,0)$)  node {\LARGE$x_1$};
      \draw ($0.5*(v1)+0.5*(v2) + (0,0.3)$)  node {\LARGE$x_2$};
      \draw ($0.5*(v2)+0.5*(v3) + (-0.5,0)$) node {\LARGE$x_3$};
      \draw ($0.5*(v3)+0.5*(v4) + (-0.5,0)$) node {\LARGE$x_4$};
      \draw ($0.5*(v4)+0.5*(v5) + (0,-0.5)$) node {\LARGE$x_5$};
      \draw ($0.5*(v6)+0.5*(v5) + (0.5,0)$)  node {\LARGE$x_6$};

      \draw ($0.5*(v6)+0.5*(v2) + (-0.5,-0.2)$) node {\LARGE$x_9$};
      \draw ($0.6*(v6)+0.4*(v3) + (0,-0.3)$)    node {\LARGE$x_8$};
      \draw ($0.5*(v5)+0.5*(v3) + (0,-0.3)$)    node {\LARGE$x_7$};

      \draw ($0.5*(v6)+0.5*(v2) + (0.5,0.4)$)   node {\LARGE$\theta_4$};
      \draw ($0.5*(v6)+0.5*(v2) + (-1.5,-0.2)$) node {\LARGE$\theta_3$};
      \draw ($0.6*(v6)+0.4*(v3) + (0.5,-1.0)$)  node {\LARGE$\theta_2$};
      \draw ($0.5*(v5)+0.5*(v3) + (-0.4,-0.7)$) node {\LARGE$\theta_1$};

      \begin {scope} [shift={(6,0)}]
          \draw (0,-1.5) -- ++(9,0) -- ++(0,3) -- ++(-9,0) -- cycle;
          \draw (3,-1.5) -- ++(0,3);     
          \draw (6,-1.5) -- ++(0,3); 

          \draw (0,0)      node[left]  {\LARGE $x_4$};
          \draw (1.5,-1.5) node[below] {\LARGE $x_5$};
          \draw (4.5,-1.5) node[below] {\LARGE $x_7$};
          \draw (7.5,-1.5) node[below] {\LARGE $x_8$};
          \draw (9,0)      node[right] {\LARGE $x_1$};
          \draw (1.5,1.5)  node[above] {\LARGE $x_8$};
          \draw (4.5,1.5)  node[above] {\LARGE $x_9$};
          \draw (7.5,1.5)  node[above] {\LARGE $x_2$};
          \draw (3,0)      node[left]  {\LARGE $x_6$};
          \draw (6,0)      node[right] {\LARGE $x_3$};

          \begin {scope}[shift={(0,-1.5)}]
              \draw (0.4,0.4) node {\large $\theta_1$};
              \draw (2.6,2.6) node {\large $\theta_2$};
              \draw (3.4,0.4) node {\large $\theta_2$};
              \draw (5.6,2.6) node {\large $\theta_3$};
              \draw (6.4,0.4) node {\large $\theta_3$};
              \draw (8.6,2.6) node {\large $\theta_4$};
          \end {scope}
      \end {scope}

      \begin {scope}[shift={(18,-3)}]
        \draw (0,0) -- ++(3,0) -- ++(0,3) -- ++(3,0) -- ++(0,3) -- ++(-6,0) -- cycle;
        \draw (0,3) -- ++(3,0) -- ++(0,3);

        \draw (0,1.5) node[left]  {\LARGE $x_4$};
        \draw (1.5,0) node[below] {\LARGE $x_5$};
        \draw (3,1.5) node[right] {\LARGE $x_8$};
        \draw (1.5,3) node[above] {\LARGE $x_6$};
        \draw (0,4.5) node[left]  {\LARGE $x_7$};
        \draw (1.5,6) node[above] {\LARGE $x_9$};
        \draw (3,4.5) node[right] {\LARGE $x_3$};
        \draw (4.5,3) node[below] {\LARGE $x_8$};
        \draw (4.5,6) node[above] {\LARGE $x_1$};
        \draw (6,4.5) node[right] {\LARGE $x_2$};

        \draw (0.4,0.4) node {\large $\theta_1$};
        \draw (2.6,2.6) node {\large $\theta_2$};
        \draw (0.4,3.4) node {\large $\theta_2$};
        \draw (2.6,5.6) node {\large $\theta_3$};
        \draw (3.4,3.4) node {\large $\theta_3$};
        \draw (5.6,5.6) node {\large $\theta_4$};
      \end {scope}
\end {tikzpicture}
\end {center}

The following figures show, for each non-ordinary $T$-path (i.e. those that include super steps), a comparison
of all the combinatorial models discussed in the present paper: super $T$-paths, twisted super $T$-paths, double dimer covers, and lattice paths.
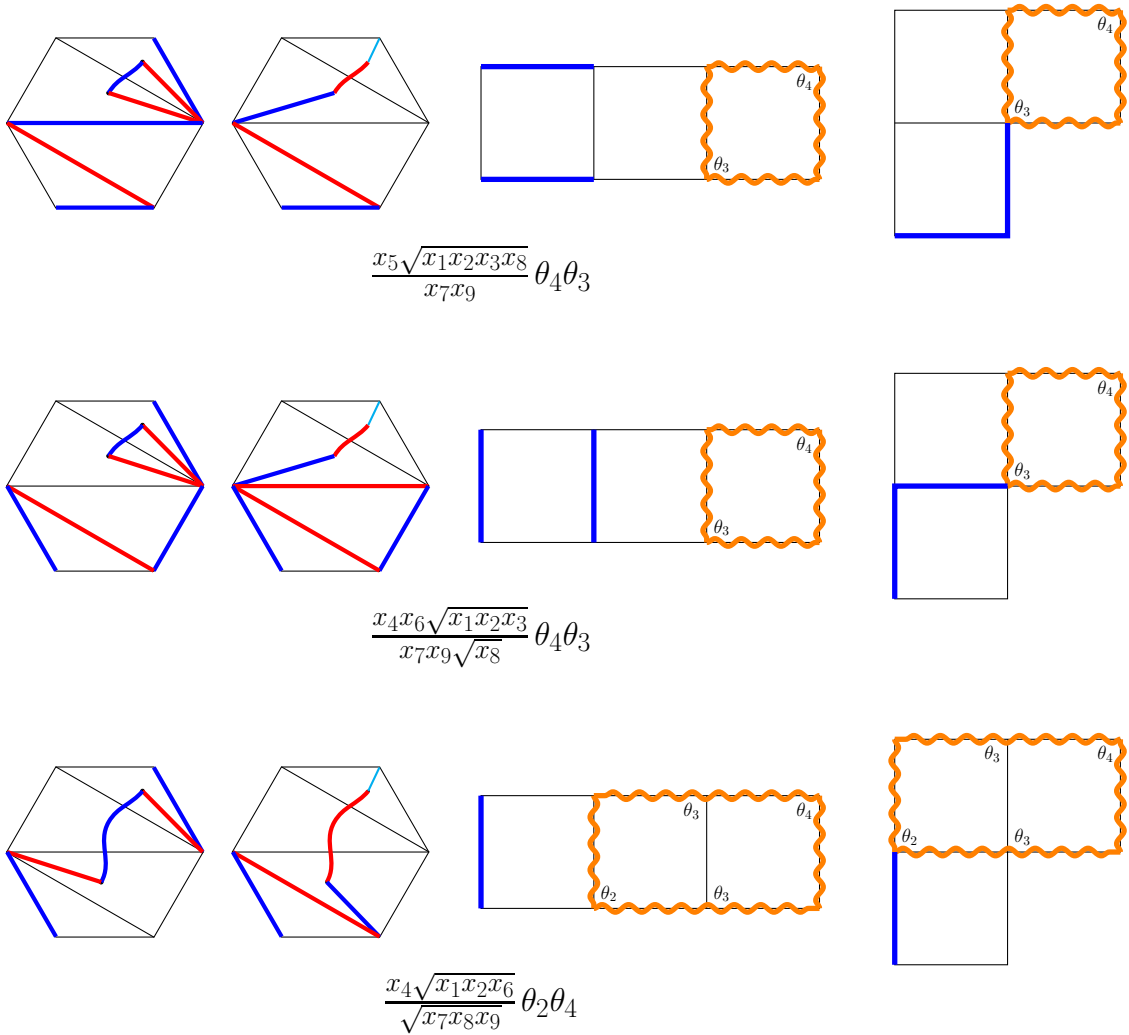
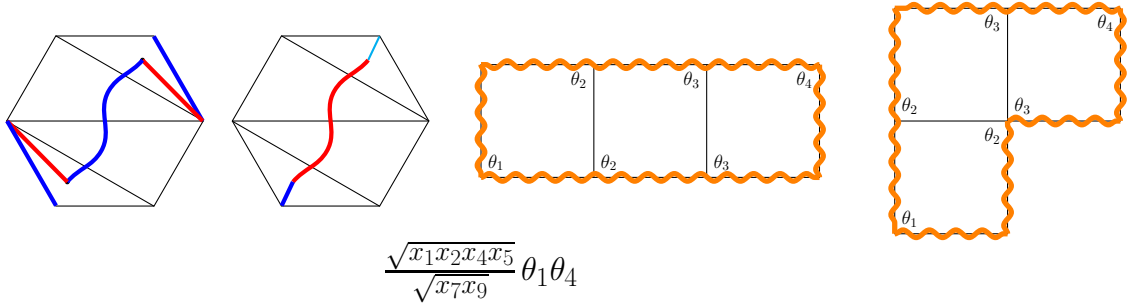
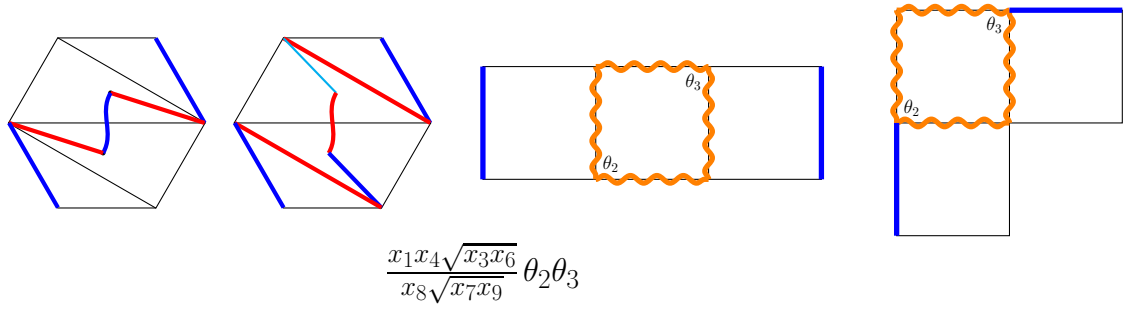
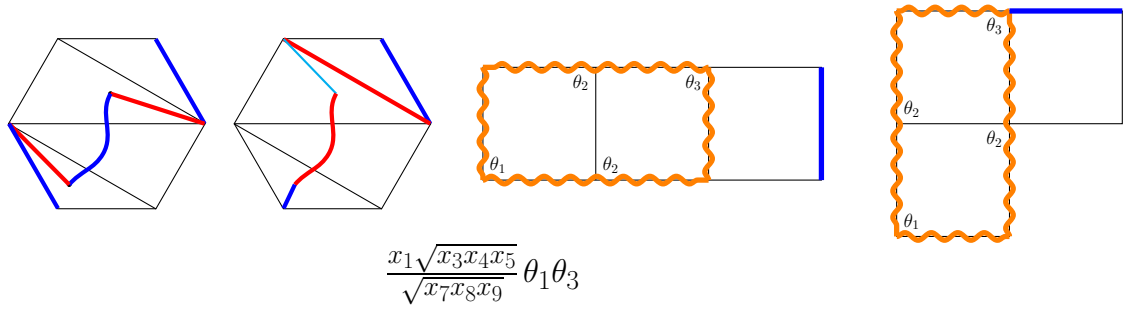
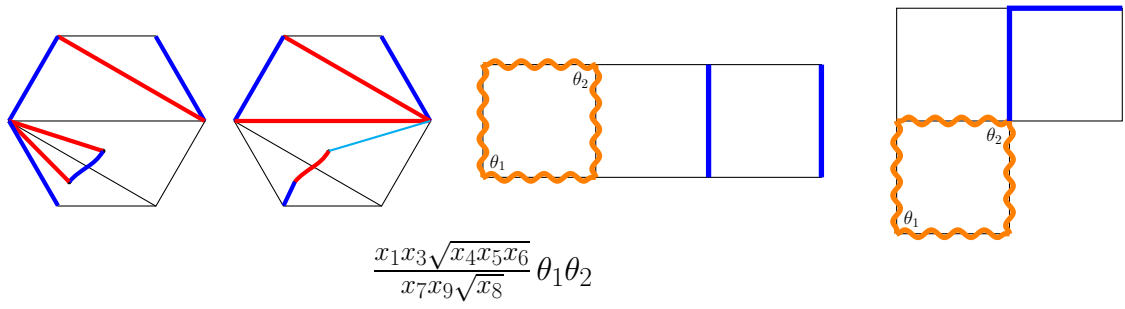
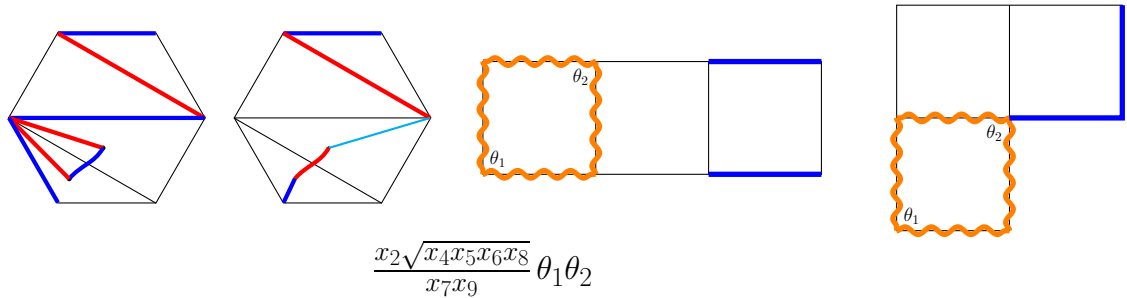
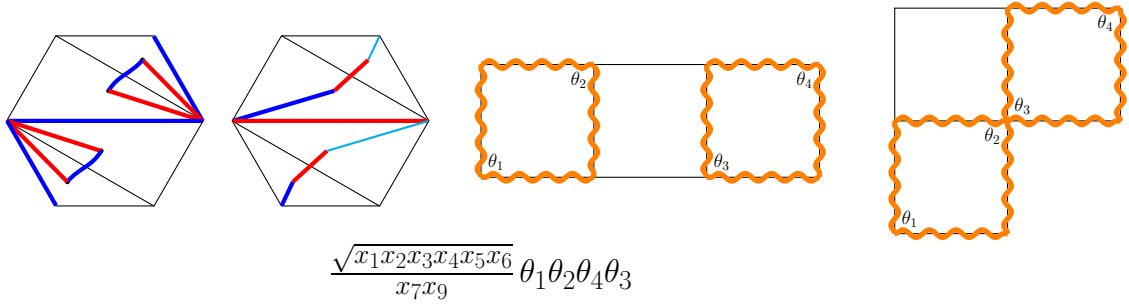
\input{example_main_theorem.tex}
\end{example}

\section{Super Fibonacci Numbers from Annuli and Once-Punctured Tori}

In this section, we generalize our formulas to two special classes of surfaces, the annuli and the once-punctured tori, which gives rise to a sequence of 
Grassmann numbers which we call \emph{super Fibonacci numbers}.

We first look at the decorated super-Teichm\"{u}ller space of a once punctured torus, 
which has been systematically studied in \cite{mcshane}\footnote{The third author also studied this situation
in an undergraduate research project, which motivates much of the current section.\label{footnote:sylvester}}.
An ideal triangulation $\Delta$ of a once punctured torus creates two different triangles on the surface.
This can be seen by presenting a torus by a piece of its universal cover --- a rectangle, and letting the four `vertices' of the rectangle be the puncture. 
Then an ideal triangulation decomposes the surface into two triangles (one white, the other gray), using three arcs:

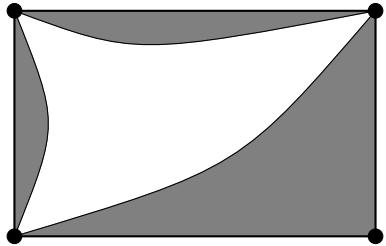
\begin{figure}[h!]
\centering
\begin{tikzpicture}[scale=0.6]
	\coordinate (a) at (0,0);
	\coordinate (b) at (0,5);
	\coordinate (c) at (8,5);
	\coordinate (d) at (8,0);

	\filldraw[color=gray]  (a)..controls(1,2.5)..(b)--cycle;
	\draw (a)..controls(1,2.5)..(b);
	
	\filldraw[color=gray] (b)..controls(2.7,4)..(c)--cycle;
	\draw (b)..controls(2.7,4)..(c);
	
	\filldraw[color=gray] (c)..controls(5,1.5)..(a)--(d);
	\draw(c)..controls(5,1.5)..(a);

	\draw[thick] (a)--(b)--(c)--(d)--cycle;
	
	\node ()[scale=0.5,draw,circle,fill=black] at (a) {};
	\node ()[scale=0.5,draw,circle,fill=black] at (b) {};
	\node ()[scale=0.5,draw,circle,fill=black] at (c) {};
	\node ()[scale=0.5,draw,circle,fill=black] at (d) {};
\end{tikzpicture}
\end{figure} 

The quadrilaterals where a flip is taking place can be realized by
taking $a=c$ and $b=d$ in \Cref{fig:super_ptolemy}. 
We also choose a spin structure such
that the edges around both triangles are oriented cyclically as in \Cref{fig:initial_spin}.
The super Ptolemy transformation in \Cref{eqn:super_ptolemy_lambda,eqn:super_ptolemy_mu_right,eqn:super_ptolemy_mu_left} 
now takes the following form (see \Cref{fig:initial_spin} for the edge labels).
\begin{align}
    ef      &= a^2+b^2+ab \, \sigma \theta \label{eq:super_lambda_torus} \\
    \sigma' &= {{b\sigma-a\theta}\over{a^2+b^2}}
    \quad \text{ and }\quad
    \theta' = {{b\theta+a\sigma}\over{a^2+b^2}} \label{eq:muinvars_torus}
\end{align}

\begin{figure}[h]

\centering

\begin{tikzpicture}[scale=0.7, baseline, thick,every node/.style={sloped,allow upside down}]

\draw (0,0)--(3,0)--(60:3)--cycle;

\draw (0,0)--(3,0)--(-60:3)--cycle;

\draw (0,0)--node {\midarrow} (3,0);
\draw (3,0)--node {\midarrow} (60:3);
\draw (60:3)--node {\midarrow} (0,0);

\draw (3,0)--node {\midarrow} (-60:3);
\draw (-60:3)--node {\midarrow} (0,0);

\draw node[above] at (70:1.5){$a$};
\draw node[above] at (30:2.8){$b$};
\draw node[below] at (-30:2.8){$a$};
\draw node[below] at (-70:1.5){$b$};
\draw node[above] at (1,-0.12){$e$};

\draw node at (1.5,1){$\theta$};

\draw node at (1.5,-1){$\sigma$};

\end{tikzpicture}
\begin{tikzpicture}[baseline]

\draw[->, thick](0,0)--(1,0);

\node[above]  at (0.5,0) {};

\end{tikzpicture}
\begin{tikzpicture}[scale=0.7, baseline, thick,every node/.style={sloped,allow upside down}]
\draw (0,0)--(60:3)--(-60:3)--cycle;
\draw (3,0)--(60:3)--(-60:3)--cycle;

\draw (1.5,-2) --node {\midarrow} (1.5,2);

\draw (60:3)--node {\midarrow} (3,0);
\draw (60:3)--node {\midarrow} (0,0);

\draw (3,0)--node {\midarrow} (-60:3);
\draw (0,0)--node {\midarrow} (-60:3);

\draw node[above] at (70:1.5){$a$};
\draw node[above] at (30:2.8){$b$};
\draw node[below] at (-30:2.8){$a$};
\draw node[below] at (-70:1.5){$b$};
\draw node[left] at (1.7,1){$f$};

\draw node at (0.8,0){$\theta'$};

\draw node at (2.2,0){$\sigma'$};

\end{tikzpicture}
\caption{Cyclic orientation is preserved under flips.}
\label{fig:initial_spin}
\end{figure}
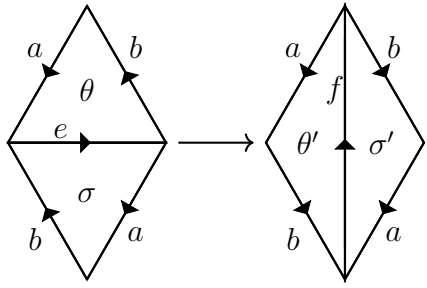

Notice that the ideal arcs surrounding both $\theta$ and $\theta'$ are counterclockwise, and those surrounding both $\sigma$ and $\sigma'$ are clockwise.
In other words, we obtain the same oriented triangulation after the flip as before.
This means flipping any of the three arcs will always result in the same relations \ref{eq:super_lambda_torus} and \ref{eq:muinvars_torus},
and will always result in the same oriented triangulation. \Cref{fig:other_flips} illustrates the result of
flipping the other two edges ($a$ and $b$).

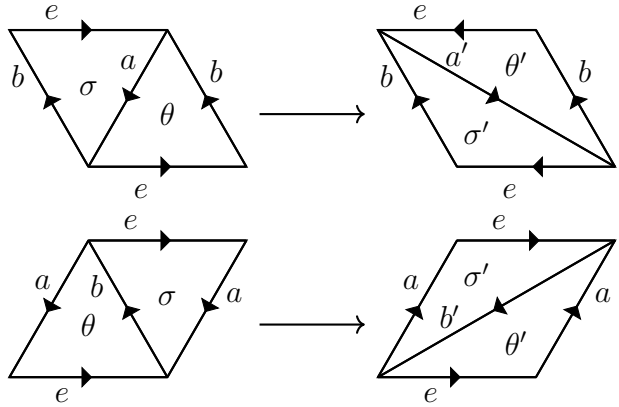
\begin{figure}[h]
\centering
\begin{tikzpicture}[scale=0.7, baseline, thick,every node/.style={sloped,allow upside down}]

\draw (0,0)--(3,0)--(60:3)--(120:3)--cycle;
\draw (0,0)--(60:3);

\draw (0,0)--node {\midarrow} (3,0);
\draw (3,0)--node {\midarrow} (60:3);
\draw (60:3)--node {\midarrow} (0,0);
\draw (120:3)--node {\midarrow} (60:3);
\draw (0,0)--node {\midarrow} (120:3);

\draw node[above] at (65:1.8){$a$};
\draw node[above] at (30:2.8){$b$};
\draw node[below] at (1,-0.12){$e$};
\draw node[above] at (-0.7,2.6){$e$};
\draw node[left] at (120:2){$b$};

\draw node at (1.5,1){$\theta$};
\draw node at (0,1.5){$\sigma$};

\draw[->, thick](3.25,1)--(5.25,1);

\begin {scope}[shift={(7,0)}]
    \draw (0,0)--(3,0)--(60:3)--(120:3)--cycle;
    \draw (120:3) -- (3,0);

    \draw (3,0)--node {\midarrow} (0,0);
    \draw (3,0)--node {\midarrow} (60:3);
    \draw (120:3)--node {\midarrow} (3,0);
    \draw (60:3)--node {\midarrow} (120:3);
    \draw (0,0)--node {\midarrow} (120:3);

    \draw node[above] at (0,1.7){$a'$};
    \draw node[above] at (30:2.8){$b$};
    \draw node[below] at (1,-0.12){$e$};
    \draw node[above] at (-0.7,2.6){$e$};
    \draw node[left] at (120:2){$b$};

    \draw node at (60:2.25){$\theta'$};
    \draw node at (60:0.75){$\sigma'$};
\end {scope}

\begin {scope}[shift={(1.5,-4)}]
    \draw (0,0)--(-3,0)--(120:3)--(60:3)--cycle;
    \draw (0,0)--(120:3);

    \draw (-3,0)--node {\midarrow} (0,0);
    \draw (0,0)--node {\midarrow} (120:3);
    \draw (120:3)--node {\midarrow} (-3,0);
    \draw (120:3)--node {\midarrow} (60:3);
    \draw (60:3)--node {\midarrow} (0,0);

    \draw node[right] at (60:1.8){$a$};
    \draw node[left] at (-2,1.8){$a$};
    \draw node[below] at (-2,0){$e$};
    \draw node[above] at (-0.7,2.6){$e$};
    \draw node[left] at (120:2){$b$};

    \draw node at (-1.5,1){$\theta$};
    \draw node at (0,1.5){$\sigma$};
\end {scope}

\draw[->, thick](3.25,-3)--(5.25,-3);

\begin {scope}[shift={(8.5,-4)}]
    \draw (0,0)--(-3,0)--(120:3)--(60:3)--cycle;
    \draw (-3,0)--(60:3);

    \draw (-3,0)--node {\midarrow} (0,0);
    \draw (60:3)--node {\midarrow} (-3,0);
    \draw (-3,0)--node {\midarrow} (120:3);
    \draw (120:3)--node {\midarrow} (60:3);
    \draw (0,0)--node {\midarrow} (60:3);

    \draw node[right] at (60:1.8){$a$};
    \draw node[left] at (-2,1.8){$a$};
    \draw node[below] at (-2,0){$e$};
    \draw node[above] at (-0.7,2.6){$e$};
    \draw node        at (145:2){$b'$};

    \draw node at (120:0.75){$\theta'$};
    \draw node at (120:2.25){$\sigma'$};
\end {scope}

\end{tikzpicture}
\caption{Effect of flipping edges $a$ and $b$}
\label{fig:other_flips}
\end{figure}

Since any sequence of flips results in the same oriented triangulation, then for any edge, $\sigma'$ will always be on the right and $\theta'$ on the left.
Since $\sigma \theta = \sigma' \theta'$, the Ptolemy relation will always take the form
\begin{equation} \label{eq:super_markov_mutation}
    ef = a^2 + b^2 + ab \, \varepsilon
\end{equation}

where $\varepsilon = \sigma \theta$.  We can therefore forget the original $\mu$-invariants and just use the variable $\varepsilon$.

In the classical (non-super) case, applying the Ptolemy transformation on two of the edges alternately gives rises to odd-indexed Fibonacci numbers. 
In the same spirit, starting with one of the four oriented triangulations of the once punctured torus illustrated in \Cref{fig:other_flips}, 
setting $a = Z_1$, $b= Z_2$, $e = 1$, and flipping the edges $a$ and $b$ alternatively
will give us a sequence of elements of the super algebra, $\{Z_m\}$ satisfying the recurrence relation:
\begin{equation} \label{eqn:super_sequence}
    Z_{m} Z_{m-2} = Z_{m-1}^2 + Z_{m-1} \varepsilon + 1
\end{equation}
for $m \geq 3$.

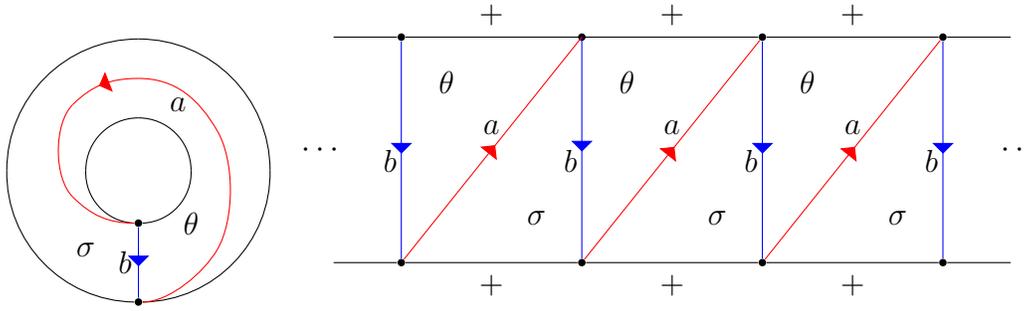
\begin{figure}
	\begin{tikzpicture}[scale=0.7,every node/.style={sloped,allow upside down}]
	
		\node [] (0) at (0, 1) {};
		\node [circle,fill=black,inner sep=1pt] (1) at (0, -1) {};
		\node [] (2) at (0, 2.5) {};
		\node [] (4) at (-1, 0) {};
		\node [] (5) at (1, 0) {};
		\node [ ] (6) at (2.5, 0) {};
		\node [ ] (7) at (-2.5, 0) {};
		\node [ ] (8) at (1.5, -1.5) {};
		\node [ ] (9) at (1.5, 0.75) {};
		\node [ ] (10) at (0, 1.75) {};
		\node [ ] (11) at (-1.25, 1.25) {};
		\node [ ] (13) at (-1.25, -0.5) {};
		\node [ ] (14) at (-0.25, -1.75) {$b$};
		\node [circle,fill=black,inner sep=1pt] (17) at (0, -2.5) {};
		\node [ ] (18) at (0.75, 1.25) {$a$};
		
		\node () at (-1,-1.5) {$\sigma$};
		\node () at (1,-1) {$\theta$};

		\draw [bend left=45] (4.center) to (0.center);
		\draw [bend left=45] (0.center) to (5.center);
		\draw [bend left=45] (5.center) to (1);
		\draw [bend right=45] (4.center) to (1);
		\draw [bend left=45] (7.center) to (2.center);
		\draw [bend left=45] (2.center) to (6.center);
		\draw [color=red, in=-60, out=60, looseness=0.75] (8.center) to (9.center);
		\draw [style=red, in=0, out=120] (9.center) to (10.center);
		\draw [style=red, in=45, out=180] (10.center) to node {\midarrow} (11.center);
		\draw [style=red, in=-180, out=-45] (13.center) to (1);
		\draw [color=red, in=135, out=-135, looseness=0.75] (11.center) to (13.center);
		\draw [bend right=45] (7.center) to (17);
		\draw [bend right=45] (17) to (6.center);
		\draw [color=red, in=-120, out=0, looseness=0.75] (17) to (8.center);
		\draw [color=blue] (1) -- node {\midarrow} (17);
	
\end{tikzpicture}
\begin{tikzpicture}[scale=0.6,every node/.style={sloped,allow upside down}]

		\node [circle,fill=black,inner sep=1pt] (0) at (-6, 3) {};
		\node [circle,fill=black,inner sep=1pt] (1) at (-6, -2) {};
		\node [circle,fill=black,inner sep=1pt] (2) at (-2, 3) {};
		\node [circle,fill=black,inner sep=1pt] (3) at (-2, -2) {};
		\node [circle,fill=black,inner sep=1pt] (4) at (2, 3) {};
		\node [circle,fill=black,inner sep=1pt] (5) at (2, -2) {};
		\node [circle,fill=black,inner sep=1pt] (6) at (6, 3) {};
		\node [circle,fill=black,inner sep=1pt] (7) at (6, -2) {};
		\node [] (8) at (7.5, 3) {};
		\node [] (9) at (7.5, -2) {};
		\node [] (10) at (-7.5, 3) {};
		\node [] (11) at (-7.5, -2) {};
		\node [] (12) at (-2, 3) {};
		\node [] (13) at (-5, 2) {$\theta$};
		\node [] (14) at (-3, -1) {$\sigma$};
		\node [] (15) at (-1, 2) {$\theta$};
		\node [] (16) at (1, -1) {$\sigma$};
		\node [] (17) at (3, 2) {$\theta$};
		\node [] (18) at (5, -1) {$\sigma$};
		\node [] (19) at (-7.75, 0.5) {$\cdots$};
		\node [] (20) at (7.75, 0.5) {$\cdots$};
		\node [] (21) at (-6.25, 0.25) {$b$};
		\node [] (22) at (-4, 1) {$a$};
		\node [] (23) at (-2.25, 0.25) {$b$};
		\node [] (24) at (0, 1) {$a$};
		\node [] (25) at (1.75, 0.25) {$b$};
		\node [] (26) at (4, 1) {$a$};
		\node [] (27) at (5.75, 0.25) {$b$};
		\node [] (28) at (-4, 3.5) {$+$};
		\node [] (29) at (0, 3.5) {$+$};
		\node [] (30) at (4, 3.5) {$+$};
		\node [] (31) at (-4, -2.5) {$+$};
		\node [] (32) at (0, -2.5) {$+$};
		\node [] (33) at (4, -2.5) {$+$};

		\draw (0) to (12.center);
		\draw (4) to (12.center);
		\draw (4) to (6);
		\draw (3) to (1);
		\draw (3) to (5);
		\draw (5) to (7);
		\draw [style=blue] (0) to node {\midarrow}  (1);
		\draw [style=blue] (12.center) to node {\midarrow} (3);
		\draw [style=blue] (4) to node {\midarrow} (5);
		\draw [style=blue] (6) to node {\midarrow} (7);
		\draw [style=red] (1) to node {\midarrow} (12.center);
		\draw [style=red] (3) to node {\midarrow} (4);
		\draw [style=red] (5) to node {\midarrow} (6);
		\draw (6) to (8.center);
		\draw (7) to (9.center);
		\draw (10.center) to (0);
		\draw (11.center) to (1);

\end{tikzpicture}

\caption{Triangulation of a marked annulus and its universal cover.}
\label{fig:annulus}
\end{figure}

One can derive \cref{eqn:super_sequence} in an alternative way, using the decorated super-Teichm\"uller space of an annulus.
Consider an annulus with two marked points, one on each boundary component, with orientation as depicted in \Cref{fig:annulus}. 
Similar to the torus case, if we set the $\lambda$-lengths of the boundary arcs to be $1$ and set $a=Z_1$ and $b=Z_2$, 
then flipping the edges $a$ and $b$ alternately gives the same sequence $\{Z_m\}$ satisfying the same recurrence \cref{eqn:super_sequence}.

Similar to the case of the once-punctured torus, the relative spin structure is preserved in both orders: either flipping $a$ then $b$ or flipping $b$ first. However, the two orders work in a slightly different way. Flipping $a$ first would only alter the boundary orientations and leave the interior ones unchanged, while flipping $b$ first does change the orientations on the interior edges. In particular, see how the spin structure is affected by flips in Figures \ref{fig:annulus_flips1} and \ref{fig:annulus_flips2}.

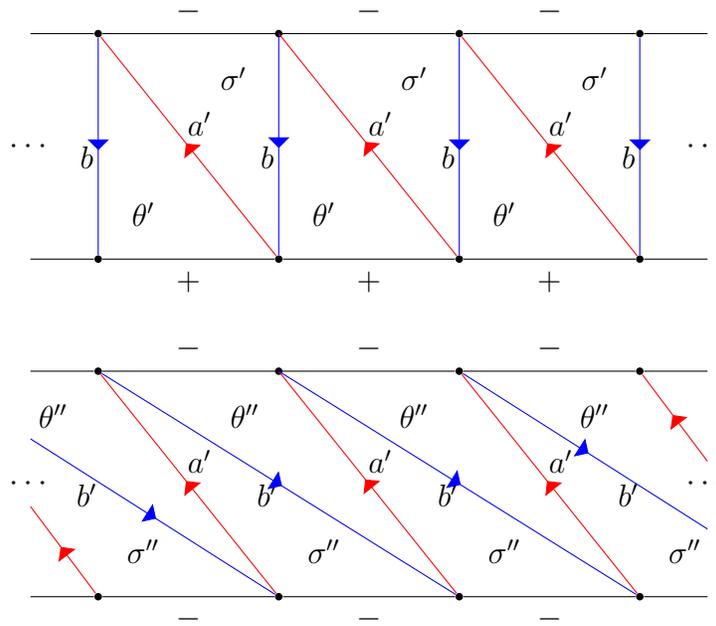
\begin{figure}

\begin{tikzpicture}[scale=0.6, every node/.style={sloped,allow upside down}]
		\node [circle,fill=black,inner sep=1pt] (0) at (-6, 3) {};
		\node [circle,fill=black,inner sep=1pt] (1) at (-6, -2) {};
		\node [circle,fill=black,inner sep=1pt] (2) at (-2, 3) {};
		\node [circle,fill=black,inner sep=1pt] (3) at (-2, -2) {};
		\node [circle,fill=black,inner sep=1pt] (4) at (2, 3) {};
		\node [circle,fill=black,inner sep=1pt] (5) at (2, -2) {};
		\node [circle,fill=black,inner sep=1pt] (6) at (6, 3) {};
		\node [circle,fill=black,inner sep=1pt] (7) at (6, -2) {};
		\node [] (8) at (7.5, 3) {};
		\node [] (9) at (7.5, -2) {};
		\node [] (10) at (-7.5, 3) {};
		\node [] (11) at (-7.5, -2) {};
		\node [] (12) at (-2, 3) {};
		\node [] (13) at (-5, -1) {$\theta'$};
		\node [] (14) at (-3, 2) {$\sigma'$};
		\node [] (15) at (-1, -1) {$\theta'$};
		\node [] (16) at (1, 2) {$\sigma'$};
		\node [] (17) at (3, -1) {$\theta'$};
		\node [] (18) at (5, 2) {$\sigma'$};
		\node [] (19) at (-7.5, 0.5) {$\cdots$};
		\node [] (20) at (7.5, 0.5) {$\cdots$};
		\node [] (21) at (-6.25, 0.25) {$b$};
		\node [] (22) at (-3.75, 1) {$a'$};
		\node [] (23) at (-2.25, 0.25) {$b$};
		\node [] (24) at (0.25, 1) {$a'$};
		\node [] (25) at (1.75, 0.25) {$b$};
		\node [] (26) at (4.25, 1) {$a'$};
		\node [] (27) at (5.75, 0.25) {$b$};
		\node [] (28) at (-4, 3.5) {$-$};
		\node [] (29) at (0, 3.5) {$-$};
		\node [] (30) at (4, 3.5) {$-$};
		\node [] (31) at (-4, -2.5) {$+$};
		\node [] (32) at (0, -2.5) {$+$};
		\node [] (33) at (4, -2.5) {$+$};
		\draw (0) to (12.center);
		\draw (4) to (12.center);
		\draw (4) to (6);
		\draw (3) to (1);
		\draw (3) to (5);
		\draw (5) to (7);
		\draw [style=blue] (0) to node {\midarrow}(1);
		\draw [style=blue] (12.center) to node {\midarrow}(3);
		\draw [style=blue] (4) to node {\midarrow}(5);
		\draw [style=blue] (6) to node {\midarrow}(7);
		\draw (6) to (8.center);
		\draw (7) to (9.center);
		\draw (10.center) to (0);
		\draw (11.center) to (1);
		\draw [style=red] (3) to node {\midarrow}(0);
		\draw [style=red] (5) to node {\midarrow}(12.center);
		\draw [style=red] (7) to node {\midarrow}(4);
\end{tikzpicture}

\vspace{0.3cm}

\begin{tikzpicture}[scale=0.6,every node/.style={sloped,allow upside down}]

		\node [circle,fill=black,inner sep=1pt] (0) at (-6, 3) {};
		\node [circle,fill=black,inner sep=1pt] (1) at (-6, -2) {};
		\node [circle,fill=black,inner sep=1pt] (2) at (-2, 3) {};
		\node [circle,fill=black,inner sep=1pt] (3) at (-2, -2) {};
		\node [circle,fill=black,inner sep=1pt] (4) at (2, 3) {};
		\node [circle,fill=black,inner sep=1pt] (5) at (2, -2) {};
		\node [circle,fill=black,inner sep=1pt] (6) at (6, 3) {};
		\node [circle,fill=black,inner sep=1pt] (7) at (6, -2) {};
		\node [] (8) at (7.5, 3) {};
		\node [] (9) at (7.5, -2) {};
		\node [] (10) at (-7.5, 3) {};
		\node [] (11) at (-7.5, -2) {};
		\node [] (12) at (-2, 3) {};
		\node [] (13) at (-7, 2) {$\theta''$};
		\node [] (14) at (-1, -1) {$\sigma''$};
		\node [] (15) at (-2.75, 2) {$\theta''$};
		\node [] (16) at (3, -1) {$\sigma''$};
		\node [] (17) at (1, 2) {$\theta''$};
		\node [] (18) at (7, -1) {$\sigma''$};
		\node [] (19) at (-7.5, 0.5) {$\cdots$};
		\node [] (20) at (7.5, 0.5) {$\cdots$};
		\node [] (21) at (-6.25, 0.25) {$b'$};
		\node [] (22) at (-3.75, 1) {$a'$};
		\node [] (23) at (-2.25, 0.25) {$b'$};
		\node [] (24) at (0.25, 1) {$a'$};
		\node [] (25) at (1.75, 0.25) {$b'$};
		\node [] (26) at (4.25, 1) {$a'$};
		\node [] (27) at (5.75, 0.25) {$b'$};
		\node [] (28) at (-4, 3.5) {$-$};
		\node [] (29) at (0, 3.5) {$-$};
		\node [] (30) at (4, 3.5) {$-$};
		\node [] (31) at (-4, -2.5) {$-$};
		\node [] (32) at (0, -2.5) {$-$};
		\node [] (33) at (4, -2.5) {$-$};
		\node [] (35) at (7.5, -0.5) {};
		\node [] (36) at (-7.5, 1.5) {};
		\node [] (37) at (7.5, 1) {};
		\node [] (38) at (-7.5, 0) {};
		\node [] (39) at (5, 2) {$\theta''$};
		\node [] (40) at (-5, -1) {$\sigma''$};

		\draw (0) to (12.center);
		\draw (4) to (12.center);
		\draw (4) to (6);
		\draw (3) to (1);
		\draw (3) to (5);
		\draw (5) to (7);
		\draw (6) to (8.center);
		\draw (7) to (9.center);
		\draw (10.center) to (0);
		\draw (11.center) to (1);
		\draw [style=red] (3) to node {\midarrow} (0);
		\draw [style=red] (5) to node {\midarrow}(12.center);
		\draw [style=red] (7) to node {\midarrow}(4);
		\draw [style=blue] (0) to node {\midarrow}(5);
		\draw [style=blue] (12.center) to node {\midarrow}(7);
		\draw [style=blue] (4) to node {\midarrow}(35.center);
		\draw [style=blue] (36.center) to node {\midarrow}(3);
		\draw [style=red] (37.center) to node {\midarrow}(6);
		\draw [style=red] (1) to node {\midarrow}(38.center);

\end{tikzpicture}

\caption{Flipping $a$ (Top) then $b$ (Bottom) of a marked annulus, labeled as above, only changes orientation on the boundaries.}
\label{fig:annulus_flips1}
\end{figure}

\begin{figure}

	\begin{tikzpicture}[scale=0.6,every node/.style={sloped,allow upside down}]

		\node [circle,fill=black,inner sep=1pt] (0) at (-6, 3) {};
		\node [circle,fill=black,inner sep=1pt] (1) at (-6, -2) {};
		\node [circle,fill=black,inner sep=1pt] (2) at (-2, 3) {};
		\node [circle,fill=black,inner sep=1pt] (3) at (-2, -2) {};
		\node [circle,fill=black,inner sep=1pt] (4) at (2, 3) {};
		\node [circle,fill=black,inner sep=1pt] (5) at (2, -2) {};
		\node [circle,fill=black,inner sep=1pt] (6) at (6, 3) {};
		\node [circle,fill=black,inner sep=1pt] (7) at (6, -2) {};
		\node [] (8) at (7.5, 3) {};
		\node [] (9) at (7.5, -2) {};
		\node [] (10) at (-7.5, 3) {};
		\node [] (11) at (-7.5, -2) {};
		\node [] (12) at (-2, 3) {};
		\node [] (13) at (-5, 2) {$\theta'$};
		\node [] (14) at (-3, -1) {$\sigma'$};
		\node [] (15) at (-1, 2) {$\theta'$};
		\node [] (16) at (1, -1) {$\sigma'$};
		\node [] (17) at (3, 2) {$\theta'$};
		\node [] (18) at (5, -1) {$\sigma'$};
		\node [] (19) at (-7.5, 0.5) {$\cdots$};
		\node [] (20) at (7.5, 0.5) {$\cdots$};
		\node [] (21) at (-6, 1) {$b'$};
		\node [] (22) at (-4, 1) {$a$};
		\node [] (23) at (-2, 1) {$b'$};
		\node [] (24) at (0, 1) {$a$};
		\node [] (25) at (2, 1) {$b'$};
		\node [] (26) at (4, 1) {$a$};
		\node [] (27) at (6, 1) {$b'$};
		\node [] (28) at (-4, 3.5) {$+$};
		\node [] (29) at (0, 3.5) {$+$};
		\node [] (30) at (4, 3.5) {$+$};
		\node [] (31) at (-4, -2.5) {$+$};
		\node [] (32) at (0, -2.5) {$+$};
		\node [] (33) at (4, -2.5) {$+$};
		\node [] (34) at (-7.5, -0.5) {};
		\node [] (35) at (7.5, 1.5) {};
		\node [] (36) at (-7, -1) {$\sigma'$};
		\node [] (37) at (7, 2) {$\theta'$};
		\node [] (39) at (7.5, 0) {};
		\node [] (40) at (-7.5, 1) {};

		\draw (0) to (12.center);
		\draw (4) to (12.center);
		\draw (4) to (6);
		\draw (3) to (1);
		\draw (3) to (5);
		\draw (5) to (7);
		\draw [style=red] (1) to node {\midrevarrow} (12.center);
		\draw [style=red] (3) to node {\midrevarrow} (4);
		\draw [style=red] (5) to node {\midrevarrow} (6);
		\draw (6) to (8.center);
		\draw (7) to (9.center);
		\draw (10.center) to (0);
		\draw (11.center) to (1);
		\draw [style=blue] (34.center) to node {\midarrow}  (12.center);
		\draw [style=blue] (1) to node {\midarrow} (4);
		\draw [style=blue] (3) to node {\midarrow} (6);
		\draw [style=blue] (5) to node {\midarrow} (35.center);
		\draw [style=red] (7) to  (39.center);
		\draw [style=red]  (40.center) to (0);

\end{tikzpicture}

\vspace{0.2cm}

\begin{tikzpicture}[scale=0.6,every node/.style={sloped,allow upside down}]

		\node [circle,fill=black,inner sep=1pt] (0) at (-6, 3) {};
		\node [circle,fill=black,inner sep=1pt] (1) at (-6, -2) {};
		\node [circle,fill=black,inner sep=1pt] (2) at (-2, 3) {};
		\node [circle,fill=black,inner sep=1pt] (3) at (-2, -2) {};
		\node [circle,fill=black,inner sep=1pt] (4) at (2, 3) {};
		\node [circle,fill=black,inner sep=1pt] (5) at (2, -2) {};
		\node [circle,fill=black,inner sep=1pt] (6) at (6, 3) {};
		\node [circle,fill=black,inner sep=1pt] (7) at (6, -2) {};
		\node [] (8) at (7.5, 3) {};
		\node [] (9) at (7.5, -2) {};
		\node [] (10) at (-7.5, 3) {};
		\node [] (11) at (-7.5, -2) {};
		\node [] (12) at (-2, 3) {};
		\node [] (13) at (-5, 2.5) {$\theta''$};
		\node [] (14) at (-2, -1) {$\sigma''$};
		\node [] (15) at (-1, 2.5) {$\theta''$};
		\node [] (16) at (2, -1) {$\sigma''$};
		\node [] (17) at (3, 2.5) {$\theta''$};
		\node [] (18) at (6, -1) {$\sigma''$};
		\node [] (19) at (-7.5, 0.5) {$\cdots$};
		\node [] (20) at (7.5, 0.5) {$\cdots$};
		\node [] (21) at (-6.75, 0.25) {$b'$};
		\node [] (22) at (-2.5, 1.5) {$a'$};
		\node [] (23) at (-3, 0.25) {$b'$};
		\node [] (24) at (1.5, 1.5) {$a'$};
		\node [] (25) at (1, 0.25) {$b'$};
		\node [] (26) at (5.5, 1.5) {$a'$};
		\node [] (27) at (5, 0.25) {$b'$};
		\node [] (28) at (-4, 3.5) {$+$};
		\node [] (29) at (0, 3.5) {$+$};
		\node [] (30) at (4, 3.5) {$+$};
		\node [] (31) at (-4, -2.5) {$+$};
		\node [] (32) at (0, -2.5) {$+$};
		\node [] (33) at (4, -2.5) {$+$};
		\node [] (34) at (-7.5, -0.5) {};
		\node [] (35) at (7.5, 1.5) {};
		\node [] (36) at (-6, -1) {$\sigma''$};
		\node [] (37) at (7, 2.5) {$\theta''$};
		\node [] (39) at (7.5, -1) {};
		\node [] (40) at (-7.5, 1) {};
		\node [] (41) at (-7.5, -1) {};
		\node [] (42) at (7.5, 2) {};
		\node [] (43) at (7.5, 0.25) {};
		\node [] (44) at (-7.5, 0.75) {};
		\node [] (45) at (-6.5, 1.5) {$a'$};
		\node [] (46) at (-7.5, 2) {};

		\draw (0) to (12.center);
		\draw (4) to (12.center);
		\draw (4) to (6);
		\draw (3) to (1);
		\draw (3) to (5);
		\draw (5) to (7);
		\draw (6) to (8.center);
		\draw (7) to (9.center);
		\draw (10.center) to (0);
		\draw (11.center) to (1);
		\draw [style=blue] (34.center) to node {\midrevarrow} (12.center);
		\draw [style=blue] (1) to node {\midrevarrow} (4);
		\draw [style=blue] (3) to node {\midrevarrow} (6);
		\draw [style=blue] (5) to node {\midrevarrow} (35.center);
		\draw [style=red] (41.center) to node {\midarrow} (4);
		\draw [style=red] (1) to node {\midarrow} (6);
		\draw [style=red] (3) to node {\midarrow} (42.center);
		\draw [style=red] (5) to (43.center);
		\draw [style=blue] (7) to (39.center);
		\draw [style=red] (44.center) to (12.center);
		\draw [style=blue] (46.center) to (0);
\end{tikzpicture}

\caption{On the other hand, flipping $b$ (Top) then $a$ (Bottom) of a marked annulus, labeled as above, changes orientation of internal arcs.}
\label{fig:annulus_flips2}
\end{figure}
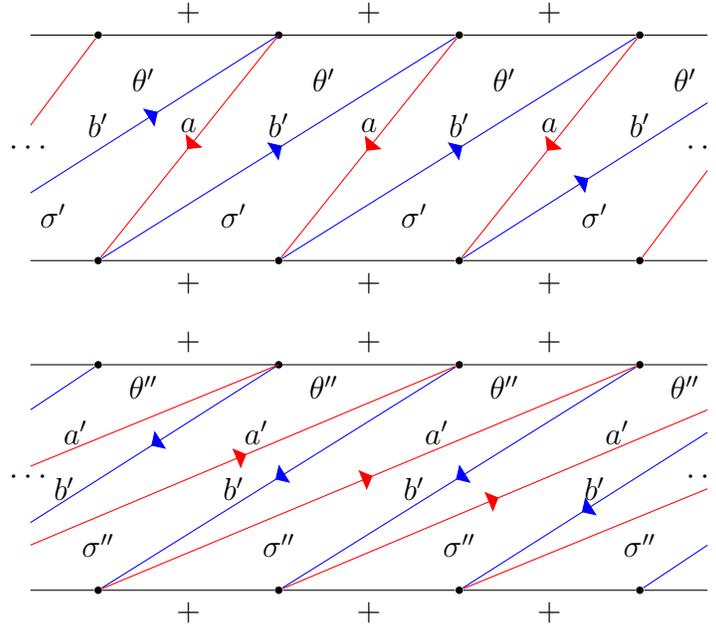

For the remainder of this section, we will give an explicit solution of recurrence (\ref{eqn:super_sequence}) in the special case that $Z_1=a=1$ and $Z_2=b=1$.
The solution is given by the double dimer partition functions of snake graphs as discussed earlier.
Although we have only defined snake graphs for arcs in a polygon, the construction of \Cref{def:Ggamma}
can easily be generalized to arcs in surfaces with nontrivial topology as in \cite{ms10}. The arcs obtained
from the flip sequence described above yield snake graphs with words $W(G) = R^{2n+1} = RR \cdots R$. In other words,
the resulting snake graphs are horizontal rows of tiles (with an odd number of tiles). Since the arcs cross the same
triangles multiple times, the $\mu$-invariants in the corners of the tiles will repeat.  With this construction in mind, we define the following.

\begin{definition} \label{def:G_m}
Let $G_m$ be the snake graph with $m$ tiles in a horizontal row, 
where all edges have weight $1$, and all tiles have the same two $\mu$-invariants $\sigma$ and $\theta$:

\begin {center}
\begin {tikzpicture}
    \draw (0,0) -- (3,0) -- (3,1) -- (0,1) -- cycle;
    \draw (1,0) -- (1,1);
    \draw (2,0) -- (2,1);

    \draw (3.5,0.5) node {$\cdots$};

    \draw (0.2,0.2) node {\tiny $\sigma$};
    \draw (0.8,0.8) node {\tiny $\theta$};
    \draw (1.2,0.2) node {\tiny $\theta$};
    \draw (1.8,0.8) node {\tiny $\sigma$};
    \draw (2.2,0.2) node {\tiny $\sigma$};
    \draw (2.8,0.8) node {\tiny $\theta$};
\end {tikzpicture}
\end {center}
\end{definition}

\begin {remark} \label{rem:snake}
    This snake graph $G_m$ from \Cref{def:G_m} can also be obtained from a triangulated polygon
    by specializing some of the variables. Specifically, take a polygon with a ``\emph{zig-zag}'' triangulation (see e.g. \Cref{ex:expansions} for the $m=3$ case).
    Then the construction from \Cref{def:Ggamma} will give a snake graph with the word $W(G) = RR \cdots R$,
    and the snake graph will be a horizontal row of boxes. Finally, setting $\theta_1 = \theta_3 = \cdots = \theta_{2n-1}$
    and $\theta_2 = \theta_4 = \cdots = \theta_{2n}$ will give the correct $\mu$-invariants.
\end {remark}

Let $x_m$ be the number of double dimer covers of $G_m$ which do not include cycles. These are in bijection with ordinary (i.e. not double) dimer covers.
Hence the sequence of $x_m$'s coincide with every-other Fibonacci number (see e.g. Section 2 of \cite{MusPropp}).  More precisely, if we index the Fibonacci numbers $F_n$ by 
$F_1 = F_2 = 1$ and $F_n = F_{n-1} + F_{n-2}$ for $n\geq 3$, then $x_m = F_{m+2}$ for $m\geq 1$.  Let $y_m$ be the number of double dimer covers of $G_m$ which include a single cycle of odd length. 
Finally, we define $p_m = x_m + y_m \varepsilon$, where $\varepsilon = \sigma \theta$. We will denote the odd-indexed $p_m$'s by
$z_n = p_{2n-5}$ and denote the even-indexed $p_m$'s by $w_n = p_{2n-4}$.  See Remark \ref{rem:w_n} below.
}

\begin{remark} \label{rem:y_m}
Note that $y_m$ is not defined as the total number of double dimer covers containing a cycle, but only those with a single 
cycle of odd length. Because every tile of $G_m$ contains the same $\mu$-invariants $\sigma$ and $\theta$, a cycle of even length 
would contain a factor $\sigma^2$ or $\theta^2$, and so such a double dimer cover would have weight zero. Similarly, 
a double dimer cover with more than one cycle would also contain multiple factors of $\sigma \theta$. So indeed $p_m = x_m + y_m \varepsilon$ 
is the double dimer partition function for $G_m$.
\end{remark}

\bigskip

Recall that the super $\lambda$-lengths in an annulus (for the flip sequence described earlier) are denoted by the $Z_m$'s.  

\begin {prop} \label{prop:annulus_lambda_lengths}
The super algebra elements $Z_m$'s defined above agree with the $z_m$'s, which, as explained in Remark \ref{rem:y_m}, are the double dimer partition functions for $G_{2m-5}$.  Here, $G_{2m-5}$ is the snake graph with $\mu$-invariants as defined in \Cref{def:G_m}.
\end {prop}

\bigskip

We will prove this proposition in steps using a sequence of lemmas.

\bigskip

\begin {lemma}
    If $m$ is even, then
    \[ p_m = p_{m-1} + p_{m-2} + \left( p_{m-2} + p_{m-4} + \cdots + p_2 + 1 \right) \varepsilon \]
    If $m$ is odd, then
    \[ p_m = p_{m-1} + p_{m-2} + \left( p_{m-2} + p_{m-4} + \cdots + p_1 + 1 \right) \varepsilon \]
\end {lemma}
\begin {proof}
    This is easily seen by the fact that $D(G_m) = D_R(G_m) \cup D_T(G_m) \cup D_{tr}(G_m)$. 
    The fact that the terms in $p_{m-1}$ correspond to the elements of $D_R(G_m)$ follows from \Cref{lem:lemma1}.
    The fact that the terms in $p_{m-2}$ correspond to the elements of $D_T(G_m)$ follows from \Cref{lem:lemma2}.
    Finally, consider the elements of $D_{tr}(G_m)$. They must have a cycle which surrounds the last tile. As mentioned
    earlier, there can only be one cycle, which must have odd length. The weight contributed by the cycle is $\varepsilon$.
    If the cycle surrounds only one tile, then there can by any double dimer cover from $D(G_{m-2})$ on the remaining part of the snake graph.
    If the cycle surrounds three tiles, the remaining part can have any double dimer cover from $D(G_{m-4})$, etc.  When $m$ is odd, these 
    contributions continue until the case of a cycle surrounding all of $G$.  On the other hand, if $m$ is even, the last contribution is the case of a cycle of 
    surrounding $(m-1)$ tiles and the remaining part must be a double dimer on the leftmost vertical edge of $G$.  In either case, there is exactly one such 
    way to complete the double dimer cover beyond the cycle of length $m$ (respectively $(m-1)$).
\end {proof}

\bigskip

\begin {lemma} \label{lem:p_m_simp}
    If $m$ is odd, then
    \[ p_m = p_{m-1} + p_{m-2} + p_{m-1} \varepsilon = (1 + \varepsilon)p_{m-1} + p_{m-2} \]
    If $m$ is even, then
    \[ p_m = p_{m-1} + p_{m-2} + (p_{m-1} - 1) \varepsilon = (1 + \varepsilon)p_{m-1} + p_{m-2} - \varepsilon \]
\end {lemma}
\begin {proof}
    In light of the previous lemma, we only need to check that 
    $(p_{m-2} + p_{m-4} + \cdots  + p_2 + 1)\varepsilon$ is equal to $p_{m-1}\varepsilon$ (resp. $(p_{m-2} + p_{m-4} + \cdots  + p_1 + 1)\varepsilon$ is equal to $(p_{m-1}-1)\varepsilon$)
    when $m$ is odd (resp. even).  By multiplying by $\varepsilon$, all terms in either of these identities already containing an $\varepsilon$ vanish.  Hence, we only need to check that $x_{m-2} + x_{m-4} + \cdots  + x_2 + 1$ is equal to $x_{m-1}$ (resp. $x_{m-2} + x_{m-4} + \cdots  + x_1 + 1$ is equal to $x_{m-1}-1$.  As observed above, the $x_m$'s coincide with the Fibonacci numbers.  We thus prove our desired result via induction using the Fibonacci recurrence.
    
    Consider the base case when $m=3$.  We have $x_1 = F_3 = 2$, hence $x_1 + 1 = 3 = F_4 = x_2$.  Inductively for $m$ odd, 
    $x_{m-2} + (x_{m-4} + x_{m-6} + \cdots + x_1 + 1) = x_{m-2} + x_{m-3} = F_m + F_{m-1} = F_{m+1} = x_{m-1}$.

    For the $m$ even case, we consider the base case when $m=4$.  We have $x_2 = F_4 = 3$, hence $x_2 + 1 = 4 = F_5 - 1 = x_3 - 1$.  Inductively for $m$ even, 
    $x_{m-2} + (x_{m-4} + x_{m-6} + \cdots + x_2 + 1) = x_{m-2} + (x_{m-3} -1) = F_m + (F_{m-1} -1) = F_{m+1} - 1 = x_{m-1} - 1$.
\end {proof}

\bigskip

\begin {lemma} \label{lem:affine_recurrence}
    The $z_n$'s satisfy the recurrence
    \[ z_n = (3+2\varepsilon) z_{n-1} - z_{n-2} - \varepsilon \]
\end {lemma}
Note that in the absence of the $\varepsilon$'s, this is the usual linear 
recurrence satisfied by every-other Fibonacci number.

\begin {proof}
    Written in terms of $p$'s, this becomes
    \[ p_{2n-5} = (3+2\varepsilon) p_{2n-7} - p_{2n-9} - \varepsilon \]
    First, since $2n-5$ is odd, we have
    \[ p_{2n-5} = (1+\varepsilon) p_{2n-6} + p_{2n-7} \]
    Since $2n-6$ is even, we have 
    \[ p_{2n-6} = (1+\varepsilon)p_{2n-7} + p_{2n-8} - \varepsilon \]
    Substitutuing this in the previous equation yields
    \begin {align*}  
        p_{2n-5} &= (1 + \varepsilon)^2 p_{2n-7} + (1+\varepsilon) p_{2n-8} - (1 + \varepsilon)\varepsilon + p_{2n-7} \\
               &= (2 + 2\varepsilon) p_{2n-7} + (1 + \varepsilon) p_{2n-8} - \varepsilon
    \end {align*}
    Finally, we note that since $2n-7$ is odd, we have the relation
    \[ (1 + \varepsilon) p_{2n-8} = p_{2n-7} - p_{2n-9} \]
    Substituting this in the equation above gives the desired result.
\end {proof}

\bigskip

\begin {lemma} \label{lem:quadratic_relation}
    The $z_n$'s satisfy the relation
    \[ z_n z_{n-2} = z_{n-1}^2 + z_{n-1} \varepsilon + 1 \]
\end {lemma}
\begin {proof}
    We follow the methodology used to prove Proposition 1 of \cite{MusPropp} in the ordinary (non-super) case.
    Using \Cref{lem:affine_recurrence}, we can rewrite the product $z_n z_{n-2}$ as 
    \begin{align*}
    z_{n} z_{n-2} &= \bigg( (3+2\varepsilon) z_{n-1} - z_{n-2} - \varepsilon\bigg) z_{n-2} \\
    &= (3+2\varepsilon) z_{n-1} z_{n-2} - \bigg( z_{n-2}^2 + \varepsilon z_{n-2}\bigg).
    \end{align*}
    By induction, we may rewrite \[ z_{n-1} z_{n-3} = z_{n-2}^2 + \varepsilon z_{n-2} + 1\] and hence 
    \begin{align*}
    z_{n} z_{n-2}  &= (3+2\varepsilon) z_{n-1} z_{n-2} - \bigg( z_{n-1}z_{n-3} - 1\bigg) \\
    &= z_{n-1} \bigg( (3+2\varepsilon) z_{n-2} - z_{n-3}\bigg) + 1.
    \end{align*}

    We use \Cref{lem:affine_recurrence} again, noting that $z_{n-1} = (3+2\varepsilon) z_{n-2} - z_{n-3} - \varepsilon$, and conclude that 
    \begin{align*} 
    z_{n} z_{n-2} &= z_{n-1} \bigg( z_{n-1} + \varepsilon\bigg) + 1 \\
    & = z_{n-1}^2 + \varepsilon z_{n-1} + 1
    \end{align*}
    as desired.
\end {proof}

\bigskip

\begin {proof}[Proof of \Cref{prop:annulus_lambda_lengths}]
    Comparing \Cref{eqn:super_sequence} and \Cref{lem:quadratic_relation}, we see that the $z_n$'s satisfy the same
    recurrence as the $\lambda$-lengths. We need only verify that they have the same initial conditions.  
    We observe that $z_3$ and $z_4$ are defined combinatorially as the double dimer partition functions for $G_1$ and $G_3$, respectively.  
    Comparing this with $Z_3$ and $Z_4$, as defined via recurrence (\ref{eqn:super_sequence}) and $Z_1 =Z_2 = 1$, 
    we indeed see that $Z_3 = (Z_2^2 + \varepsilon Z_2 + 1)/Z_1 = 2 + \varepsilon = z_3$ and $Z_4 = (Z_3^2 + \varepsilon Z_3 + 1)/Z_2 = 5 + 6 \varepsilon = z_4$.
\end {proof}

\begin{remark} \label{rem:z_n_with_vars}
In the more general case where $Z_1 = a$ and $Z_2 = b$ are not set to be $1$, we conjecture that the solution to recurrence (\ref{eqn:super_sequence}) 
is still given as the double dimer partition function of the snake graphs $G_{2m-5}$'s, while taking into account edge weights.

For example, letting $z_1=a$, $z_2=b$, we calculate double dimer partition functions for $G_1$, $G_2$, and $G_3$ respectively as 
$$p_1 = z_3 = \frac{b^2 + 1}{a} + \frac{b}{a}\varepsilon, $$
$$p_2 = \frac{b^2 + a^2 + 1}{ab} + \frac{b+a}{ab}\varepsilon\mathrm{~and~}$$ 
$$p_3 = z_4 = \frac{b^4 + 2b^2 + a^2 + 1}{a^2b} + \frac{ 2b^3 + 2b + a b^2 + a}{a^2b}\varepsilon.$$
In particular, $z_2 z_4 = z_3^2 + \varepsilon z_3 + 1$ in this case, hence $z_m = Z_m$ is still a solution to recurrence  (\ref{eqn:super_sequence}) 
for $1 \leq m \leq 4$ in this more general case.
\end{remark}

Above we defined $p_m = x_m + y_m \varepsilon$ based on counting the number of double dimer covers on the $2 \times (m+1)$ grid graph $G_m$ with $m$ tiles with $\mu$-invariants repeating $2$-periodically.  Recall that $x_m = F_{m+2}$ counts the double dimer covers with no cycles and hence is in bijection with dimer covers of $G_m$.   The $y_m$'s count the number of double dimer covers containing exactly one cycle, and the cycle must be of odd length.
We now give several different compact algebraic formulae for the $y_m$'s.  First we express them in terms of Fibonacci numbers.  Second, as a double-sum over binomial coefficients.  Finally, as a single sum utilizing binomial coefficients.

\begin{lemma} \label{lem:g_m}
    Define $g_m$ as the self-convolution of the Fibonacci sequence\footnote{$g_m$ is the OEIS sequence \href{https://oeis.org/A001629}{A001629}.}:
    \[ g_m = \sum_{k=1}^m F_k F_{m-k+1} \]
    Then $y_m$ can be expressed as the sum\footnote{$y_m$ is the OEIS sequence \href{https://oeis.org/A054454}{A054454} as well as the third
    column of the triangular array \href{https://oeis.org/A054453}{A054453}.}
    \[ y_m = \sum_{j=0}^{\lfloor m/2 \rfloor} g_{m-2j} \]
\end{lemma} 

\begin{proof} 
    We will first show that $g_m$ as the number of double dimer covers of the graph $G_m$ which contain exactly one cycle and that cycle is of exactly length one.  The placement in the graph $G_m$ of the unique cycle bisects the graph.  Consequently, a double dimer cover is completed by picking the equivalent of a dimer cover on the left side of $G_m$ and a dimer cover on the right side of $G_m$.  Summing over the possibilities gives us 
    $g_m = \sum_{k=1}^m F_k F_{m-k+1}$ as desired.

    Since $y_m$ counts all double dimers containing a cycle, rather than only those of length one, we next count the double dimer covers containing a 
    cycle of length three, five, etc.  In each of these cases, counting the number of double dimer covers with a cycle of length $(2j+1)$ on $G_m$ is equivalent to counting the number of double dimer covers with a cycle of length $1$ on $G_{m-2j}$.  Consequently, we get the formula 
    $y_m = \sum_{j=0}^{\lfloor \frac{m}{2} \rfloor} g_{m-2j}$.
\end{proof}

\begin{lemma} \label{lem:g_m_recurrence}
    The values of $g_m$ satisfy the recurrence 
    $$g_m = g_{m-2} + g_{m-1} + x_{m-2}$$
    for $m \geq 3$. 
\end{lemma}

\begin{proof}
    We use $F_1 = 1$, $x_{m-2} = F_m$, and \Cref{lem:g_m}  to expand the right-hand-side as
    \begin{align*}
        g_{m-2} + g_{m-1} + x_{m-2} &= \sum_{k=1}^{m-2} F_k F_{m-1-k} + \sum_{k=1}^{m-1} F_k F_{m-k} + F_m F_1 \\
                                    &= \sum_{k=1}^{m-2} F_k F_{m-1-k} \\ 
                                       &+ \left( \sum_{k=1}^{m-2} F_k F_{m-k} + F_{m-1}F_1\right) + F_m F_1.
    \end{align*}
    Note that the two sums involving $(m-2)$ terms each can now be combined using the Fibonacci recurrence as 
    $$\sum_{k=1}^{m-2} F_k (F_{m-1-k} + F_{m-k}) = \sum_{k=1}^{m-2} F_k F_{m+1-k}.$$ 
    Consequently, we see
    $$g_{m-2} + g_{m-1} + x_{m-2} = \left(\sum_{k=1}^{m-2} F_k F_{m+1-k}\right) + F_{m-1} F_1 + F_m F_1$$
    However, using the equality $F_1 = F_2$, we can rewrite this as 
    $$g_{m-2} + g_{m-1} + x_{m-2} = \sum_{k=1}^{m-2} F_k F_{m+1-k} + F_{m-1} F_2 + F_m F_1 = \sum_{k=1}^{m} F_k F_{m+1-k} = g_m$$
    where the last equality follows from \Cref{lem:g_m}.
\end{proof}

\begin{lemma} \label{lem:Fib}
    The Fibonacci numbers can also be expanded in terms of binomial coefficients as 
    $$x_m = F_{m+2} = \sum_{j=0}^{ \lfloor \frac{m+1}{2} \rfloor} {m+1 - j \choose j}$$
\end{lemma}

\begin{proof}
    This identity goes back to Lucas \cite{lucas1878theorie} and appears explicitly in Hoggatt-Lind \cite{hoggatt1969compositions}.  
    Here we give a combinatorial proof analogous to that of \Cref{lem:g_m}.  We recall that a double dimer cover with no cycles 
    is in bijection with a dimer cover, and utilize this throughout.
    
    Observe that the graph $G_m$ has $m$ possible tiles on which to place a pair of horizontal dimers but $(m+1)$ possible vertical edges 
    in which to place a vertical dimer.  Hence a dimer cover of $G_m$ with exactly $2j$ horizontal dimers consists 
    of $j$ tiles which contain a pair of horizontal dimers as well as $m+1-2j$ vertical dimers.  Writing out this dimer cover as a word of 
    $j + (m+1-2j) = m+1-j$ letters consisting of $j$ instances of $H$ and $(m+1-2j)$ instances of $V$, we see such dimer covers are 
    indeed counted by the binomial coefficient ${m+1-j \choose j}$.   
\end{proof}

\bigskip

\begin{prop} \label{prop:g_n_sum}
    The values of $g_n$ can be expressed compactly via the algebraic formula
    $$g_n = \sum_{j=1}^{\lfloor \frac{n+1}{2} \rfloor} j {n+1-j \choose j}.$$
\end{prop}
\begin{proof}
    This formula appears in the notes on the page for OEIS sequence \href{https://oeis.org/A001629}{A001629}.  See also \cite{czabarka2015discrete}.  However, here we give a self-contained proof. 

    Induction and Lemmas \ref{lem:g_m_recurrence} and \ref{lem:Fib} allow us to rewrite 
    $$g_m = g_{m-2} + g_{m-1} + x_{m-2} = 
    \sum_{j=1}^{\lfloor \frac{m-1}{2} \rfloor} j {m-1-j \choose j}
    + \sum_{j=1}^{\lfloor \frac{m}{2} \rfloor} j {m-j \choose j}
    + \sum_{j=0}^{\lfloor \frac{m-1}{2} \rfloor} {m-1-j \choose j}$$
    $$ =  {m-1 \choose 0} + \sum_{j=1}^{\lfloor \frac{m-1}{2} \rfloor} (j+1) {m-1-j \choose j} + \sum_{j=1}^{\lfloor \frac{m}{2} \rfloor} j {m-j \choose j}.$$

    At this point, we divide into cases, based on the parity of $m$. First consider the case when $m = 2k$ is even. Then $\lfloor \frac{m-1}{2} \rfloor = k-1$ 
    and $\lfloor \frac{m}{2} \rfloor = k$, and so we have
    \begin {align*}
        g_m &= 1 + \sum_{j=1}^{k-1} (j+1) \binom{m-1-j}{j} + \sum_{j=1}^k j \binom{m-j}{j} \\
            &= 1 + \sum_{j=2}^k j \binom{m-j}{j-1} + \sum_{j=1}^k j \binom{m-j}{j} \\
            &= 1 + \binom{m-1}{1} + \sum_{j=2}^k j \left[ \binom{m-j}{j-1} + \binom{m-j}{j} \right] \\
            &= m + \sum_{j=2}^k j \binom{m-j+1}{j} 
    \end {align*}
    Finally, note that $m = 1 \cdot \binom{m-1+1}{1}$, and also $k = \lfloor \frac{m+1}{2} \rfloor$, so this agrees with the desired formula in the case $m=2k$.

    Next, we consider the case where $m = 2k+1$ is odd. In this case, $\lfloor \frac{m-1}{2} \rfloor = \lfloor \frac{m}{2} \rfloor = k$, so we have
    \begin {align*}
        g_m &= 1 + \sum_{j=1}^k (j+1) \binom{m-1-j}{j} + \sum_{j=1}^k j \binom{m-j}{j} \\
            &= 1 + \sum_{j=2}^{k+1} j \binom{m-j}{j-1} + \sum_{j=1}^k j \binom{m-j}{j} \\
            &= 1 + \binom{m-1}{1} + (k+1) \binom{m-k-1}{k} + \sum_{j=2}^k j \left[ \binom{m-j}{j-1} + \binom{m-j}{j} \right] \\
            &= m + (k+1) + \sum_{j=2}^k j \binom{m-j+1}{j}
    \end {align*}
    As before, $m = 1 \cdot \binom{m-1+1}{1}$ gives the $j=1$ term, and since $m=2k+1$ we see additionally $k+1 = (k+1) \binom{m-(k+1)+1}{k+1}$ is the $j=k+1$ term.
    Since $ \lfloor \frac{m+1}{2} \rfloor = k+1$, this gives the result.
\end{proof}

\begin{corollary} 
    Hence the parameter $y_m$ can also be expressed compactly as
    $$y_m = \sum_{n=0}^{\lfloor \frac{m}{2} \rfloor} \sum_{j=1}^{\lfloor \frac{n+1}{2} \rfloor} j {n+1-j \choose j}$$
\end{corollary}

\bigskip

\begin{prop} \label{prop:xyg}
    We can express
    \[ 
        x_n + y_n = \begin {cases}
            g_{n+1}     & \text{ if $n$ is even} \\[1ex]
            g_{n+1} + 1 & \text{ if $n$ is odd}
        \end {cases}
    \]
\end{prop}
\begin {proof}
    We will prove this result by induction on $n$. For $n=1$, we have $x_n + y_n = 3$, and 
    $g_2 + 1 = F_1 F_2 + F_2 F_1 + 1 = 1 \cdot 1 + 1 \cdot 1 + 1 = 3$.
    For $n=2$, $x_n + y_2 = 5$, and $g_3 = F_1 F_3 + F_2^2 + F_3 F_1 = 1 \cdot 2 + 1^2 + 2 \cdot 1 = 5$. So the base cases are confirmed.

    For the inductive step we will make use of the fact that $D(G_n) = D_T(G_n) \cup G_R(G_n) \cup D_{tr}(G_n)$. 
    However, we require a slight modification. Let $\widehat{D}(G)$ denote the set of double dimer covers with non-zero weight
    (and similarly for $\widehat{D}_R(G)$, $\widehat{D}_T(G)$, and $\widehat{D}_{tr}(G)$). Note that with this notation, we have $x_n + y_n = \left| \widehat{D}(G_n) \right|$.

    \Cref{lem:lemma1} and \Cref{lem:lemma2} easily generalize to give bijections $\widehat{D}_R(G_n) \to \widehat{D}(G_{n-1})$ 
    and $\widehat{D}_T(G_n) \to \widehat{D}(G_{n-2})$. This means
    \begin {equation} \label{eq:DR_union_DT}
        \left| \widehat{D}_R(G_n) \cup \widehat{D}_T(G_n) \right| = x_{n-1} + y_{n-1} + x_{n-2} + y_{n-2} = 1 + g_{n-1} + g_n,
    \end {equation}
    where the last equality follows by induction.
    
    Finally, we consider $\widehat{D}_{tr}(G_n)$. Note that elements of $D_{tr}(G)$ have a cycle around the last tile. But since elements
    of $\widehat{D}(G_n)$ can only have a single cycle (of odd length), this means the cycle around the last tile is the \emph{only} cycle.
    If the cycle at the end of $G_n$ surrounds $2k+1$ tiles,
    then the remaining part of the double dimer cover is an element of $D(G_{n-2-2k})$ which contains no cycles. Recall that these are counted by
    $x_{n-2-2k}$, which by \Cref{lem:Fib} is equal to $F_{n-2k}$. So we have
    \[ \left| \widehat{D}_{tr}(G_n) \right| = \sum_{k=1}^{\lfloor \frac{n+1}{2} \rfloor} F_{n-2k} \]
    Using the identities $\sum_{i=1}^k F_{2i-1} = F_{2k}$ and $\sum_{i=1}^k F_{2i} = F_{2k+1} - 1$, just like in the proof of 
    \Cref{lem:p_m_simp}, we get two cases depending on the parity:
    \begin {equation} \label{eq:D_tr} 
        \left| \widehat{D}_{tr}(G_n) \right| = 
        \begin {cases}
            F_{n+1} - 1 = x_{n-1} -1 & \text{ if $n$ is even} \\[1ex]
            F_{n+1}     = x_{n-1}    & \text{ if $n$ is odd}
        \end {cases}
    \end {equation}
    Finally, recall that $x_n + y_n = \left| \widehat{D}(G_n) \right| = \left| \widehat{D}_R(G_n) \cup \widehat{D}_T(G_n) \right| + \left| \widehat{D}_{tr}(G_n) \right|$.      
    Combining \Cref{eq:DR_union_DT}, \Cref{eq:D_tr}, and \Cref{lem:g_m_recurrence}, gives the result.
\end {proof}

\bigskip

The following result was conjectured by the third author in an unpublished work (see \Cref{footnote:sylvester}). 

\bigskip

\begin{prop} \label{prop:y_m_combined}
    $$y_{2m}   =  \sum_{k=0}^m (2k)   \binom{m+k+1}{2k+1}$$
    $$y_{2m+1} =  \sum_{k=0}^m (2k+1) \binom{m+k+2}{2k+2}$$
    Note that these two expressions can be combined to give a formula for all $y_n$:
    \[ y_n = \sum_{k=0}^{\lfloor n/2 \rfloor} (n - 2k) \binom{n-k+1}{n-2k+1} \]
\end{prop}

\begin{proof}
    We will prove the formula for $y_{2m+1}$. The calculation for $y_{2m}$ is completely analogous.

    Using Propositions \ref{prop:g_n_sum} and \ref{prop:xyg}, respectively, we write 
    $$g_{2m+2} = \sum_{j=1}^{m+1} j {2m+3-j \choose j} = x_{2m+1} + y_{2m+1} - 1.$$
    Then, using Lemma \ref{lem:Fib}, we can rearrange the above equation to obtain
    $$y_{2m+1} = 1 + \sum_{j=1}^{m+1} j {2m+3-j \choose j}  - \sum_{j=0}^{m+1} {2m+2-j \choose j}.$$

    By peeling off the last term, and then shifting the index in the second sum, we get
    $$y_{2m+1} = 1 + \left(\sum_{j=1}^{m+1} \left( j {2m+3-j \choose j}  - {2m+3-j \choose j-1}\right)\right) - {m+1 \choose m+1} .$$

    By applying the classical binomial coefficient identity $\binom{n}{k} = \frac{n+1-k}{k} \binom{n}{k-1}$ to the expressions above, we can
    see that $j \binom{2m+3-j}{j} = (2m+4-2j) \binom{2m+3-j}{j-1}$. Subtracting $\binom{2m+3-j}{j-1}$ from both sides, we see that
    the summands are each equal to $(2m+3-2j) \binom{2m+3-j}{j-1}$. We can thus rewrite the equation above as
    $$y_{2m+1} = \sum_{j=1}^{m+1} (2m+3-2j) {2m+3-j \choose j-1}.$$
    Finally, to get the desired result, we let $k = m+1 - j$ and run through the sum backwards.
\end{proof}

\begin{remark} \label{rem:p_even}
    Recall that for $n \geq 3$, the $x_{n}$'s are the Fibonacci numbers when $x_1=x_2=1$.  
    In the cluster algebra associated to an annulus with a marked point on each boundary, where $\{x_1,x_2\}$ is the initial cluster 
    consisting of two formal variables, then the values of $z_n = x_{2n-5}$ are the corresponding cluster variables associated to other possible arcs around the annulus.

    One can also focus on the even-indexed entries, but instead of corresponding to cluster variables, the $x_{2n-4}$'s 
    correspond to the lambda lengths of peripheral arcs which start at the marked point on the inner boundary, wind around $(n-1)$ times, and 
    end at the marked point on the inner boundary.  Such peripheral arcs have self-crossings, in particular $(n-2)$ such self-intersections, when $n\geq 3$.  
    See \cite{gunawan2019cluster}.
\end{remark}

\begin{remark} \label{rem:w_n}
    As explained above after \Cref{rem:snake}, for $n\geq 3$, $x_{2n-4}$'s have a combinatorial interpretation as the (weighted) 
    number of dimer covers in the snake graph $G_{2n-4}$.  Additionally, as we just saw in Remark \ref{rem:p_even}, such numbers 
    (resp. expressions) have interpretations as lambda lengths of peripheral arcs in the annulus.
     
    Defining $w_n = p_{2n-4} = x_{2n-4} + y_{2n-4}$, we see that for $n\geq 3$, $w_n = p_{2n-4}$ counts the (weighted) number of 
    double dimer covers in $G_{2n-4}$.  For example, as computed in Remark \ref{rem:z_n_with_vars}, we see that 
    $w_3 = p_2 = \frac{b^2 + a^2 + 1}{ab} + \frac{b+a}{ab}\varepsilon$.  Comparison with the classical case motivates the following conjecture.
\end{remark}

\begin{conj} \label{conj:w_n}
    If we let $w_1=w_2=1$ (or if we let $w_1=a$ and $w_2=b$), the values of $w_n = p_{2n-4}$ correspond to the super $\lambda$-lengths of a peripheral arc in 
    an annulus, as described in \Cref{rem:p_even}, except in the context of the decorated super-Teichm\"uller space.
\end{conj}

The main obstruction to proving \ref{conj:w_n} and such a geometric significance is that unlike the ordinary case, we 
only have $\mu$-invariants attached to triangles in a triangulation.  To be able to compute the super $\lambda$-length 
of an arc with self-intersections would require a super analogue of the skein relations for resolving a crossing rather than 
simply a super analogue of the Ptolemy exchange relation for flipping a diagonal in an oriented quadrilateral.

Nonetheless, since we have the candidate of $w_n$'s for the peripheral arcs in an annulus, along with a conjectured 
combinatorial interpretation that would allow us to calculate the corresponding expressions, perhaps it is possible to 
reverse engineer how to properly define super skein relations in the context of decorated super-Teichm\"uller space.

Another natural follow-up question to the above work involves super-Markov numbers.

\begin{remark} \label{rem:Markoff}
    The super $\lambda$-lengths of a once-punctured torus are studied by Huang, Penner, and Zeitlin in \cite{mcshane} 
    where triples of arcs in a triangulation have super $\lambda$-lengths satisfying the super analogue 
    of the Markoff equation (c.f. \cite[equation 26]{mcshane})
    \[x^2+y^2+z^2+(xy+yz+xz)\varepsilon=3(1+\varepsilon)xyz\]  
    where $\varepsilon=\sigma\theta$.
    This motivates the following conjecture.
\end{remark}

\begin{conj} \label{conj:Markoff}
    Begin with an oriented triangulation of the once-punctured torus like in \Cref{fig:other_flips}.  
    However, in this case, keep $e$ as a formal variable instead of setting it to be $1$.  Allowing flips in all three directions 
    iteratively yields all possible arcs in a once punctured torus.  We claim that the super $\lambda$-lengths of such arcs are super analogues 
    of the Markoff numbers, and have combinatorial interpretations using double dimer covers of the snake graphs appearing in Section 7 of \cite{propp05} 
    in the presence of appropriately specialized $\mu$-invariants.
\end{conj}

We end with the following questions.

\begin{question} 
    Assuming the validity of \Cref{conj:Markoff}, do the coefficients of the $\varepsilon$'s in super Markoff numbers have compact algebraic formulae analogous to the 
    formula for the $y_m$'s as coefficients of $\varepsilon$'s in super Fibonacci numbers (see \Cref{prop:y_m_combined})?
\end{question}

\begin{question} \label{ques:MSW}
    Going beyond the annlus with one marked point on each boundary, or the once punctured torus, can we find formulae for super $\lambda$-lengths 
    of arcs on unpunctured surfaces (or certain punctured surfaces) that have combinatorial interpretations as double dimer covers of snake graphs, 
    using the same snake graphs interpreting ordinary $\lambda$-lengths of arcs as in \cite{msw_11}?
\end{question}

Before approaching \Cref{ques:MSW}, we note that it is not obvious for every triangulation and every arc that there is a corresponding spin structure (i.e. oriented triangulation relative to the boundary and equivalence on triangles) and flip sequence that preserves positive expansion formulas as one iterates the computation.  Nonetheless, the above combinatorial interpretation would at least give a candidate for what algebraic expressions may be reasonable to try.

\section*{Acknowledgements}

The authors would like to thank the support of the NSF grants DMS-1745638 and DMS-1854162, as well as the University of Minnesota UROP program.  We would also like to thank Bruce Sagan for many helpful conversations and Ralf Schiffler for his questions motivating some of this work.

\bibliographystyle {alpha}
\bibliography {dimers}
\end{document}